\documentclass[pagesize,pdftex]{scrartcl}
\usepackage{latexsym} 
\usepackage[intlimits]{amsmath}
\usepackage{amsthm}
\usepackage{amsfonts}
\usepackage{amssymb}
\usepackage{amsxtra}
\usepackage{amscd}
\usepackage{ifthen}
\usepackage{graphicx}
\usepackage[shortlabels]{enumitem}
\usepackage{mathrsfs}
\usepackage{mathtools}
\usepackage{MnSymbol}
\usepackage{soul}
\usepackage[pagebackref=true]{hyperref}

\hypersetup{
pdfauthor={Thomas Eiter, Mads Kyed, Yoshihiro Shibata},
pdftitle={Falling drop in an unbounded liquid reservoir: Steady-state solutions},
breaklinks=true,
colorlinks=true,
linkcolor=blue,
citecolor=blue,
urlcolor=blue,
filecolor=blue,
}

\numberwithin{equation}{section} 
\setkomafont{title}{\normalfont}
\allowdisplaybreaks[4]

\newcommand{\ie}{\textit{i.e.}}

\newcommand{\cred}[1]{#1}
\newcommand{\cblue}[1]{#1}

\newcommand{\barycenter}{\xi}
\newcommand{\tddropdomain}{{\Omega^{(1)}_t}}
\newcommand{\tddropdomainclosed}{{\overline{\Omega}^{(1)}_t}}
\newcommand{\tdreservoirdomain}{{\Omega^{(2)}_t}}
\newcommand{\tdliquiddomain}{\Omega_t}
\newcommand{\tdinterface}{{\Gamma_t}}
\newcommand{\tdrho}{\rho}
\newcommand{\tdmu}{\mu}
\newcommand{\tdvel}{\vvel}
\newcommand{\tdpres}{\vpres}
\newcommand{\tdveldrop}{\vvel^{(1)}}
\newcommand{\tdpresdrop}{\vpres^{(1)}}
\newcommand{\tdvelreservoir}{\vvel^{(2)}}
\newcommand{\tdpresreservoir}{\vpres^{(2)}}
\newcommand{\tdnvec}{\mathrm{n}}

\newcommand{\tdinterfacepara}{\Phi}
\newcommand{\ssinterfacepara}{\Phi_\ssinterface}
\newcommand{\tdlagrangecoordinate}{\Phi_\Gamma}
\newcommand{\tdgravitation}{b}

\newcommand{\ssdropdomain}{\cred{\Omega^\inner}}
\newcommand{\ssreservoirdomain}{\cred{\Omega^\outer}}
\newcommand{\ssliquiddomain}{{\cred{\Omega}}}
\newcommand{\ssinterface}{{\cred{\Gamma}}}
\newcommand{\ssdropdomainh}{{\cblue{\Omega^\inner_\height}}}
\newcommand{\ssdropdomainhclosed}{{\cblue{\overline{\Omega}^\inner_\height}}}
\newcommand{\ssreservoirdomainh}{{\cblue{\Omega^\outer_\height}}}
\newcommand{\ssreservoirdomainhclosed}{{\cblue{\overline{\Omega}^\outer_\height}}}
\newcommand{\ssliquiddomainh}{{\cblue{\Omega_\height}}}
\newcommand{\ssinterfaceh}{{\cblue{\Gamma_\height}}}
\newcommand{\sspres}{\upres}
\newcommand{\sspresreservoir}{\upres^\outer}
\newcommand{\sspresdrop}{\upres^\inner}
\newcommand{\refdomain}{{\Omega_0}}
\newenvironment{pdeq}{ \left\{ \begin{aligned}}{\end{aligned}\right.}

\newcommand{\eqrefsub}[2]{\eqref{#1}\textsubscript{#2}}
\newcommand{\np}[1]{(#1)}
\newcommand{\nb}[1]{[#1]}
\newcommand{\bp}[1]{\big(#1\big)}
\newcommand{\bb}[1]{\big[#1\big]}
\newcommand{\Bp}[1]{\bigg(#1\bigg)}

\newcommand{\Bb}[1]{\bigg[#1\bigg]}
\newcommand{\calg}{{\mathcal G}}

\newcommand{\call}{{\mathcal L}}
\newcommand{\calm}{{\mathcal M}}
\newcommand{\caln}{{\mathcal N}}
\newcommand{\calo}{{\mathcal O}}
\newcommand{\calp}{{\mathcal P}}

\newcommand{\R}{\mathbb{R}}

\newcommand{\N}{\mathbb{N}}
\newcommand{\e}{e}
\newcommand{\B}{B}

\DeclareMathOperator{\id}{Id}
\DeclareMathOperator{\Div}{div}
\newcommand{\DivAdjoint}{\Div^*}

\DeclareMathOperator{\cof}{cof}
\DeclareMathOperator{\supp}{supp}

\DeclareMathOperator{\trace}{Tr}

\newcommand{\embeds}{\hookrightarrow}
\DeclareMathOperator{\vecspan}{span}
\DeclareMathOperator{\sgn}{sgn}

\newcommand{\diffq}[2]{D^{#1}_{#2}}
\newcommand{\extopr}{E}
\newcommand{\soloprvel}{\mathcal{K}}
\newcommand{\linoprspace}{\mathscr{L}}
\newcommand{\ra}{\rightarrow}

\newcommand{\set}[1]{\ensuremath{\{#1\}}}

\newcommand{\setc}[2]{\ensuremath{\{#1\ \mid\ #2\}}}
\newcommand{\setcl}[2]{\ensuremath{\bigl\{#1\ \big\mid\ #2\bigr\}}}
\newcommand{\setcL}[2]{\ensuremath{\biggl\{#1\ \bigg\mid\ #2\biggr\}}}

\newcommand{\ball}{\mathrm{B}}
\newcommand{\sphere}{\mathbb{S}}
\renewcommand{\restriction}[2]{#1\big | _{#2}}
\newcommand{\seqkN}[1]{\ensuremath{\set{#1_k}_{k=1}^\infty}}
\newcommand{\proj}{\calp}
\newcommand{\projcompl}{\calp_\bot}

\newcommand{\sorthomatrixspace}[1]{SO(#1)}

\newcommand{\transpose}{\top}

\newcommand{\idmatrix}{\mathrm{I}}
\newcommand{\Rn}{{\R^n}}

\newcommand{\Rdot}{\dot{\R}}
\newcommand{\cube}{Q}
\newcommand{\cuberx}{\cube_r(\tilde{x})}
\newcommand{\ballrixi}{\ball_{r_i}(x_i)}
\newcommand{\ballri}{\ball_{r_i}(0)}

\newcommand{\cubehalfr}{\cube_{\frac{r}{2}}(0)}
\newcommand{\cuber}{\cube_r(0)}
\newcommand{\cubeplus}{\cube^+}
\newcommand{\cubeminus}{\cube^-}
\newcommand{\cubepr}{\cube_r'(0)}
\newcommand{\grad}{\nabla}

\newcommand{\pt}{\partial_t}

\newcommand{\dx}{{\mathrm d}x}

\newcommand{\dr}{{\mathrm d}r}

\newcommand{\dS}{{\mathrm d}S}

\newcommand{\symmgrad}{\mathrm S}
\newcommand{\LaplaceBeltrami}{\Delta_{\sphere^2}}
\newcommand{\nvec}{\mathrm{n}}
\newcommand{\meancurv}{\mathrm{H}}
\newcommand{\projtang}{(\idmatrix-\nvec\otimes\nvec)}
\newcommand{\tdprojtang}{(\idmatrix-\tdnvec\otimes\tdnvec)}
\newcommand{\projtangtrafo}{\mathcal{P}^\height}
\newcommand{\projtangtrafozero}{\mathcal{P}^0}
\newcommand{\SR}{\mathscr{S}}

\newcommand{\TDR}{\mathscr{S^\prime}}

\newcommand{\ft}[1]{\widehat{#1}}

\newcommand{\FT}{\mathscr{F}}
\newcommand{\iFT}{\mathscr{F}^{-1}}

\newcommand{\norm}[1]{\lVert#1\rVert}
\newcommand{\oseennorm}[2]{\norm{#1}_{#2,\mathrm{Oseen}}}

\newcommand{\normL}[1]{\Bigl\lVert#1\Bigr\rVert}
\newcommand{\snorm}[1]{{\lvert #1 \rvert}}
\newcommand{\snorml}[1]{{\bigl\lvert #1 \big\rvert}}
\newcommand{\snormL}[1]{{\Bigl\lvert #1 \Big\rvert}}

\newcommand{\testspcube}{\mathrm{W}^{1,2}_{0,\Gamma_0}}
\newcommand{\testsp}{C}
\newcommand{\testsolsp}{\mathscr{C}}
\newcommand{\hsp}{H}
\newcommand{\hspm}{H_M}
\newcommand{\hspdual}{H'}
\newcommand{\hsolsp}{\mathscr{H}}
\newcommand{\hsolspdual}{\mathscr{H}'}
\newcommand{\WSR}[2]{\mathrm{W}^{#1,#2}} 
\newcommand{\WSRN}[2]{\mathrm{W}^{#1,#2}_0} 
\newcommand{\DSR}[2]{\mathrm{D}^{#1,#2}} 
\newcommand{\DSRN}[2]{\mathrm{D}^{#1,#2}_0} 
\newcommand{\WSRloc}[2]{\mathrm{W}^{#1,#2}_{\mathrm{loc}}}
\newcommand{\CR}[1]{\mathrm{C}^{#1}}  
\newcommand{\LR}[1]{\mathrm{L}^{#1}}

\newcommand{\LRN}[1]{\mathrm{L}^{#1}_0} 
\newcommand{\LRNm}[1]{\mathrm{L}^{#1}_{0,M}} 
\newcommand{\LRloc}[1]{\mathrm{L}^{#1}_{\mathrm{loc}}} 
\newcommand{\CRi}{\CR \infty}
\newcommand{\CRci}{\CR \infty_0}

\newcommand{\WSRhom}[2]{\mathrm{D}^{#1,#2}} 
\newcommand{\WSRhomN}[2]{\mathrm{D}^{#1,#2}_0} 
\newcommand{\dpair}[2]{\langle #1, #2 \rangle}

\newcommand{\Xspace}{\mathbf X^{q,r}}
\newcommand{\Yspace}{\mathbf Y^{q,r}}

\newcommand{\Xrey}{{\mathbf X^{q,r}_\Rey}}
\newcommand{\Xreyrhod}{{\mathbf X^{q,r}_\Reyrhod}}
\newcommand{\XreyOne}{{\mathbf X^{q,r}_{1,\Rey}}}
\newcommand{\Xoseen}{\mathrm{X}^{q,r,\Rey}_{\mathrm{Oseen}}}
\newcommand{\nsnonlinb}[2]{#1\cdot\grad #2}
\newcommand{\nsnonlin}[1]{\nsnonlinb{#1}{#1}}
\newcommand{\vvel}{v}
\newcommand{\vpres}{p}

\newcommand{\wvel}{w}
\newcommand{\wpres}{\mathfrak{q}}

\newcommand{\uvel}{u}
\newcommand{\upres}{\mathfrak{p}}
\newcommand{\uvelhf}{u_\#}
\newcommand{\upreshf}{\mathfrak{p}_\#}
\newcommand{\liftgvel}{G_\#}
\newcommand{\liftfvel}{V_\#}
\newcommand{\liftfpres}{Q_\#}
\newcommand{\lifthzerovel}{W_\#}
\newcommand{\lifthzeropres}{\Pi_\#}
\newcommand{\lifthtwotwovel}{\widetilde{W}_\#}
\newcommand{\lifthonetwopres}{\widetilde{\Pi}_\#}
\newcommand{\phihf}{{\phi_{\#}}}
\newcommand{\Uvelhf}{U_\#}
\newcommand{\Upreshf}{\mathfrak{P}_\#}
\newcommand{\uvellf}{u_{\bot}}
\newcommand{\uvelonelf}{u_{\bot1}}
\newcommand{\uveltwolf}{u_{\bot2}}
\newcommand{\uvelthreelf}{u_{\bot3}}
\newcommand{\upreslf}{\mathfrak{p}_\bot}

\newcommand{\Uvel}{U}

\newcommand{\Upres}{\mathfrak{P}}

\newcommand{\tuvel}{\widetilde{u}}
\newcommand{\tupres}{\widetilde{\mathfrak{p}}}

\newcommand{\szvelhf}{\mathfrak{z}_{\#}}
\newcommand{\szpreshf}{\mathfrak{q}_{\#}}

\newcommand{\fluidstress}{\mathrm T}
\newcommand{\fluidstresstrafo}{\mathrm T^\height}
\newcommand{\fluidstresstrafoOne}{\mathrm T^{\height_1}}
\newcommand{\fluidstresstrafoTwo}{\mathrm T^{\height_2}}
\newcommand{\inner}{{(1)}}
\renewcommand{\outer}{{(2)}}
\newcommand{\jump}[1]{\big\lsem #1 \big\rsem}
\newcommand{\height}{\eta}
\newcommand{\theight}{\widetilde{\eta}}
\newcommand{\gravity}{\mathrm{g}}
\newcommand{\fundsoloseen}{\varGamma_{\textnormal{\tiny{Oseen}}}^\rey}

\newcommand{\tin}{\text{in }}
\newcommand{\tif}{\text{if }}
\newcommand{\ton}{\text{on }}

\newcommand{\tand}{\text{and }}

\newcommand{\tas}{\text{as }}
\newcommand{\half}{\frac{1}{2}}

\renewcommand{\epsilon}{\varepsilon}
\renewcommand{\phi}{\varphi}
\newcommand{\rey}{\lambda}
\newcommand{\Reyrhod}{{\lambda_0\np{\rhod}}}
\newcommand{\Rey}{{\lambda_0}}
\newcommand{\reyd}{\kappa}
\newcommand{\treyd}{\widetilde{\kappa}}
\newcommand{\rhod}{\widetilde{\rho}}
\newcommand{\snrhod}{\snorm{\rhod}}
\newcommand{\rhodpot}{{\widetilde{\rho}^\ppot}}
\newcommand{\reybar}{\overline{\rey}}

\newcommand{\ethree}{\e_3}

\newcommand{\rotmatrix}{R}
\newcommand{\rotmatrixi}{\mathrm{R}_i}
\newcommand{\bigo}{O}
\newcommand{\smallo}{o}

\newcommand{\kroneckerdelta}{\delta}

\newcommand{\cutoff}{\chi}
\newcommand{\cutoffkappa}{\kappa}

\newcommand{\change}[1]{}

\newcommand{\ctrafo}{\Phi^\height}

\newcommand{\tctrafo}{\Phi^{\widetilde{\height}}}
\newcommand{\gradctrafo}{F_\height}
\newcommand{\gradctrafoone}{F_{\height_1}}
\newcommand{\gradctrafotwo}{F_{\height_2}}
\newcommand{\detctrafo}{J_\height}
\newcommand{\detctrafoone}{J_{\height_1}}
\newcommand{\detctrafotwo}{J_{\height_2}}
\newcommand{\cofctrafo}{A_\height}
\newcommand{\cofctrafoone}{A_{\height_1}}
\newcommand{\cofctrafotwo}{A_{\height_2}}

\newcommand{\xip}{\xi'}
\newcommand{\Rtwo}{{\R^2}}
\newcommand{\Rthree}{{\R^3}}
\newcommand{\fhf}{f_\#}

\newcommand{\ghf}{g_\#}

\newcommand{\hzerohf}{h_{0\#}}

\newcommand{\htwohf}{{h_{2\#}}}
\newcommand{\honehf}{{h_{1\#}}}
\newcommand{\Htwo}{{H_2}}
\newcommand{\Hone}{{H_1}}
\newcommand{\htwo}{{h_2}}
\newcommand{\hone}{{h_1}}

\newcommand{\lopr}{\call^\Rey}
\newcommand{\loprrhod}{\call^\Reyrhod}
\newcommand{\loprrhodinv}{\bp{\call^\Reyrhod}^{-1}}
\newcommand{\loprcomp}[1]{\call_{#1}}
\newcommand{\loprinv}{\np{\lopr}^{-1}}
\newcommand{\nopr}{\caln^{R,\rhod}}
\newcommand{\noprs}{\caln^{R(\rhod),\rhod}}
\newcommand{\noprcomp}{\caln}

\newcommand{\ppot}{\alpha}
\newcommand{\mmap}{\calm}
\newcommand{\newCCtr}[2][d]{
\newcounter{#2}\setcounter{#2}{0}
\expandafter\xdef\csname kyedtheconst#2\endcsname{#1}
}
\newcommand{\Cc}[2][nolabel]{
\stepcounter{#2}
\expandafter\ensuremath{\csname kyedtheconst#2\endcsname_{\arabic{#2}}}
\ifthenelse{\equal{#1}{nolabel}}
{}
{\expandafter\xdef\csname kyedconst#1\endcsname
{\expandafter\ensuremath{\csname kyedtheconst#2\endcsname_{\arabic{#2}}}}}
}
\newcommand{\Ccn}[2][nolabel]{
\expandafter\ensuremath{\csname kyedtheconst#2\endcsname}
\ifthenelse{\equal{#1}{nolabel}}
{}
{\expandafter\xdef\csname kyedconst#1\endcsname
{\expandafter\ensuremath{\csname kyedtheconst#2\endcsname}}}
}
\newcommand{\CcSetCtr}[2]{
\setcounter{#1}{#2}
}
\newcommand{\Cclast}[1]{
\expandafter\ensuremath{\csname kyedtheconst#1\endcsname_{\arabic{#1}}}
}
\newcommand{\Ccllast}[1]{
\addtocounter{#1}{-1}
\expandafter\ensuremath{\csname kyedtheconst#1\endcsname_{\arabic{#1}}}
\addtocounter{#1}{1}
}
\newcommand{\const}[1]{
\expandafter{\ifcsname kyedconst#1\endcsname
  \csname kyedconst#1\endcsname
\else
  \errmessage{Undefined Kyedconstant #1.}%
\fi}
}

\theoremstyle{plain}
\newtheorem{thm}{Theorem}[section]
\newtheorem{defn}[thm]{Definition}
\newtheorem{lem}[thm]{Lemma}
\newtheorem{prop}[thm]{Proposition}

\theoremstyle{remark}

\begin{document}
\title{Falling drop in an unbounded liquid reservoir: Steady-state solutions}

\author{
Thomas Eiter\\ 
Fachbereich Mathematik\\
Technische Universit\"at Darmstadt\\
Schlossgartenstr. 7, 64289 Darmstadt, Germany\\
Email: {\texttt{eiter@mathematik.tu-darmstadt.de}}
\and
Mads Kyed\\ 
Flensburg University of Applied Sciences\\
Kanzleistrasse 91-93, 24943 Flensburg\\
Email: {\texttt{mads.kyed@hs-flensburg.de}}
\and
Yoshihiro Shibata\\ 
Department of Mathematics\\
Waseda University\\
3-4-1 Okubo Shinjuku-ku, Tokyo 169-8555, Japan\\
Email: {\texttt{yshibata@waseda.jp}}
}

\date{\today}
\maketitle

\begin{abstract}
The equations governing the motion of a three-dimensional liquid drop moving freely in an
unbounded liquid reservoir under the influence of a gravitational force
are investigated. Provided the (constant) densities in the two liquids are sufficiently close, existence of a steady-state solution is shown. The 
proof is based on a suitable linearization of the equations. A setting
of function spaces is introduced in which the corresponding linear
operator acts as a homeomorphism.
\end{abstract}


\newCCtr[C]{C}
\newCCtr[c]{c}
\newCCtr[M]{M}
\newCCtr[\delta]{delta}
\CcSetCtr{delta}{-1}
\newCCtr[\epsilon]{eps}
\CcSetCtr{eps}{-1}
\let\oldproof\proof
\def\proof{\CcSetCtr{c}{-1}\oldproof} 

\section{Introduction}

Consider a drop of liquid with density $\rho_1$ submerged into an unbounded reservoir 
of liquid with density $\rho_2$. Assume the liquids are immiscible. We investigate the motion of the drop under the influence of a constant gravitational force and surface tension on the interface. Specifically, we shall show existence of a steady-state solution to the governing equations of motion, provided the difference $\snorm{\rho_1-\rho_2}$ of the densities is sufficiently small.
 
The dynamics of a falling (or rising) drop in a quiescent fluid has attracted a lot of attention  
in the field of fluid mechanics. 
Such flows have been studied extensively both experimentally and numerically with truly fascinating outcomes
(see \cite{Bothe_RisingDroplet} for a comprehensive overview and further references), but it remains an intriguing task to analytically validate the observations. 
The observed dynamics can be characterized as a series of bifurcations with respect to the Reynolds number as parameter. Broadly speaking, steady-state solutions are observed for small Reynolds numbers, with
bifurcations into oscillating motions as the Reynolds number increases. Bifurcations into more complex solutions can be observed as the Reynolds number increases even further. 

In the following, we shall investigate the steady-state solutions corresponding to small Reynolds numbers. 
A small Reynolds number is equivalent to a small 
density difference $\snorm{\rho_1-\rho_2}$. One of our aims is to develop a functional analytic framework that can be used
not only to study steady states but also 
as foundation for further investigations into the dynamics described above. In particular, the framework should facilitate 
a stability analysis of the steady states, and an investigation of the 
Hopf-type bifurcations (into oscillating motions) observed in experiments. 
For this purpose, it should satisfy 
certain properties. First and foremost, it should be possible to identify function spaces within the framework 
such that the differential operator of the linearized equations of motion acts as a homeomorphism. 
Second, the framework should have a natural extension to a suitable time-periodic framework (recall that a steady-state solution is trivially also time-periodic). Third, 
the framework should adequately facilitate a spectral analysis of the operators obtained by linearizing 
the equations of motion around a steady state. 
To meet these criteria, we propose a framework of Sobolev spaces. Although a setting of Sobolev spaces seems natural, and by far the most convenient to work with, it is by no
means trivial to identify one that conforms to the problem of a freely falling (or rising) drop. Indeed, one of the novelties of this article is the introduction of such a Sobolev-space setting that meets at least the first and most important criteria, and possibly also the other two, mentioned above, and in which existence of steady-state solutions can be shown effortlessly for small data.
The investigation of steady-state solutions is not new, though. It was initiated by \textsc{Bemelmans} \cite{Bemelmans_LiquidDrops1981}
and advanced by \textsc{Solonnikov} \cite{Solonnikov_LiquidDropInfiniteLiquidMedium,Solonnikov_LiquidDropCylinder}. 
However, the analysis carried out by \textsc{Bemelmans} and \textsc{Solonnikov} do not lead to 
a framework of Sobolev spaces. Indeed, for reasons that will be explained in detail below, the approaches of both \textsc{Bemelmans} and \textsc{Solonnikov} \emph{cannot} be adapted
to a Sobolev-space setting with the desired properties.

We shall consider the most commonly used model for two-phase flows with surface tension on the interface.
It is assumed both fluids are Navier--Stokes liquids, that is, incompressible, viscous, and Newtonian. 
It is further assumed that the fluids are immiscible with surface tension on their interface in normal direction proportional to the mean curvature.  
Moreover, we consider a system in which the drop is a ball $\ball_{R_0}$ of radius $R_0$ when no external forces act on the system, that is, in its stress free configuration.
If we choose a coordinate system attached to the falling drop,
these assumptions lead to the following equations of motion for a steady state (see Section \ref{EqOfMotionSection} for details on the derivation):
\begin{align}\label{intro_mainsystem}
\begin{pdeq}
-\Div \fluidstress\np{\uvel,\upres}+ \rho \, \np{\uvel \cdot \grad \uvel + \rey\partial_3\uvel} & = - \rho\,\gravity\,  e_3 && \text{in } \R^3\setminus\ssinterfaceh, \\
\Div \uvel &= 0 && \text{in } \R^3\setminus\ssinterfaceh,\\
\jump{\fluidstress\np{\uvel,\upres}\nvec}&=\sigma\, \meancurv\np{\height} \nvec  && \text{on } \ssinterfaceh,\\
\jump{\uvel}&=0  && \text{on } \ssinterfaceh, \\
\uvel \cdot \nvec  &= - \rey e_3 \cdot \nvec && \text{on } \ssinterfaceh, \\
\snorml{\Omega^\inner_\eta}=\frac{4\pi}{3}R_0^3,\quad \lim_{\snorm{x}\ra\infty}\uvel\np{x}&=0.
\end{pdeq}
\end{align}
Here, $\ssinterfaceh$ denotes the interface between the two liquids, which we may assume to be a closed manifold parameterized by a  
``height'' function $\height\colon\partial\ball_{R_0}\ra\R$ describing the displacement of the drop's boundary points in normal direction. 
The domain $\Omega^\inner_\eta\subset\R^3$ bounded by $\ssinterfaceh$ 
describes the domain occupied by the drop, and the exterior domain $\Omega^\outer_\eta\coloneqq \R^3\setminus\ssdropdomainhclosed$  
the region of the liquid reservoir.
The drop velocity $-\rey e_3$, $\rey\in\R$, is assumed to be directed along the axis of the (constant) gravitational force $\gravity e_3$.
The first two equations in \eqref{intro_mainsystem} are the Navier--Stokes equations written in a moving frame of reference, 
where $\uvel\colon\R^3\setminus\ssinterfaceh\ra\R^3$ denotes the Eulerian velocity field of the liquids, $\upres\colon\R^3\setminus\ssinterfaceh\ra\R$ the
scalar pressure field, and $\fluidstress\np{\uvel,\upres}$ denotes the corresponding Cauchy stress tensor. The density function $\rho\colon\R^3\setminus\ssinterfaceh\ra\R$ is constant in both components of $\R^3\setminus\ssinterfaceh$.  
The third equation states that the surface tension in normal direction on the interface $\ssinterfaceh$ is proportional to the mean curvature $\meancurv$,
with $\sigma>0$ a constant. The notation $\jump{\cdot}$ is used to denote 
the jump of a quantity across $\ssinterfaceh$. 
Immiscibility of the two liquids under a no-slip assumption at the interface is expressed via the
fourth and fifth equation. Observe that the normal velocity on the interface then coincides with that of the moving frame, 
which moves with the same velocity $- \rey e_3$ as the falling drop.
The equations are augmented with a volume condition for the drop and the requirement that the liquid in the reservoir is at rest at spatial infinity
in the sixth and seventh equation, respectively.

A key part of our investigation is directed towards finding an appropriate linearization of \eqref{intro_mainsystem} with respect to the
unknowns $\uvel$, $\upres$, $\rey$ and $\eta$. The canonical linearization, \textit{i.e.}, around the trivial state $(0,0,0,0)$, leads to
the Navier--Stokes equations \eqrefsub{intro_mainsystem}{1-2} being replaced with the Stokes system 
\begin{align}\label{intro_stokeslinearization}
\begin{pdeq}
-\Div \fluidstress\np{\uvel,\upres} & = f && \text{in } \R^3\setminus\partial\ball_{R_0}, \\
\Div \uvel &= 0 && \text{in } \R^3\setminus\partial\ball_{R_0}.
\end{pdeq}
\end{align}
An analysis based on this linearization would have to be carried out in a setting of function spaces conforming 
to the properties of the Stokes problem. Such a setting, however, is not suitable for an investigation of the exterior domain Navier--Stokes equations in a 
moving frame. 
Since the falling drop, and thus the frame of reference, moves with a nonzero velocity $-\rey e_3$, the appropriate linearization 
of the Navier--Stokes equations in the exterior domain is an Oseen system. At least in a setting of classical Sobolev spaces, 
the steady-state exterior-domain Navier--Stokes equations in a moving frame can only be solved in a framework of Sobolev spaces conforming to the Oseen linearization.
To resolve this issue, we propose to rewrite the system \eqref{intro_mainsystem} as a perturbation around a state $(\uvel_0,\upres_0,\rey_0,\eta_0)$
with $\rey_0\neq 0$. A subsequent linearization of \eqref{intro_mainsystem} then yields the Oseen problem
\begin{align}\label{intro_oseenlinearization}
\begin{pdeq}
-\Div \fluidstress\np{\uvel,\upres} + \rey_0\partial_3\uvel & = f && \text{in } \R^3\setminus\partial\ball_R, \\
\Div \uvel &= 0 && \text{in } \R^3\setminus\partial\ball_R. 
\end{pdeq}
\end{align}
The main challenge, and indeed novelty of this article, is to determine a suitable state 
$(\uvel_0,\upres_0,\rey_0,\eta_0)$ that renders the problem well posed in a framework of classical Sobolev spaces.

The starting point of our investigation were the articles \cite{Solonnikov_LiquidDropInfiniteLiquidMedium,Solonnikov_LiquidDropCylinder} by \textsc{Solonnikov},
which contain a number of truly outstanding ideas on how to analyze \eqref{intro_mainsystem}.
However, \textsc{Solonnikov} overlooks the necessity of an  Oseen linearization as described above.
Instead, he employs a Stokes linearization and consequently a setting of function spaces in which the nonlinear term $\rey\partial_3\uvel$ cannot be correctly treated on the right-hand side.
Our approach resolves this issue.

We derive the steady-state equations of motion for the falling drop and state the main theorem in the following Section \ref{EqOfMotionSection}.
The aforementioned framework of Sobolev spaces is then introduced in Section \ref{Preliminaries_Section}.
Fundamental $\LR{r}$ estimates  are established in Section \ref{AuxiliaryLinearProblemSection}, and a reformulation of \eqref{intro_mainsystem} in
a fixed reference configuration in Section \ref{ReformulationFixedDomainSection}. The linearization around a non-trivial state is
carried out in Section \ref{LinearizationSection}.
In Section \ref{ExistenceSteadySolutionSection} we show in Theorem \ref{ThmLinearOperatorHom} that the operator corresponding to this linearization is
a homeomorphism in our framework of Sobolev spaces, which finally enables us to establish a proof of the main theorem, namely the existence
of a steady-state solution for $\snorm{\rho_1-\rho_2}$ sufficiently small.

\section{Equations of motion and statement of the main theorem}\label{EqOfMotionSection}

We derive the system of equations 
governing the motion of a freely falling drop in a liquid under the influence of a constant gravitational force.
We shall express these equations in a frame of reference with origin in the barycenter 
of the drop. More specifically, we denote by $\barycenter(t)$ the barycenter of the falling drop with respect to an inertial frame, whose coordinates we
denote by $y$, and express the equations of motion in barycentric coordinates $x(t,y)\coloneqq y-\barycenter(t)$.
In these coordinates, the domain
$\tddropdomain\subset\R^3$ occupied by the drop at time $t$ satisfies
\begin{equation}\label{derivation_barycenter}
\int_\tddropdomain x \,\dx = 0. 
\end{equation}
We let $\tdreservoirdomain \coloneqq \R^3\setminus\tddropdomainclosed$
denote the domain of the surrounding liquid reservoir, and put
$\tdliquiddomain \coloneqq \tddropdomain\cup\tdreservoirdomain$.
The surface $\tdinterface \coloneqq\partial\tddropdomain$ describes the interface between the two liquids.
Moreover, we let
$\mu_1,\mu_2$ and $\rho_1,\rho_2$ denote the constant viscosities and densities
of the drop and the liquid reservoir, respectively. The functions 
\begin{align*}
&\tdmu:\bigcup_{t\in\R_+}\set{t}\times\tdliquiddomain\ra\R,\quad
\tdmu(t,x)\coloneqq
\begin{cases}
\mu_1, & x \in \tddropdomain, \\
\mu_2, & x \in \tdreservoirdomain,
\end{cases}\\
&\tdrho:\bigcup_{t\in\R_+}\set{t}\times\tdliquiddomain\ra\R,\quad
\tdrho(t,x)\coloneqq
\begin{cases}
\rho_1, & x \in \tddropdomain, \\
\rho_2, & x \in \tdreservoirdomain,
\end{cases}
\end{align*}
then describe the viscosity and density of the liquid occupying the point $x$ at a given time $t$.
Expressed in a frame of reference attached to the barycenter $\barycenter$, the conservation of
momentum and mass of both liquids
is described by the Navier--Stokes system
\begin{align}\label{derivation_navierstokes}
\begin{pdeq}
\tdrho \np{\pt\tdvel+\nsnonlin{\tdvel}-\nsnonlinb{\dot{\barycenter}}{\tdvel}} &= \Div\fluidstress(\tdvel,\tdpres) + \tdrho\tdgravitation \\
\Div\tdvel &= 0 
\end{pdeq}\qquad \tin \bigcup_{t\in\R_+}\set{t}\times\tdliquiddomain,
\end{align}
where $\tdvel$ denotes the Eulerian velocity field in the liquids, $\tdpres$ the pressure,
\begin{equation*}
\fluidstress(\tdvel,\tdpres)\coloneqq 2\mu \, \symmgrad(\tdvel) - \tdpres \idmatrix, \quad \symmgrad(\tdvel)\coloneqq \half \bp{\grad \tdvel + \grad \tdvel^\transpose}
\end{equation*}
the Cauchy stress tensor, and $\tdgravitation\in\R^3$ a constant
gravitational acceleration.
One can decompose the velocity field and pressure term into  
\begin{equation*}
\tdveldrop:\bigcup_{t\in\R_+}\set{t}\times\tddropdomain\ra\R^3,\quad
\tdpresdrop:\bigcup_{t\in\R_+}\set{t}\times\tddropdomain\ra\R
\end{equation*}
describing the liquid flow in the drop, and another part
\begin{equation*}
\tdvelreservoir:\bigcup_{t\in\R_+}\set{t}\times\tdreservoirdomain\ra\R^3,\quad
\tdpresreservoir:\bigcup_{t\in\R_+}\set{t}\times\tdreservoirdomain\ra\R
\end{equation*}
describing the flow in the reservoir. We employ the notation
\begin{align*}
\jump{\tdvel}\coloneqq\restriction{\tdvel^\inner}{\Gamma}-\restriction{\tdvel^\outer}{\Gamma}
\end{align*}
to denote the jump in a quantity on the interface between the two liquids.
Concerning the physical nature of the interface, we make the basic assumption that slippage between the two liquids cannot occur,
\textit{i.e.}, a no-slip boundary condition, and that liquid cannot be absorbed in the interface. Consequently, there is no jump in the
velocity field neither in tangential nor in normal direction:
\begin{equation}\label{derivation_noslip}
\jump{\tdvel}=0\quad\ton\bigcup_{t\in\R_+}\set{t}\times\tdinterface.
\end{equation}
Since the liquids are immiscible, the normal component of the liquid velocity at the interface coincides with the velocity of the interface itself. If $\tdlagrangecoordinate$
denotes a Lagrangian description of the interface in
barycentric coordinates, the immiscibility condition therefore takes the form 
\begin{equation}\label{derivation_immiscibility}
\tdvel\cdot \tdnvec = \pt\tdlagrangecoordinate\cdot \tdnvec + \dot{\barycenter}\cdot \tdnvec\qquad  \ton\bigcup_{t\in\R_+}\set{t}\times\tdinterface. 
\end{equation}
In the classical two-phase flow model, surface tension on the interface, \ie,  
the difference in normal stresses of the two liquids, is
proportional to the mean curvature in normal direction and in balance in tangential direction: 
\begin{align}
\tdnvec\cdot \jump{\fluidstress\np{\tdvel,\tdpres}\tdnvec}&=\sigma\meancurv \quad
\ton\bigcup_{t\in\R_+}\set{t}\times\tdinterface,\label{derivation_normalstress_normal}\\
\tdprojtang \jump{\fluidstress\np{\tdvel,\tdpres}\tdnvec}&=0 \qquad  
\ton\bigcup_{t\in\R_+}\set{t}\times\tdinterface.\label{derivation_normalstress_tangential}
\end{align}
Since we consider the motion of a drop in a \emph{quiescent} liquid, the velocity
in the reservoir vanishes at spatial infinity
\begin{equation}\label{derivation_velatinfinity}
\lim_{\snorm{x}\ra\infty}\tdvel(t,x) = 0.
\end{equation}
Due to the incompressibility of the liquid drop, its volume is constant. 
Since we consider a drop that takes the shape of the ball $\ball_{R_0}$ in its stress free configuration, 
this volume is prescribed by
\begin{equation}\label{derivation_volumecond}
\snorm{\tddropdomain}=\frac{4\pi}{3}R_0^3.
\end{equation}
In conclusion, the system obtained by combining \eqref{derivation_barycenter}--\eqref{derivation_volumecond} governs the
motion of a liquid drop falling freely in a liquid reservoir under the influence of a constant gravitational force. 

In this article, we seek to establish existence of a steady-state solution, that is, a time-independent solution to
\eqref{derivation_barycenter}--\eqref{derivation_volumecond}.
Such a solution is of course only \emph{steady} with respect to the
chosen frame of reference; in our case the frame attached to the barycenter.
Other types of steady states can be investigated by analyzing time-independent solutions in
other frames. For example, it is conceivable that falling drops can perform 
steady \emph{rotating} motions, which should be investigated by considering the equations of motion
in a rotating frame of reference. 

The unknowns in  
\eqref{derivation_barycenter}--\eqref{derivation_volumecond}
are the functions $\tdvel,\tdpres,\dot{\barycenter},\tdinterfacepara_\Gamma$.
The mean curvature $\meancurv$ can be computed from $\tdinterfacepara_\Gamma$.
The viscosities $\mu_1,\mu_2>0$, surface tension $\sigma>0$ and the prescribed volume $\frac{4\pi}{3}R_0^3$ of the drop are constants, which may be chosen arbitrarily.
Also the gravitational force $\tdgravitation\in\R^3$ is an arbitrary constant, but upon a re-orientation of the coordinates we may assume without loss
of generality that it is directed along the negative $e_3$ axis, \ie, $\tdgravitation=-\gravity e_3$ with $\gravity>0$.
The constant densities $\rho_1,\rho_2>0$ shall be restricted to pairs whose difference $\rho_1-\rho_2$ is sufficiently small.
In this sense, we treat $\rho_1-\rho_2$ as the data of the system.
Since the geometry $\bp{\tddropdomain,\tdreservoirdomain,\tdinterface}$ of the problem is determined by the unknown description
$\tdinterfacepara_\Gamma$ of the interface, \eqref{derivation_barycenter}--\eqref{derivation_volumecond} is a free boundary problem.

As mentioned above, we shall establish existence of a steady-state, that is, time-independent, solution $\np{\tdvel,\tdpres,\dot{\barycenter},\tdinterfacepara_\Gamma}$ to
\eqref{derivation_barycenter}--\eqref{derivation_volumecond}. In this case, the velocity $\dot{\barycenter}$ is a constant vector.
We focus on solutions with $\dot{\barycenter}$ directed along the axis of gravity, \ie, $\dot{\barycenter}=-\rey e_3$.
The steady-state equations of motion then read
\begin{align}\label{derivation_sseq_dimform}
\begin{pdeq}
\tdrho \np{\nsnonlin{\tdvel}+\rey\partial_3\tdvel} &= \Div\fluidstress(\tdvel,\tdpres)-\gravity e_3  && \tin \ssliquiddomain, \\
\Div\tdvel &= 0 && \tin \ssliquiddomain,\\ 
\jump{\tdvel}&=0 && \ton \ssinterface,\\
\tdvel\cdot \tdnvec &= -\rey e_3\cdot \tdnvec && \ton \ssinterface,\\
\tdnvec\cdot\jump{\fluidstress\np{\tdvel,\tdpres}\tdnvec}&=\sigma\meancurv && \ton \ssinterface,\\
\tdprojtang \jump{\fluidstress\np{\tdvel,\tdpres}\tdnvec}&=0 && \ton\ssinterface,\\
\lim_{\snorm{x}\ra\infty}\tdvel(x) = 0,\quad 
\snorm{\ssdropdomain}&=\frac{4\pi}{3}R_0^3,\quad
\int_{\ssdropdomain} x \,\dx = 0,
\end{pdeq}
\end{align}
where the interface $\ssinterface$ is an unknown computed from the parameterization $\ssinterfacepara$. The unknowns in \eqref{derivation_sseq_dimform}
are $\tdvel,\tdpres,\rey,\ssinterfacepara$.

At the outset, it is clear that \eqref{derivation_sseq_dimform} can have multiple solutions. 
This is best illustrated by considering  $\rho_1=\rho_2$, in which case the trivial solution with  $\tdvel=0$, $\rey=0$ and constant pressures $\tdpresdrop$, $\tdpresreservoir$ is a steady-state solution if 
$\sigma\meancurv$ equals the constant hydrostatic pressure difference $\tdpresdrop-\tdpresreservoir$ between the drop and the reservoir.
Since a constant mean curvature $\meancurv$ is realized whenever $\ssdropdomain$ is a
multiple of disjoint balls, we obtain
for each $\ssinterfacepara$ describing one or more spheres a trivial solution by 
adjusting the
hydrostatic pressure difference accordingly (depending on the fixed volume $\snorm{\ssdropdomain}$).
In the case \eqref{derivation_sseq_dimform} above, the fixed volume of $\snorm{\ssdropdomain}$ coincides with the volume of the ball $\ball_{R_0}$.
With constant pressures satisfying $\tdpresdrop-\tdpresreservoir=\frac{2}{R_0}$, 
the ball $\ball_{R_0}$ therefore becomes an admissible steady-state drop configuration when $\rho_1=\rho_2$.
We shall single out this configuration for further investigation in the sense that we investigate non-trivial steady-states with a configuration close to the ball $\ball_{R_0}$ for $\rho_1\neq\rho_2$ with $\rho_1-\rho_2$ sufficiently small. 

From a physical perspective, a smallness condition is only meaningful when expressed in a non-dimensional form.
In order to obtain a dimensionless formulation of \eqref{derivation_sseq_dimform}, we
choose $R_0$ as characteristic length scale, $V_0\coloneqq\sqrt{\gravity R_0}$ as the characteristic velocity, $\rho_1+\rho_2$ as characteristic
density, $(\rho_1+\rho_2)R_0 V_0$ as the characteristic viscosity, and $(\rho_1+\rho_2)R_0 V_0^2$ as the characteristic surface tension.
Investigating the resulting non-dimensional equations of motion, we will establish existence of a non-trivial steady-state solution with drop configuration close to the unit ball $\B_{1}$. For this purpose, it is convenient to introduce (in the non-dimensionalized coordinates) the normalized  pressures
\begin{align*}
&\sspresdrop(x):\ssdropdomain\ra\R,\quad
\sspresdrop(x)\coloneqq  \tdpresdrop(x)+\rho_1\, e_3\cdot x - 2\sigma\\  
&\sspresreservoir(x):\ssreservoirdomain\ra\R,\quad
\sspresreservoir(x)\coloneqq  \tdpresreservoir(x)+\rho_2\,e_3\cdot x.  
\end{align*}
We then obtain the following system of non-dimensional equations:
\begin{align}\label{derivation_sseq}
\begin{pdeq}
\tdrho \np{\nsnonlin{\tdvel}+\rey\partial_3\tdvel} &= \Div\fluidstress(\tdvel,\sspres)  && \tin \ssliquiddomain, \\
\Div\tdvel &= 0 && \tin \ssliquiddomain,\\ 
\jump{\tdvel}&=0 && \ton \ssinterface,\\
\tdvel\cdot \tdnvec &= -\rey e_3\cdot\nvec && \ton \ssinterface,\\
\tdnvec\cdot\jump{\fluidstress\np{\tdvel,\sspres}\tdnvec}&=\sigma(\meancurv+2) + (\rho_1-\rho_2)e_3\cdot x && \ton \ssinterface,\\
\tdprojtang \jump{\fluidstress\np{\tdvel,\sspres}\tdnvec}&=0 && \ton\ssinterface,\\
\lim_{\snorm{x}\ra\infty}\tdvel(x) = 0,\quad 
\snorm{\ssdropdomain}&=\frac{4\pi}{3},\quad
\int_{\ssdropdomain} x \,\dx = 0.
\end{pdeq}
\end{align}
Observe that the mean curvature now appears in the form $\np{\meancurv+2}$ that vanishes if $\Gamma$ is the unit sphere, which means that
$(\tdvel,\sspres,\rey)=(0,0,0)$ is a trivial solution when $\rho_1-\rho_2=0$.

We shall employ a parameterization of $\ssinterface$ 
over the unit sphere $\sphere^2\subset\R^3$ and subsequently linearize \eqref{derivation_sseq}. The linearization of the operator $\sigma(\meancurv+2)$, however, has a non-trivial kernel. To circumvent an introduction of the corresponding compatibility conditions, we
employ an idea from \cite{Solonnikov_LiquidDropCylinder} and replace the two equations
\begin{equation}\label{derivation_SolTrick_eqtoreplace}
\tdnvec\cdot \jump{\fluidstress\np{\tdvel,\sspres}\tdnvec} = \sigma(\meancurv+2) + (\rho_1-\rho_2)e_3\cdot x,\qquad
\int_{\ssdropdomain} x \,\dx = 0
\end{equation}
in \eqref{derivation_sseq} with the equations
\begin{align}\label{derivation_SolTrick_eqreplacement}
\begin{aligned}
&\tdnvec\cdot \jump{\fluidstress\np{\tdvel,\sspres}\tdnvec} = \sigma(\meancurv+2) +
\frac{1}{4\pi}\nvec \cdot \int_{\ssdropdomain} x \, \dx + (\rho_1-\rho_2)e_3\cdot x,\\
&\int_\ssinterface \jump{\fluidstress(\tdvel,\sspres)\nvec} \,\dS = (\rho_1-\rho_2) \frac{4\pi}{3} e_3.
\end{aligned}
\end{align}
The resulting system then reads
\begin{align}\label{derivation_sseq_AfterSolonnikovTrick}
\begin{pdeq}
\tdrho \np{\nsnonlin{\tdvel}+\rey\partial_3\tdvel} &= \Div\fluidstress(\tdvel,\sspres)  && \tin \ssliquiddomain, \\
\Div\tdvel &= 0 && \tin \ssliquiddomain,\\ 
\jump{\tdvel}&=0 && \ton \ssinterface,\\
\tdvel\cdot \tdnvec &= -\rey e_3 \cdot\nvec&& \ton \ssinterface,\\
\tdprojtang \jump{\fluidstress\np{\tdvel,\sspres}\tdnvec}&=0 && \ton\ssinterface,\\
\tdnvec\cdot \jump{\fluidstress\np{\tdvel,\sspres}\tdnvec} &= \sigma(\meancurv+2) +
\frac{1}{4\pi}\nvec \cdot \int_{\ssdropdomain} x \, \dx + (\rho_1-\rho_2)e_3\cdot x&& \ton\ssinterface,\\
\int_\ssinterface \jump{\fluidstress(\tdvel,\sspres)\nvec} \,\dS &=(\rho_1-\rho_2) \frac{4\pi}{3} e_3,\\
\lim_{\snorm{x}\ra\infty}\tdvel(x) = 0,\quad 
\snorm{\ssdropdomain}&=\frac{4\pi}{3}.
\end{pdeq}
\end{align}
The systems \eqref{derivation_sseq} and \eqref{derivation_sseq_AfterSolonnikovTrick} are equivalent. 
Clearly, \eqref{derivation_sseq} implies \eqref{derivation_sseq_AfterSolonnikovTrick}. To verify the reverse implication, observe that
\eqrefsub{derivation_sseq_AfterSolonnikovTrick}{5-7} imply
\begin{align*}
(\rho_1-\rho_2) \frac{4\pi}{3} e_3
&= \int_\ssinterface \jump{\fluidstress(\tdvel,\sspres)\nvec} \,\dS \\
&= \int_\ssinterface \bp{\tdnvec\cdot \jump{\fluidstress(\tdvel,\sspres)\nvec}}\tdnvec \,\dS \\
&= \int_{\ssinterface} \sigma(\meancurv+2)\tdnvec\,\dS + \frac{1}{4\pi} \int_\ssinterface \tdnvec\otimes\tdnvec \,\dS \int_{\ssdropdomain} x\,\dx +(\rho_1-\rho_2) \frac{4\pi}{3} e_3 \\
&= \int_{\ssinterface} \sigma\Delta_{\ssinterface}x\,\dS + 2\sigma \int_{\ssinterface}\tdnvec\,\dS + \frac{1}{4\pi} \int_\ssinterface \tdnvec\otimes\tdnvec \,\dS \int_{\ssdropdomain} x\,\dx +(\rho_1-\rho_2) \frac{4\pi}{3} e_3\\
&= 0 + 0 + \frac{1}{4\pi} \int_{\ssinterface} \tdnvec\otimes\tdnvec \,\dS \int_{\ssdropdomain} x\,\dx +(\rho_1-\rho_2) \frac{4\pi}{3} e_3.
\end{align*}
The matrix $\int_{\ssinterface} \tdnvec\otimes\tdnvec \,\dS$ is symmetric positive definite and thus invertible. Consequently, the equation above implies
$\int_{\ssdropdomain} x\,\dx=0$. We conclude that \eqref{derivation_sseq_AfterSolonnikovTrick} implies \eqref{derivation_sseq}. 

Since we investigate existence of non-trivial steady-states in a drop configuration close to the ball $\ball_{1}$ (in non-dimensionalized coordinates)
under the restriction that the difference in
densities of the two liquids is sufficiently small,
it is convenient to introduce 
\begin{equation*}
\rhod\coloneqq\rho_1-\rho_2
\end{equation*}
as smallness parameter.
Moreover, it is convenient to parameterize the interface $\ssinterface$ via a height function $\height : \sphere^2 \to \R$ that
describes the drop's displacement in normal direction with respect to its
unit sphere $\sphere^2\subset\R^3$ stress-free configuration.
The geometry then becomes a function of $\height$:
\begin{align*}
&\ssdropdomain=\ssdropdomainh \coloneqq \setcl{r\zeta}{\zeta\in\sphere^2,\ 0\leq r < 1+\height(\zeta)},
\quad \ssreservoirdomain= \ssreservoirdomainh  \coloneqq\setcl{r\zeta}{\zeta\in\sphere^2,\ 1+\height(\zeta)<r},\\ 
&\ssinterface=\ssinterfaceh\coloneqq \setcl{(1+\height(\zeta))\zeta}{\zeta\in\sphere^2},
\quad \ssliquiddomain=\ssliquiddomainh\coloneqq  \ssdropdomainh\cup\ssreservoirdomainh.
\end{align*}
The system of steady-state equations of motion finally takes the form
\begin{align}\label{derivation_sseq_final}
\begin{pdeq}
\tdrho \np{\nsnonlin{\tdvel}+\rey\partial_3\tdvel} &= \Div\fluidstress(\tdvel,\sspres)  && \tin \ssliquiddomainh, \\
\Div\tdvel &= 0 && \tin \ssliquiddomainh,\\ 
\jump{\tdvel}&=0 && \ton \ssinterfaceh,\\
\tdvel\cdot \tdnvec &= -\rey e_3\cdot\nvec && \ton \ssinterfaceh,\\
\tdprojtang \jump{\fluidstress\np{\tdvel,\sspres}\tdnvec}&=0 && \ton\ssinterfaceh,\\
\tdnvec\cdot \jump{\fluidstress\np{\tdvel,\sspres}\tdnvec} &= \sigma(\meancurv+2) +
\frac{1}{16\pi}\nvec \cdot \int_{\sphere^2}\zeta\bp{\np{1+\height(\zeta)}^4-1}\,\dS
+ \rhod e_3\cdot x&& \ton\ssinterfaceh,\\
\int_\ssinterfaceh \jump{\fluidstress(\tdvel,\sspres)\nvec} \,\dS &= \rhod \frac{4\pi}{3} e_3,\quad
\int_{\sphere^2}\bp{\np{1+\height(\zeta)}^3-1} \,\dS = 0,\quad
\lim_{\snorm{x}\ra\infty}\tdvel(x) = 0
\end{pdeq}
\end{align}
with respect to unknowns  $\np{\tdvel,\sspres,\rey,\height}$. 

As the main result in the article we prove existence of a solution to the steady-state equations of motion
\eqref{derivation_sseq_final} under a smallness condition on the density difference $\rhod$.

\begin{thm}[Main Theorem]\label{MainThm}
There is an $\epsilon>0$ such that for $0<\snorm{\rhod}\leq \epsilon$ there is a
solution 
\begin{align*}
\np{\tdvel,\sspres,\rey,\height}\in\CRi\np{\ssliquiddomainh}^3\times\CRi\np{\ssliquiddomainh}\times\R\times\CRi\np{\sphere^2}
\end{align*}
to \eqref{derivation_sseq_final}. The solution is smooth up to the interface, that is,
\begin{align}\label{MainThm_IntegrabilityProperties_SmoothUpToBoundary}
\restriction{\tdvel}{\ssreservoirdomainh},\,\restriction{\sspres}{\ssreservoirdomainh}\in\CRi\bp{\ssreservoirdomainhclosed},\qquad
\restriction{\tdvel}{\ssdropdomainh},\,\restriction{\sspres}{\ssdropdomainh}\in\CRi\bp{\ssdropdomainhclosed}.
\end{align}
Moreover, it 
possesses the integrability properties
\begin{align}\label{MainThm_IntegrabilityProperties}
\forall q\in\np{1,2}:
\quad \tdvel\in\LR{\frac{2q}{2-q}}\np{\ssliquiddomainh},
\quad \grad\tdvel\in\LR{\frac{4q}{4-q}}\np{\ssliquiddomainh},
\quad \partial_3\tdvel,\grad^2\tdvel,\grad\sspres\in\LR{q}\np{\ssliquiddomainh},
\end{align}
and admits the representation
\begin{align}\label{MainThm_AsymptoticProfile}
\tdvel(x) = \frac{4\pi}{3}\rhod\,\fundsoloseen(x)\,e_3
+ \bigo\bp{\snorm{x}^{-\frac{3}{2}+\epsilon}} \quad\tas \snorm{x}\ra\infty
\end{align}
for all $\epsilon>0$, where $\fundsoloseen$ denotes the Oseen fundamental solution%
\footnote{An explicit formula for $\fundsoloseen$ can be found in \cite[Section VII.3]{GaldiBookNew} for example.}%
. The solution is symmetric with respect to rotations leaving $e_3$ invariant:
\begin{align}\label{MainThm_Symmetry}
\forall \rotmatrix\in\sorthomatrixspace{3},\ \rotmatrix e_3=e_3:\quad R^\transpose \tdvel(Rx)=\tdvel(x),\ \tdpres(Rx)=\tdpres(x),\ \height(Rx)=\height(x),
\end{align}
and the velocity $\rey$ of the drop's barycenter is non-vanishing.
\end{thm}

By far the most challenging part of proving Theorem \ref{MainThm} is to establish the existence of a solution. As mentioned in the introduction, via a
perturbation around a non-trivial state we are able to solve the system in a setting of Sobolev spaces adapted from the 3D exterior-domain Oseen linearization of the
Navier--Stokes equations. Consequently, we are led to a solution with the integrability properties \eqref{MainThm_IntegrabilityProperties}. The symmetry \eqref{MainThm_Symmetry} follows from
the observation that \eqref{derivation_sseq_final} is invariant with respect to rotations leaving $e_3$ invariant.
Higher-order regularity is obtained via a standard approach utilizing the ellipticity of \eqref{derivation_sseq_final}, while the asymptotic profile \eqref{MainThm_AsymptoticProfile} is a direct consequence of
\eqref{MainThm_IntegrabilityProperties} and and a celebrated result of \textsc{Babenko} \cite{babenko1973}  and \textsc{Galdi} \cite{Galdi1992b}. Observe that the coefficient vector 
in the asymptotic expansion, which at the outset is given by 
\begin{align*}
\int_{\ssinterfaceh} \restriction{\fluidstress(\tdvel,\sspres)}{\ssreservoirdomainh}\nvec\,\dS ,
\end{align*}
coincides with the net force $\frac{4\pi}{3} \rhod e_3=\rhod\snorm{\ssdropdomainh}e_3$ acting on 
the liquid drop, that is, the difference of the gravitational force and the buoyancy force.

\section{Notation}\label{Notation_Section}

We use capital letters to denote global constants in the proofs and theorems, 
and small letters for local constants appearing in the proofs.

By $\ball_R\coloneqq\ball_R(0)$ we denote a ball in $\Rn$ centered at $0$ with radius $R$. 
Moreover, we let 
\begin{align*}
\ball^R\coloneqq\R^3\setminus\overline{\ball_R},\quad
\ball_{R,r}\coloneqq \ball_R\setminus\overline\ball_r,\quad
\Omega_R\coloneqq\Omega\cap\ball_R,\quad
\Omega^R\coloneqq\Omega\cap\ball^R
\end{align*}
for a domain $\Omega\subset\Rn$.
Additionally, we use $\sphere^2\coloneqq\partial\ball_1$ to denote the unit sphere. 
By
\begin{align*}
\Rdot^3\coloneqq\setcl{\np{x_1,x_2,x_3}\in\R^3}{x_3\neq0}
\end{align*}
we denote the twofold half space,
which is the union of the two domains
\begin{align*}
\Rdot^3_+\coloneqq\setcl{\np{x_1,x_2,x_3}\in\R^3}{x_3>0}, \qquad
\Rdot^3_-\coloneqq\setcl{\np{x_1,x_2,x_3}\in\R^3}{x_3<0}.
\end{align*}
We use the notation $\np{x',x_3}$ for a vector $x=\np{x_1,x_2,x_3}\in\R^3$.

Lebesgue spaces are denoted by $\LR{q}(\Omega)$
with associated norms $\norm{\cdot}_{q,\Omega}$.
By $\WSR{k}{q}(\Omega)$ we denote the corresponding Sobolev space of order $k\in\N_0$
with norm $\norm{\cdot}_{k,q,\Omega}$,
and we introduce the subspaces
\begin{align*}
\WSRN{k}{q}(\Omega)\coloneqq\overline{\CRci(\Omega)}^{\norm{\cdot}_{k,q,\Omega}}.
\end{align*}
Moreover, $\WSR{-k}{q}(\Omega)$ and $\WSRN{-k}{q}(\Omega)$ denote the 
dual spaces of $\WSR{k}{q'}(\Omega)$ and $\WSRN{k}{q'}(\Omega)$, respectively,
where $q'\coloneqq\frac{q}{q-1}$.
We further introduce homogeneous Sobolev spaces $\WSRhom{k}{q}(\Omega)$ defined by
\begin{align*}
\WSRhom{k}{q}(\Omega) \coloneqq
\setc{
\uvel\in\LRloc{1}(\Omega)}{
\grad^k\uvel\in\LR{q}(\Omega)
},
\end{align*}
and the corresponding seminorm
\begin{align*}
\snorm{\uvel}_{k,q,\Omega}\coloneqq
\norm{\grad^k\uvel}_{q,\Omega}\coloneqq
\sum_{\snorm{\alpha}=k}\norm{\partial^\alpha\uvel}_{q,\Omega}.
\end{align*}
In general, $\WSRhom{k}{p}(\Omega)$ is not a Banach space. 
However, $\snorm{\,\cdot\,}_{k,q,\Omega}$ defines a norm on $\CRci(\Omega)$,
and the completion
\begin{align*}
\WSRhomN{k}{q}(\Omega)\coloneqq\overline{\CRci(\Omega)}^{\snorm{\,\cdot\,}_{k,q,\Omega}}
\end{align*}
is therefore a Banach space.
By Sobolev's Embedding Theorem, $\WSRhomN{k}{q}(\Omega)$ can be identified with a subspace of $\LRloc{1}(\Omega)$ if $kq<3$.
We denote its dual space by $\WSRhomN{-k}{q'}(\Omega)$.
For a sufficiently smooth manifold $\Gamma\subset\R^3$ and  $s>0$, $s\not\in\N$, we let
$\WSR{s}{q}(\Gamma)$ denote the Sobolev--Slobodeckij space of order $s$ with norm $\norm{\cdot}_{s,q,\Gamma}$.

\section{Preliminaries}\label{Preliminaries_Section}

In this section we introduce a bespoke framework of Sobolev spaces for the investigation of  
\eqref{derivation_sseq_final}. For this purpose, we let $\Omega\subset\R^3$ denote a domain of the same type as in
Section \ref{EqOfMotionSection}, that is, we assume 
\begin{align}\label{Preliminaries_DomainTypeDef}
\begin{aligned}
&\Omega^\inner\subset\R^3\text{ is a bounded domain and } \Omega^\outer\coloneqq \R^3\setminus\overline{\Omega^\inner} \text{ is a domain,}\\
&\Gamma\coloneqq\partial\Omega^\inner,\qquad
\Omega\coloneqq\Omega^\inner\cup\Omega^\outer=\R^3\setminus\Gamma.
\end{aligned}
\end{align}
For a function $\uvel:\Omega\to\R$ we use the abbreviations
\begin{align*}
\uvel^\inner\coloneqq\restriction{\uvel}{\Omega^\inner}, \qquad 
\uvel^\outer\coloneqq\restriction{\uvel}{\Omega^\outer}.
\end{align*}
The function $\nvec=\nvec_\Gamma$ denotes the unit outer normal at $\Gamma$.
If $\uvel$ is sufficiently regular, we set
\begin{align*}
\jump{\uvel}\coloneqq\restriction{\uvel^\inner}{\Gamma}-\restriction{\uvel^\outer}{\Gamma},
\end{align*}
where the restriction has to be understood in the trace sense.
Furthermore, $\delta(\Omega)$ denotes the diameter of $\Omega^\inner$.

When considering a function $\uvel:\Omega\to\R$,
we often have to distinguish between its properties
on the disjoint sub-domains $\Omega^\inner$ and $\Omega^\outer$.
To this end, function spaces of the type
\begin{align*}
X\coloneqq\setcl{\uvel:\Omega\to\R}{\uvel^\inner\in X^\inner,\, \uvel^\outer\in X^\outer}
\end{align*}
are introduced.
Equipped with the norm
\begin{align*}
\norm{\uvel}_X\coloneqq\norm{\uvel^\inner}_{X^\inner}+\norm{\uvel^\outer}_{X^\outer},
\end{align*} 
such a space $X$  is isomorphic to the direct sum of the spaces $X^\inner$ and $X^\outer$.
Clearly, $X$ is a Banach space if $X^\inner$ and $X^\outer$ are so.

Let $q\in\bp{1,\frac{3}{2}}$ and $r\in(3,\infty)$. 
For $\Rey\in\R, \Rey\neq 0$ the space
\begin{align*}
\Xoseen\coloneqq\Xoseen(\Omega^\outer)
\coloneqq\setcl{\uvel\in\LRloc{1}(\Omega^\outer)^3}{
\uvel\in\LR{\frac{2q}{2-q}}\cap\WSRhom{1}{\frac{4q}{4-q}}\cap\WSRhom{2}{q}\cap\WSRhom{2}{r},\,
\partial_3 \uvel\in\LR{q}\cap\LR{r}
}
\end{align*}
equipped with the norm
\begin{align*}
\oseennorm{\uvel}{\Rey}\coloneqq
\snorm{\Rey}^{\frac{1}{2}}\norm{\uvel}_{\frac{2q}{2-q}}
+\snorm{\Rey}^{\frac{1}{4}}\norm{\grad \uvel}_{\frac{4q}{4-q}}
+\norm{\grad^2 \uvel}_{q}
+\norm{\grad^2 \uvel}_{r}
+\snorm{\Rey}\,\norm{\partial_3 \uvel}_{q}
+\snorm{\Rey}\,\norm{\partial_3 \uvel}_{r}
\end{align*}
is the canonical solution space for solutions to the exterior domain Oseen problem
\begin{align}\label{Preliminaries_BespokeOseen}
\begin{pdeq}
-\Div \fluidstress(\uvel,\upres) + \Rey \partial_3 \uvel&=f && \text{in } \Omega^\outer , \\
\Div \uvel &= g && \text{in } \Omega^\outer
\end{pdeq}
\end{align}
for forcing terms $f$ in $\LR{q}\bp{\Omega^\outer}\cap\LR{r}\bp{\Omega^\outer}$; see for example \cite[Chapter VII.7]{GaldiBookNew}. 
Let
\begin{align*}
\XreyOne
&\coloneqq\setcl{\uvel\in\LRloc{1}(\R^3)^3}{
\uvel^\inner\in\WSR{2}{r},\,
\uvel^\outer\in\Xoseen,\,
\jump{u}=0
},\\
\Xspace_2
&\coloneqq\setcl{\upres\in\LRloc{1}(\R^3)}{
\upres^\inner\in\WSR{1}{r},\,
\upres^\outer\in\WSRhom{1}{q}\cap\WSRhom{1}{r}\cap\LR{\frac{3q}{3-q}}
},\\
\Xspace_3&\coloneqq\R,\\
\Xspace_4&\coloneqq\WSR{3-1/r}{r}(\Gamma),
\end{align*}
and
\begin{align*}
\Yspace_1&
\coloneqq \LR{q}(\R^3)^3\cap\LR{r}(\R^3)^3,\\
\Yspace_2
&\coloneqq \setcl{g\in\LRloc{1}(\R^3)}{
g^\inner\in\WSR{1}{r},\,
g^\outer\in\WSRhom{1}{q}\cap\WSRhom{1}{r}\cap\LR{\frac{3q}{3-q}}
},\\
\Yspace_3
&\coloneqq \WSR{2-1/r}{r}(\Gamma), \\
\Yspace_{2,3}
&\coloneqq \setcL{(g,h)\in\Yspace_2\times\Yspace_3}{
\int_{\Omega^\inner}g\,\dx=\int_{\Gamma}h\,\dS
},\\
\Yspace_4
&\coloneqq \setcl{h\in\WSR{1-1/r}{r}(\Gamma)^3}{
h\cdot\nvec=0
},\\
\Yspace_5
&\coloneqq\Yspace_6\coloneqq\R,\\
\Yspace_7
&\coloneqq\WSR{1-1/r}{r}(\Gamma).
\end{align*}
The bespoke framework of Sobolev spaces we shall employ in our investigation of \eqref{derivation_sseq_final} is then given by
\begin{align*}
&\Xrey\coloneqq\Xrey(\Omega)\coloneqq \XreyOne\times\Xspace_2\times\Xspace_3\times\Xspace_4,\\
&\Yspace\coloneqq\Yspace(\Omega)\coloneqq \Yspace_1\times\Yspace_{2,3}\times\Yspace_4\times\Yspace_5\times\Yspace_6\times\Yspace_7.
\end{align*}
In Theorem \ref{ThmLinearOperatorHom} we show that the operator corresponding to the appropriate linearization of \eqref{derivation_sseq_final} maps $\Xrey$ homeomorphically onto $\Yspace$.

The following embedding is valid:
\begin{prop}\label{PropEmbeddingsX1}
Let $\uvel\in\XreyOne$ with $q\in\bp{1,\frac{3}{2}}$, $r\in(3,\infty)$, and consider
$s\in\bb{\frac{2q}{2-q},\infty}$ and $t\in\bb{\frac{4q}{4-q},\infty}$. 
Then $\uvel\in\LR{s}(\Omega)\cap\WSRhom{1}{t}(\Omega)$.
If $s\geq\frac{3q}{3-2q}$ and $t\geq\frac{3q}{3-q}$, then
\begin{align}\label{PropEmbeddingsX1_Embedding1}
\norm{\uvel}_{s}+\norm{\grad\uvel}_{t}
\leq\Cc[ConstEmbeddingXOne1]{C} \norm{\uvel}_{\XreyOne}.
\end{align}
If $\frac{2q}{2-q}\leq s <\frac{3q}{3-2q}$, $\frac{4q}{4-q}\leq t <\frac{3q}{3-q}$,  
$\theta_s\coloneqq 2+\frac{3}{s}-\frac{3}{q}$ and
$\theta_t\coloneqq 1+\frac{3}{t}-\frac{3}{q}$,
then
\begin{align}\label{PropEmbeddingsX1_Embedding2}
\snorm{\Rey}^{\theta_s}\norm{\uvel}_{s}
+\snorm{\Rey}^{\theta_t}\norm{\grad\uvel}_{t}
\leq\const{ConstEmbeddingXOne1} \norm{\uvel}_{\XreyOne}.
\end{align}
Here $\const{ConstEmbeddingXOne1}=\const{ConstEmbeddingXOne1}(q,r,s,t,\Omega)>0.$ 
\end{prop}

\begin{proof}
The above estimates for the part $\uvel^\inner$ of $\uvel$ defined on a bounded domain follows directly 
from Sobolev embedding theorems.
Concerning the part $\uvel^\outer$ defined on an exterior domain, it follows from \cite[Lemma II.6.1]{GaldiBookNew} and the Sobolev inequality that
\begin{align*}
\norm{\grad\uvel^\outer}_\infty 
\leq\Cc{c}\bp{\snorm{\grad\uvel^\outer}_{1,r}+\norm{\grad\uvel^\outer}_{\frac{3q}{3-q}}}
\leq\Cc{c}\bp{\snorm{\uvel^\outer}_{2,r}+\snorm{\uvel^\outer}_{2,q}}
\leq\Cc{c}\norm{\uvel}_{\XreyOne}.
\end{align*}
Interpolation with the Sobolev-type inequality
\begin{align}\label{EstPropEmbeddingsX1SobolevIneq}
\norm{\grad\uvel}_{\frac{3q}{3-q}}
\leq\Cc{c}\norm{\grad^2\uvel}_{q}
\leq\Cc{c}\norm{\uvel}_{\XreyOne}
\end{align}
yields estimate \eqref{PropEmbeddingsX1_Embedding1} of $\grad\uvel$. 
Estimate \eqref{PropEmbeddingsX1_Embedding2} of $\grad\uvel$ follows by interpolating \eqref{EstPropEmbeddingsX1SobolevIneq}
with the trivial estimate
$\snorm{\Rey}^\frac{1}{4}\norm{\grad\uvel^\outer}_{\frac{4q}{4-q}}
\leq\norm{\uvel}_{\XreyOne}$.
The estimates \eqref{PropEmbeddingsX1_Embedding1}--\eqref{PropEmbeddingsX1_Embedding2} of $\uvel$ can be verified in a similar manner.
\end{proof}

\section{Auxiliary linear problem}\label{AuxiliaryLinearProblemSection}

Let $\Omega$ be a domain of the same type as in Section \ref{Preliminaries_Section}, \ie,
satisfying \eqref{Preliminaries_DomainTypeDef}. We further assume that the boundary $\Gamma$ is at least Lipschitz.
The linear system
\begin{align}\label{SystemOseenOmegaGeneral}
\begin{pdeq}
-\Div \fluidstress(\uvel,\upres) + \Rey \partial_3 \uvel&=f && \text{in } \Omega , \\
\Div \uvel &= g && \text{in } \Omega, \\
\jump{\uvel}&=0  &&\text{on } \Gamma, \\
\uvel \cdot \nvec &=\hone &&\text{on } \Gamma,\\
\projtang \jump{\fluidstress(\uvel,\upres)\nvec}&=\htwo && \text{on }\Gamma
\end{pdeq}
\end{align}
is an integral part of the linearization of \eqref{derivation_sseq_final}.
It is a two-phase strongly coupled Oseen $(\Rey\neq 0)$ or Stokes $(\Rey= 0)$ system. Since the coupling is strong, the questions of existence and  uniqueness of solutions as well as \textit{a priori} estimates hereof cannot be investigated by means of a simple decomposition into two classical Oseen/Stokes problems. In the following, we carry out an analysis of \eqref{SystemOseenOmegaGeneral} in the framework of the Sobolev spaces introduced in the previous section. Existence and uniqueness of solutions is
first shown in a setting of weak solutions, and strong \textit{a priori} estimates of Agmon--Douglis--Nirenberg type are subsequently established
first in the half space and then in the general case via a localization technique.
The main result of the section is contained in Theorem \ref{ThmOseenOmegaHom} and Theorem \ref{ThmStokesProblemOmega}.

\subsection{Weak solutions}\label{WeakSolutions}

We introduce a weak formulation of \eqref{SystemOseenOmegaGeneral} in the setting of the function spaces:
\begin{align*}
&\testsp \coloneqq \setcl{\varphi\in\CRci(\R^3)^3}{\varphi\cdot\nvec=0 \text{ on }\Gamma },\\
&\testsolsp\coloneqq \setcl{\phi\in\CRci(\R^3)^3}{\varphi\cdot\nvec=0\ \ton\Gamma,\ \Div\phi=0},\\
&\hsp\coloneqq  \overline{\testsp}^{\snorm{\cdot}_{1,2}} = 
\setcl{\varphi\in\DSRN{1}{2}\np{\R^3}^3}{\varphi\cdot\nvec=0 \text{ on }\Gamma },\\
&\hsolsp\coloneqq  \overline{\testsolsp}^{\snorm{\cdot}_{1,2}} = 
\setcl{\varphi\in\DSRN{1}{2}\np{\R^3}^3}{\varphi\cdot\nvec=0 \text{ on }\Gamma,\ \Div\phi=0 },\\
&\LRN{2}\np{\R^3}\coloneqq \setcL{p\in\LR{2}\np{\R^3}}{\int_{\Omega^\inner}p\,\dx = 0}.
\end{align*}

In the following, we establish existence and uniqueness as well as higher-order regularity of weak solutions to \eqref{SystemOseenOmegaGeneral} in this framework.
We start with the definition of a weak solution:

\begin{defn}\label{DefWeakSolution}
Let $f\in\hsolspdual$, 
$g\in\LR{2}(\R^3)$, 
$\hone\in\WSR{\half}{2}(\Gamma)$ and
$\htwo\in\WSR{-\half}{2}(\Gamma)^3$.
A vector field $\uvel\in\DSRN{1}{2}\np{\R^3}^3$ is called a 
\emph{weak solution to \eqref{SystemOseenOmegaGeneral}} 
if
\begin{align}\label{EqWeakSolutionDef}
\forall \phi\in\testsolsp:\quad
\int_{\R^3} 2\mu\,\symmgrad(\uvel):\symmgrad(\varphi)\,\dx
+\Rey\int_{\R^3}\partial_3\uvel\cdot\varphi
= \dpair{f}{\varphi}
+\dpair{\htwo}{\varphi}
\end{align}
as well as $\Div\uvel=g$ in $\Omega$ and $\uvel\cdot\nvec=\hone$ on $\Gamma$. 
\end{defn}

Existence of a weak solution $\uvel$ can be shown by standard techniques; we sketch a proof below. 

\begin{thm}\label{ThmOseenWeakSolutions}
Assume that the boundary $\Gamma$ is Lipschitz.
For every $f\in\hsolspdual$, 
$g\in\LR{2}\np{\R^3}$, 
$\hone\in\WSR{\half}{2}\np{\Gamma}$ and
$\htwo\in\WSR{-\half}{2}\np{\Gamma}^3$ 
satisfying
\begin{align}\label{ThmOseenWeakSolutions_CompCond}
\int_{\Omega^\inner}g\,\dx=\int_{\Gamma}\hone\,\dS
\end{align}
there is a weak solution $\uvel\in\DSRN{1}{2}\np{\R^3}^3$ to \eqref{SystemOseenOmegaGeneral}
satisfying
\begin{align}\label{EstimateWeakSolution}
\snorm{\uvel}_{1,2}
\leq\Cc{C} \bp{
\norm{f}_{\hsolspdual}
+\norm{g}_{2}
+\norm{\hone}_{\half,2}
+\norm{\htwo}_{-\half,2}},
\end{align}
where $\Cclast{C}=\Cclast{C}\np{\Gamma,\Rey}$.
\end{thm}

\begin{proof}
We sketch a proof of existence following \cite[Proof of Theorem VII.2.1]{GaldiBookNew} based on a Galerkin approximation.
To this end, a Schauder basis $\seqkN{\phi}\subset\testsolsp$ for the function space
$\setc{\phi\in\WSR{1}{2}\np{\R^3}}{\varphi\cdot\nvec=0\ \ton\Gamma,\ \Div\phi=0}$
satisfying $\int_{\Omega} 2\mu\,\symmgrad(\phi_k):\symmgrad(\phi_l)\,\dx = \kroneckerdelta_{k,l}$ is constructed.
This function space is clearly separable, whence such a basis can be constructed via a Gram--Schmidt procedure.
We consider first the
case $(g,\hone)=(0,0)$. Existence of an approximate solution of order $m\in\N$, that is, a vector field $\uvel_m\coloneqq \sum_{l=1}^m \xi_l\phi_l$ satisfying
the equation in \eqref{EqWeakSolutionDef} for all test functions in $\vecspan\set{\phi_1,\ldots,\phi_m}$, then follows directly from the fact that
the matrix $A\in\R^{m\times m}$, $A_{kl}\coloneqq \int_{\R^3}\partial_3\phi_l\cdot \phi_k$, is skew symmetric and $I+\Rey A$ therefore invertible. Specifically, 
the coefficient vector $\xi\coloneqq \np{I+\Rey A}^{-1}F$ with $F_k\coloneqq \dpair{f}{\phi_k}+\dpair{\htwo}{\phi_k}$ induces an approximate solution $\uvel_m$. Employing $\uvel_m$ itself
as a test function in the weak formulation, one obtains a uniform bound on $\norm{\symmgrad\np{\uvel_m}}_2$, which, since $\uvel_m$ is solenoidal, also implies a uniform bound as in \eqref{EstimateWeakSolution} on $\snorm{\uvel_m}_{1,2}$. A weak solution to \eqref{SystemOseenOmegaGeneral} is now obtained as the limit $\uvel\coloneqq \lim_{m\ra\infty}\uvel_m$ in $\hsolsp$.
The general case of non-vanishing $g$ and $\hone$ follows by a lifting argument. Employing a right inverse of the trace operator $\WSR{1}{2}\np{\R^3}\ra\WSR{\half}{2}\np{\Gamma}$, we find $\uvel_1\in\WSR{1}{2}\np{\R^3}$
with $\uvel_1=\hone$ on $\Gamma$ satisfying $\norm{\uvel_1}_{1,2}\leq \Cc{c}\norm{\hone}_{\half,2}$. The compatibility condition \eqref{ThmOseenWeakSolutions_CompCond} ensures
that $\int_{\Omega^\inner}g-\Div\uvel_1\,\dx=0$ so that we can find $\uvel_2\in\WSRhomN{1}{2}(\R^3)$ with 
$\Div \uvel_2 = g - \Div\uvel_1$ and satisfying $\uvel_2=0$ on $\Gamma$ as well as \eqref{EstimateWeakSolution}; see for example \cite[Theorem III.3.1 and III.3.6]{GaldiBookNew}. The ansatz $\uvel=\vvel+\uvel_1+\uvel_2$ now reduces
the problem to the case above with respect to the unknown $\vvel$. We thus conclude existence of a weak solution.
\end{proof}

A pressure $\upres$ can be associated to a weak solution $\uvel$ such that $(\uvel,\upres)$ becomes a solution to \eqref{SystemOseenOmegaGeneral} in the sense of distributions:
\begin{thm}\label{PropAssPressure}
Assume $f\in\hspdual$. To every weak solution $\uvel\in\DSRN{1}{2}\np{\R^3}$ to \eqref{SystemOseenOmegaGeneral} 
there is a unique $\upres\in\LRN{2}\np{\R^3}$ such that 
\begin{align}\label{EqAssPressure}
\forall\phi\in\testsp:\ \int_{\R^3} 2\mu\,\symmgrad(\uvel):\symmgrad(\phi)\,\dx
+\Rey\int_{\R^3}\partial_3\uvel\cdot\phi\,\dx
=\int_{\R^3}\upres\,\Div\phi\,\dx
+\dpair{f}{\phi}
+\dpair{\htwo}{\phi}
\end{align}
and 
\begin{align}\label{EstAssPressure}
\norm{\upres}_2\leq \Cc[ConstAssPressure]{C}
\bp{\norm{f}_{\hspdual}+\norm{g}_{2}
+\norm{\hone}_{\half,2}+\norm{\htwo}_{-\half,2}}
\end{align}
with $\const{ConstAssPressure}=\const{ConstAssPressure}(\Gamma)>0$. 
\end{thm}

\begin{proof}
The proof is modification of \cite[Lemma VII.1.1]{GaldiBookNew}. For arbitrary $M\in\N$ with $M>\delta(\Omega)$ we let 
$\hspm\coloneqq \setc{\phi\in\hsp}{\supp\phi\subset\B_M}$ and consider the functional
\begin{align*}
F_M:\hspm\ra\R,\
F_M(\phi)\coloneqq \int_{\B_M}2\mu\, \symmgrad(\uvel):\symmgrad(\phi)\,\dx
+\Rey\int_{\B_M}\partial_3\uvel\cdot\phi\,\dx
-\dpair{f}{\phi}
-\dpair{\htwo}{\phi},
\end{align*}
which is continuous on $\hspm$ by Sobolev embedding. We further introduce the space
\begin{align*}
\LRNm{2}\coloneqq \setcL{p\in\LR{2}\np{\B_M}}{\int_{\Omega^\outer\cap\B_M} p\,\dx=\int_{\Omega^\inner} p\,\dx=0}
\end{align*}
and the operator $\Div:\hspm\ra\LRNm{2}$. The operator is surjective, which is seen by solving for arbitrary
$p\in\LRNm{2}$ the two equations
\begin{align*}
\begin{pdeq}
\Div\uvel^\inner &= p && \tin\Omega^\inner,\\
\uvel^\inner &= 0 && \ton\Gamma,
\end{pdeq}\qquad
\begin{pdeq}
\Div\uvel^\outer &= p && \tin\Omega^\outer\cap\B_M,\\
\uvel^\outer &= 0 && \ton\Gamma\cup\partial\B_M,
\end{pdeq}
\end{align*}
according to \cite[Theorem III.3.1]{GaldiBookNew}. It follows that the operator and hence also its adjoint $\DivAdjoint$ are both closed.
Since $\uvel$ is a weak solution, $F_M$ vanishes on the kernel of $\Div$ and consequently belongs to the image of $\DivAdjoint$. We thus obtain
a function $\upres_M\in\LRNm{2}$ such that $\dpair{F_M}{\phi}\coloneqq\int_{\B_R}\upres_M \Div \phi\,\dx$.
After possibly adding a constant to $\upres_{M+1}^\outer$, we may assume  
$\upres_{M+1}=\upres_M$ in $\B_M$. The sequence $\set{\upres_M}_{M=1}^\infty$ then induces a pressure $\upres\in\LRloc{2}\np{\R^3}$ satisfying \eqref{EqAssPressure} and
$\int_{\Omega^\inner} \upres\,\dx=0$.
It remains to establish $\LR{2}\np{\R^3}$ integrability of $\upres$. If $\Rey=0$, the functional $F_M$ remains continuous if $\hspm$ is replaced with $\hsp$. In this case the argument above directly yields a pressure $\upres\in\LRN{2}\np\R^3$ satisfying \eqref{EqAssPressure}. Subsequently choosing a function $\phi\in\hsp$ with $\Div\phi=\upres$ in $\R^3$ and $\snorm{\phi}_{1,2}\leq \Cc{c}\norm{\upres}_2$,
which can be done via \cite[Theorem III.3.1 and Theorem III.3.6]{GaldiBookNew}, one obtains
\eqref{EstAssPressure} by inserting $\phi$ into \eqref{EqAssPressure}. 
If $\Rey\neq 0$, it suffices to observe that $(\uvel,\upres)$ is a weak solution to an Oseen problem in the exterior domain
$\Omega^\outer$, whence \cite[Theorem VII.7.2]{GaldiBookNew} yields $\upres\in\LRN{2}\np\R^3$ satisfying \eqref{EqAssPressure}--\eqref{EstAssPressure}.
\end{proof}

Provided $\uvel$ and $\upres$ are sufficiently regular, integration by parts in \eqref{EqAssPressure} reveals that $\np{\uvel,\upres}$ is a classical solution to \eqref{SystemOseenOmegaGeneral}.
Higher-order regularity of $(\uvel,\upres)$ can be obtained via a classical approach  under appropriate regularity assumptions on the data: 

\begin{thm}\label{ThmHigherRegularityWeakSolutions}
Let $k\in\N_0$ and assume that $\Gamma$ is a $\CR{k+3}$-smooth closed surface. If  
\begin{align*}
f\in\WSR{k}{2}(\Omega)^3, \quad
g\in\WSR{k+1}{2}(\Omega), \quad
\hone\in\WSR{k+3/2}{2}(\Gamma), \quad
\htwo\in\WSR{k+1/2}{2}(\Gamma)^3,
\end{align*}
then a weak solution $\uvel\in\DSRN{1}{2}\np{\R^3}^3$ to \eqref{SystemOseenOmegaGeneral} with associated pressure  
$\upres\in\LRloc{2}(\R^3)$ satisfying \eqref{EqAssPressure} also satisfies 
\begin{align}\label{ThmHigherRegularityWeakSolutions_RegularityAssertion}
\uvel\in\bigcap_{\ell=0}^k\WSRhom{\ell+2}{2}(\Omega),\qquad 
\upres\in\bigcap_{\ell=0}^k\WSRhom{\ell+1}{2}(\Omega).
\end{align}
\end{thm}

\begin{proof}
The proof is a standard application of a well-known technique based on difference quotients. In fact, with only minimal modification it is similar to
a proof of higher-order regularity for solutions to the Stokes system with prescribed normal velocity and tangential stress on the boundary; see  \cite[Proof of Theorem 2]{SolonnikovScadilov73}.
For the sake of completeness, we sketch the proof. We include only the case $\hone=0$ and $k=0$. The general case $\hone\neq 0$ and $k>0$ follows by a simple lifting technique and iteration procedure, respectively. Since higher-order regularity in $\Omega$ away from the boundary $\Gamma$ is well known for Stokes systems (see for example \cite[Section IV.2]{GaldiBookNew}), we focus on regularity up to the boundary $\Gamma$. To this end, consider an arbitrary $\tilde{x}\in\Gamma$ and
choose a cube $\cuberx\subset\R^3$, centered at $\tilde{x}$ with side length $r$, such that $\Gamma\cap\cuberx$ can be parameterized by a $\CR{3}$
function $\omega$. Without loss of generality, we may assume that $\tilde{x}=0$ and 
\begin{align*}
\Gamma\cap\cuberx = \Gamma\cap\cuber = \setcl{\bp{x_1,x_2,\omega(x_1,x_2)}}{\np{x_1,x_2}\in\cubepr},
\end{align*}
where $\cubepr\subset\R^2$ is the two-dimensional cube centered around $0$, and that $\grad\omega(0)=0$ as well as $\norm{\grad\omega}_\infty\ra 0$ as $r\ra 0$.
Let $\cutoff\in\CRci\np{\R^3}$ be a cut-off function with $\cutoff=1$ on $\cubehalfr$ and put 
\begin{align*}
&\Phi(x)\coloneqq \bp{x_1,x_2,x_3-\omega\np{x_1,x_2}},\\
&\Uvel:\cuber\ra\R^3,\ \Uvel\coloneqq \bb{\grad\Phi\np{\cutoff\uvel}}\circ\Phi^{-1},\quad
\Upres:\cuber\ra\R,\ \Upres\coloneqq \nb{\cutoff\upres}\circ\Phi^{-1}.
\end{align*}
We introduce test functions 
\[
\testspcube\bp{\cuber}\coloneqq \setc{\psi\in\WSR{1}{2}\bp{\cuber}^3}{\psi=0\ \ton\partial\cuber,\ \psi\cdot \nvec=0\ \ton\Gamma_0}
\]
with 
\[
\Gamma_0\coloneqq \setc{x\in\cuber}{x_3=0}.
\]
The transformed fields $\np{\Uvel,\Upres}$ satisfy the weak formulation
\begin{align}\label{ThmHigherRegularityWeakSolutions_WeakEqOnCube}
\forall\psi\in\testspcube\bp{\cuber}:\ \int_{\R^3} 2\mu\,\symmgrad(\Uvel):\symmgrad(\psi)\,\dx -
\int_{\R^3}\Upres\,\Div\psi\,\dx
=
\dpair{F_0}{\psi}
+\dpair{F_1}{\grad\psi},
\end{align}
where $F_0$ contains up to first-order terms of $\uvel$ and zeroth-order terms of $\upres$, and
$F_1$ contains first-order terms of $\Uvel$  multiplied with components of $\grad\omega$. 
The magnitude of the latter terms can be made small by choosing $r$  small.
Difference quotients are denoted by
$\diffq{h}{l}\Uvel(x)\coloneqq \frac{1}{h}\bp{\Uvel(x+h\e_l)-\Uvel(x)}$.
Importantly, difference quotients $\diffq{-h}{l}\diffq{h}{l}\Uvel$ in tangential direction $l=1,2$ are admissible as test functions in $\testspcube\bp{\cuber}$ and can therefore be inserted into \eqref{ThmHigherRegularityWeakSolutions_WeakEqOnCube},
which yields an estimate of $\norm{\symmgrad\np{\diffq{h}{l}\Uvel}}_2$ in terms of lower-order norms of $\uvel$ and $\upres$
as well as
$\norm{\diffq{h}{l}\Upres}_2$. A similar bound on $\norm{\grad\diffq{h}{l}\Uvel}_2$ follows from Korn's inequality.
Choosing in \eqref{ThmHigherRegularityWeakSolutions_WeakEqOnCube} a test function $\diffq{-h}{l}\psi\in\testspcube\bp{\cuber}$ with
$\Div\psi=\diffq{h}{l}{\Upres}$,
a bound on $\norm{\diffq{h}{l}\Upres}_2$ in terms
of lower-order norms of $\uvel$ and $\upres$ is obtained.
Such a test function is constructed by setting 
$\psi\coloneqq \psi^+$ in $\cubeplus\coloneqq \setc{x\in\cuber}{x_3>0}$ and 
$\psi\coloneqq \psi^-$ in $\cubeminus\coloneqq \setc{x\in\cuber}{x_3<0}$ where
\begin{align*}
\begin{pdeq}
\Div\psi^+ &= \diffq{h}{l}\Upres && \tin\cubeplus,\\
\psi^+ &= 0 && \ton\partial\cubeplus,
\end{pdeq}\qquad
\begin{pdeq}
\Div\psi^- &= \diffq{h}{l}\Upres && \tin\cubeminus,\\
\psi^- &= 0 && \ton\partial\cubeminus.
\end{pdeq}
\end{align*}
Existence of solutions to the two equations above and the estimates $\norm{\psi^\pm}_{1,2}\leq\Cc{c}\norm{\diffq{h}{l}\Upres}_2 $
are secured by \cite[Corollary III.5.1]{GaldiBookNew}.
It follows that $\norm{\grad\diffq{h}{l}\Uvel}_2+\norm{\diffq{h}{l}\Upres}_2$
is uniformly bounded in $h$, which implies $\partial_l\grad\Uvel,\partial_l\Upres\in\LR{2}\np{\cuber}$ for $l=1,2$.
Since $\Div\Uvel=G$ with $G$ containing only zeroth-order terms of $\uvel$, $\partial_3^2\Uvel\in\LR{2}\np{\cuber}$ follows as a combination
of $\partial_3\Div\Uvel=\partial_3 G$ and the regularity of $\Uvel$'s tangential derivatives. Finally, the distributional derivative $\partial_3\Upres$ can now be isolated in \eqref{ThmHigherRegularityWeakSolutions_WeakEqOnCube}
to deduce in each half of the cube
$\Upres\in\WSR{1}{2}\bp{\cubeplus}$ and
$\Upres\in\WSR{1}{2}\bp{\cubeminus}$.
It follows that  
$(\uvel,\upres)\in\WSR{2}{2}\bp{\calo^\inner(\tilde x)}\times\WSR{1}{2}\bp{\calo^\inner\np{\tilde x}}$ as well as
$(\uvel,\upres)\in\WSR{2}{2}\bp{\calo^\outer(\tilde x)}\times\WSR{1}{2}\bp{\calo^\outer\np{\tilde x}}$, where
$\calo(\tilde x)$ is a neighborhood of $\tilde{x}$ and
$\calo^\inner(\tilde x)\coloneqq \calo(\tilde x)\cap\Omega^\inner$,
$\calo^\outer(\tilde x)\coloneqq \calo(\tilde x)\cap\Omega^\outer$. Higher-order regularity of $\np{\uvel,\upres}$ up to the boundary $\Gamma$ is thereby established. 
\end{proof}

Finally, uniqueness of a weak solution to \eqref{SystemOseenOmegaGeneral} can be established. In fact, uniqueness can be obtained in a
much larger class of distributional solutions with even less summability at spatial infinity than $\uvel\in\LR{6}(\Rthree)$ satisfied by
a weak solution via Sobolev embedding. The theorem below is not optimal in
this respect, but suffices for the purposes of this article.

\begin{thm}\label{PropUniquenessWeakSolutionGeneral}
Let $\Gamma$ be a $\CR{2}$-smooth closed surface, and let $(\uvel,\upres)\in\WSRloc{1}{2}(\Rthree)^3\times\LRloc{2}(\Rthree)$ be a solution to \eqref{SystemOseenOmegaGeneral} in the
sense of \eqref{EqAssPressure} with $\uvel\in\LR{q}(\R^3)^3$ and $\upres\in\LR{r}\np{\Rthree}$ for some $q,r\in(1,\infty)$. If $(f,g,\hone,\htwo)=(0,0,0,0)$, then $\uvel=0$. 
\end{thm}
\begin{proof}
The integrability assumption $\uvel\in\LR{q}(\R^3)$ combined with the fact that $\np{\uvel,\upres}$ solves a classical Stokes $(\Rey=0)$ or Oseen $(\Rey\neq0)$ problem with homogeneous right-hand side
in the exterior domain $\Omega^\outer$
implies that $\uvel$ exhibits the same pointwise rate of decay as
the three-dimensional Stokes fundamental solution $(\Rey=0)$ or 
the three-dimensional Oseen fundamental solution $(\Rey\neq 0)$;
see \cite[Theorem V.3.2 and Theorem VII.6.2]{GaldiBookNew} for example.
This means that  $\uvel(x)=\bigo(\snorm{x}^{-1})$ as $\snorm{x}\ra\infty$.
Moreover, we obtain $\upres\in\bigo\np{\snorm{x}^{-2}}$.
Let $\cutoff\in\CRci(\R)$ be a cut-off function with $\chi=1$ for $\snorm{x}<1$ and $\chi=0$ for $\snorm{x}>2$, and put $\chi_R\coloneqq \chi\bp{\frac{\snorm{x}}{R}}$.
Then $\cutoff_R\uvel$ is admissible as a test function in \eqref{EqAssPressure},
which implies
\begin{align*}
\int_{\R^3} 2\mu\,\symmgrad(\uvel):\bp{\cutoff_R\symmgrad(\uvel)+\grad\cutoff_R\otimes\uvel + \uvel\otimes\grad\cutoff_R}\dx
+\Rey\int_{\R^3}\np{\partial_3\uvel\cdot\uvel}\cutoff_R\,\dx
=\int_{\R^3}\upres\np{\grad\cutoff_R\cdot\uvel}\dx.
\end{align*}
Utilizing that $\uvel=\bigo\bp{\snorm{x}^{-1}}$, we use H\"older's inequality to estimate
\begin{align*}
\snormL{\int_{\R^3} \symmgrad\np{\uvel}:\grad\cutoff_R\otimes\uvel\,\dx}
&\leq \Cc{c}\norm{\symmgrad\np{\uvel}}_2 \, \Bp{\int_{\B_{2R,R}}\frac{\snorm{\uvel}^2}{R^2}\,\dx}^\half
\leq \Cc{c}R^{-1/2}\overset{R\ra\infty}{\longrightarrow} 0.
\end{align*}
Furthermore, in the Oseen case ($\Rey\neq0$) we even have the better averaged decay estimate
\[
\int_{\partial\ball_r} \snorm{\uvel}^2\,\dS\leq\Cc{c}r^{-1};
\]
see \cite[Exercise VII.6.1]{GaldiBookNew}; 
which leads to
\begin{align*}
\snormL{\int_{\R^3}\np{\partial_3\uvel\cdot\uvel}\cutoff_R\,\dx} = \snormL{\half\int_{\R^3}\snorm{\uvel}^2 \partial_3\cutoff_R\,\dx}
\leq \Cc{c}\int_R^{2R}\int_{\partial\ball_r} \frac{\snorm{\uvel}^2}{R}\,\dS\dr
\leq \Cc{c} R^{-1}
\overset{R\ra\infty}{\longrightarrow} 0.
\end{align*}
Since also
\begin{align*}
\snormL{\int_{\R^3}\upres\np{\grad\cutoff_R\cdot\uvel}\dx}\leq \Cc{c}
\int_{\ball_{2R,R}} R^{-4}\,\dx \leq \Cc{c} R^{-1}
\overset{R\ra\infty}{\longrightarrow} 0,
\end{align*}
we deduce $\norm{\symmgrad\np\uvel}_2=0$ and thus $\uvel=0$. 
\end{proof}

\subsection{Twofold half space} 

The main challenge towards \textit{a priori} $\LR{r}$ estimates of solutions to \eqref{SystemOseenOmegaGeneral}, \ie, \textit{a priori} estimates of 
Agmon--Douglis--Nirenberg type, is to obtain such estimates in the half-space case under disregard of the lower-order terms in the equations; the general case
then follows via a localization argument. We therefore first consider system 
\begin{align}\label{SystemStokesTwofoldHalfSpace}
\begin{pdeq}
\Div \fluidstress(\uvel,\upres) &=f && \text{in } \Rdot^3 , \\
\Div \uvel &= g && \text{in } \Rdot^3, \\
\jump{\uvel}&=h_0  &&\text{on } \partial\Rdot^3, \\
\restriction{\uvel}{\R^3_+} \cdot \nvec &=\hone &&\text{on } \partial\Rdot^3,\\
\projtang \jump{\fluidstress(\uvel,\upres)\nvec}&=\htwo && \text{on }\partial\Rdot^3,
\end{pdeq}
\end{align}
where $\nvec=-e_3$. We shall implicitly identify $\partial\Rdot^3$ with $\Rtwo$. 

In the celebrated work \cite{ADN2} of \textsc{Agmon}, \textsc{Douglis} and \textsc{Nirenberg}, \textit{a priori}
$\LR{r}$ estimates for strong solutions to elliptic systems with boundary values of a certain type were established.
For Stokes systems, both the Dirichlet boundary condition and the slip boundary condition that make up the boundary values in \eqref{SystemStokesTwofoldHalfSpace} fall into
the category of so-called Agmon--Douglis--Nirenberg problems covered by \cite{ADN2}. However, since the system \eqref{SystemStokesTwofoldHalfSpace} is strongly coupled, it
does not itself fall into this category, nor can it be decomposed into two systems to which the estimates from \cite{ADN2} can be applied separately. It therefore seems unavoidable that $\LR{r}$ estimates
for solutions to \eqref{SystemStokesTwofoldHalfSpace} have to be established without the help of \cite{ADN2}. We present a comprehensive proof below. We shall not employ the technique from \cite{ADN2} based on singular integrals, but choose instead an approach based on Fourier multipliers and real interpolation that seems particularly well suited for coupled
systems such as \eqref{SystemStokesTwofoldHalfSpace}.

The main result on \textit{a priori} $\LR{r}$ estimates of solutions to \eqref{SystemStokesTwofoldHalfSpace} is contained in Theorem \ref{ADN_TwoFoldHalfSpace} below.
The proof is divided into two lemmas, Lemma \ref{ADN_DirichletStokes} and Lemma \ref{ADN_LinTwoPhaseFlowProblem}, which the reader may find interesting in their own right.
For technical reasons, it is convenient to decompose 
both the data and the solution to \eqref{SystemStokesTwofoldHalfSpace} into
one part with lower frequency support and another part with higher frequency support in tangential directions $\e_1,\e_2$.
We shall repeatedly employ the Fourier transform $\FT_\Rtwo$ with respect to these two directions. 
To this end, observe that
$\FT_\Rtwo\bb{\uvel(\cdot,x_3)}(\xip)$
is well-defined in the sense of distributions $\TDR\np{\Rthree}$ when $\uvel\in\LR{r}\np\Rthree$ for some $r\in\np{1,\infty}$,
which will be the case whenever such an expression appears below.

\begin{lem}\label{ADN_DirichletStokes}
Let $r\in\np{1,\infty}$ and $b\in\WSR{2-1/r}{r}(\Rtwo)^3$ with
$\supp\FT_\Rtwo\nb{b}\subset\R^2\setminus\ball_1(0)$. Then there is
a solution $(\uvel,\upres)\in\WSR{2}{r}\np{\Rdot^3}^3\times\WSR{1}{r}\np{\Rdot^3}$ to 
\begin{align}\label{ADN_DirichletStokes_Eq}
\begin{pdeq}
\Div \fluidstress(\uvel,\upres) &=0 && \text{in } \Rdot^3 , \\
\Div \uvel &= 0 && \text{in } \Rdot^3, \\
\uvel&=b&&\text{on } \partial\Rdot^3,
\end{pdeq}
\end{align}
which satisfies 
\begin{align}\label{ADN_DirichletStokes_Est}
\norm{\uvel}_{2,r}+\norm{\upres}_{1,r} \leq \Cc{C}\,\norm{b}_{2-1/r,r},
\end{align}
where $\Cclast{C}=\Cclast{C}(r)$. Moreover,
$\FT_\Rtwo\bb{\uvel(\cdot,x_3)}(\xip)$ and 
$\FT_\Rtwo\bb{\upres(\cdot,x_3)}(\xip)$
are supported away from  $(\xip,x_3)\in\B_{1/2}(0)\times\R$.
\end{lem}

\begin{proof}
A solution to \eqref{ADN_DirichletStokes_Eq} can be constructed explicitly. To this end, consider first a sufficiently smooth right-hand side
$b\in\SR\np{\R^2}^3$ with $\supp\FT_\Rtwo\nb{b}\subset\R^2\setminus\ball_{1/2}(0)$. 
We employ the notation $\ft{b}\coloneqq \FT_\Rtwo\nb{b}$ and $\vvel\coloneqq (\uvel_1,\uvel_2)$, $\wvel\coloneqq \uvel_3$ as well as 
$b_\vvel\coloneqq (b_1,b_2)$ and $b_\wvel\coloneqq b_3$.
An application of the Fourier transform $\FT_\Rtwo$ with respect to $x'\in\R^2$ in \eqref{ADN_DirichletStokes_Eq} yields
\begin{align}\label{SystemStokesDirichletDHspaceFourierTransform}
\begin{pdeq}
-\mu \snorm{\xip}^2 \ft \vvel +\mu\partial_3^2 \ft\vvel- i\xip\ft \upres &=0 && \text{in } \Rdot^3 , \\
-\mu \snorm{\xip}^2 \ft \wvel +\mu\partial_3^2 \ft\wvel- \partial_3\ft \upres &=0 && \text{in } \Rdot^3 , \\
i\xip\cdot \ft \vvel + \partial_3 \ft \wvel &= 0 && \text{in } \Rdot^3, \\
(\ft\vvel,\ft\wvel)&=(\ft b_\vvel,\ft b_\wvel) &&\text{on } \partial\Rdot^3.
\end{pdeq}
\end{align}
Therefore $\upres$ satisfies $\snorm{\xip}^2\,\ft\upres-\partial_3^2\,\ft\upres =0$ and thus
\begin{align*}
\ft\upres(\xip,x_3) =
\begin{pdeq}
&A_1(\xip)\e^{-\snorm{\xip}x_3} &&\tif\ x_3>0,\\
&A_2(\xip)\e^{\snorm{\xip}x_3}  &&\tif\ x_3<0.
\end{pdeq}
\end{align*}
We insert $\upres$ into 
\eqrefsub{SystemStokesDirichletDHspaceFourierTransform}{1} and \eqrefsub{SystemStokesDirichletDHspaceFourierTransform}{2}
and solve the resulting differential equations. 
Taking into account the boundary conditions \eqrefsub{SystemStokesDirichletDHspaceFourierTransform}{4},
we obtain
\begin{align*}
\ft\uvel =
\begin{pdeq}
&\Bb{\frac{A_1(\xip) x_3}{2\mu\snorm{\xip}}
\begin{pmatrix}
-i\xip \\
\snorm{\xip}
\end{pmatrix}
+
\begin{pmatrix}
\ft b_v\\ \ft b_w
\end{pmatrix}
}
e^{-\snorm{\xip}x_3} && \tif\ x_3>0,\\
&\Bb{\frac{A_2(\xip) x_3}{2\mu\snorm{\xip}}
\begin{pmatrix}
i\xip \\
\snorm{\xip}
\end{pmatrix}
+
\begin{pmatrix}
\ft b_v\\ \ft b_w
\end{pmatrix}
}
e^{\snorm{\xip}x_3}&& \tif\ x_3<0.
\end{pdeq}
\end{align*}
Inserting the above formula for $\ft\uvel$ into \eqrefsub{SystemStokesDirichletDHspaceFourierTransform}{3}, we find that
\begin{align*}
A_1(\xip)=A_2(\xip)=2\mu\bp{\sgn\np{x_3}\snorm{\xip}\ft b_\wvel-i\xip\cdot \ft b_\vvel}.
\end{align*}
Consequently, a solution to \eqref{ADN_DirichletStokes_Eq} is given by
\begin{align}\label{ADN_DirichletStokes_DefOfuvelupres}
\begin{aligned}
&\uvel(x',x_3)\coloneqq \iFT_\Rtwo\bb{M_b(\xip,x_3)\,e^{-\snorm{\xip}\snorm{x_3}}},\\
&\upres(x',x_3)\coloneqq \iFT_\Rtwo\bb{m_b(\xip,x_3)\,e^{-\snorm{\xip}\snorm{x_3}}},
\end{aligned}
\end{align}
where
\begin{align*}
&M_b(\xip,x_3)\coloneqq {\frac{ \bp{\sgn\np{x_3}\snorm{\xip}\ft b_\wvel-i\xip\cdot \ft b_\vvel}\snorm{x_3}}{\snorm{\xip}}
\begin{pmatrix}
-i\xip \\
\sgn(x_3)\snorm{\xip}
\end{pmatrix}
+
\begin{pmatrix}
\ft b_v\\ \ft b_w
\end{pmatrix}},\\
&m_b(\xip,x_3)\coloneqq 2\mu\bp{\sgn\np{x_3}\snorm{\xip}\ft b_\wvel-i\xip\cdot \ft b_\vvel}.
\end{align*}
Although $M_b$ has a singularity, $(\uvel,\upres)$ as defined above is a well-defined solution, smooth on $\Rdot^3$ even, due to the assumption that $\ft b(\xip)$ has support away from $0$.
In order to provide an estimate for the solution, we
let $\cutoffkappa_{1/4}\in\CRci\np{\R^2}$ with $\kappa_{1/4}=0$ on $\ball_{1/4}(0)$ and
$\kappa_{1/4}=1$ on $\R^2\setminus\ball_{1/2}(0)$, and 
consider the truncation  
\begin{align}\label{ADN_DirichletStokes_TruncatedSolOpr}
\soloprvel:\SR\np{\Rtwo}^3\ra\SR\np{\Rthree}^3,\quad
\soloprvel(\phi)\coloneqq \iFT_\Rtwo\Bb{\cutoffkappa_{1/4}(\xip) M_\phi(\xip,x_3)\,e^{-\snorm{\xip}\snorm{x_3}}}
\end{align}
of the solution operator.
The singularity of $M_\phi$ makes it necessary to employ the truncation $\cutoffkappa_{1/4}$ to ensure that $\soloprvel$ is well-defined.
We shall use real interpolation to show that $\soloprvel$ extends to a bounded operator 
$\soloprvel:\WSR{2-1/r}{r}\np{\Rtwo}\ra\WSR{2}{r}\np{\Rdot^3}$. To this end, we 
observe for $m\in\N_0$ and any $x_3\in\R$  that
the symbol $\xip\mapsto\np{\snorm{\xip}\snorm{x_3}}^m\e^{-\snorm{\xip}\snorm{x_3}}$ is an $\LR{r}\np{\Rtwo}$-multiplier.
Specifically, one may verify that  
\begin{align*}
\sup_{x_3\in\R}\sup_{\varepsilon\in\{0, 1\}^{2}}\sup_{\xip\in\Rtwo}
\left|{\xip_1}^{\varepsilon_1}{\xip_2}^{\varepsilon_2}
\partial_{\xip_1}^{\varepsilon_1}\partial_{\xip_2}^{\varepsilon_2}
\bb{\np{\snorm{\xip}\snorm{x_3}}^m\e^{-\snorm{\xip}\snorm{x_3}}} \right|<\infty,
\end{align*}
whence it follows from the Marcinkiewicz Multiplier Theorem (see for example \cite[Corollary 6.2.5]{Grafakos2}) that the Fourier-multiplier operator with  symbol 
$\xip\mapsto\np{\snorm{\xip}\snorm{x_3}}^m\e^{-\snorm{\xip}\snorm{x_3}}$ is a bounded operator on $\LR{r}\np\Rtwo$ with operator norm independent of $x_3$, that is,
\begin{align}\label{PurelyOscProblem_HomoBrdData_MultiplierOprNormIndep}
\sup_{x_3\in\R}\,\normL{\phi\mapsto\FT_{\Rtwo}
\Bb{\np{\snorm{\xip}\snorm{x_3}}^m\e^{-\snorm{\xip}\snorm{x_3}}\FT_{\Rtwo}\nb{\phi}}}_{\linoprspace\np{\LR{r}\np{\Rtwo},\LR{r}\np{\Rtwo}}}<\infty.
\end{align}
We return to \eqref{ADN_DirichletStokes_TruncatedSolOpr} and employ \eqref{PurelyOscProblem_HomoBrdData_MultiplierOprNormIndep} to deduce
\begin{align*}
\norm{\grad^2_{x'}\soloprvel(\phi)}_{\LR{\infty}_{x_3}(\R;\LR{r}(\Rtwo))}
&\leq \Cc{c}\,\norm{\grad ^2 \phi}_{\LR{r}\np{\Rtwo}},\\
\norm{\partial^2_{x_3}\soloprvel(\phi)}_{\LR{\infty}_{x_3}(\Rdot;\LR{r}(\Rtwo))} 
&\leq \Cc{c}\,\norm{\grad ^2 \phi}_{\LR{r}\np{\Rtwo}},\\
\norm{\soloprvel(\phi)}_{\LR{\infty}_{x_3}(\R;\LR{r}(\Rtwo))}
&\leq \Cc{c}\,\norm{\phi}_{\LR{r}\np{\Rtwo}},
\end{align*}
where the restriction in the norm of the left-hand side to the twofold real line $\Rdot$ in the second estimate is required since
$\partial_{x_3}\soloprvel(\phi)$ has a singularity at $x_3=0$.
It follows that
\begin{align}\label{ADN_DirichletStokes_TruncatedSolOpr_FirstInterpolationEst}
\norm{\grad^2\soloprvel(\phi)}_{\LR{\infty}_{x_3}(\R;\LR{r}(\Rtwo))} +\norm{\soloprvel(\phi)}_{\LR{\infty}_{x_3}(\R;\LR{r}(\Rtwo))}
\leq \Cc{c}\,\norm{\phi}_{\WSR{2}{r}\np{\Rtwo}}.
\end{align}
This estimate shall serve as an interpolation endpoint. To obtain the opposite endpoint, we again employ \eqref{PurelyOscProblem_HomoBrdData_MultiplierOprNormIndep} to infer
\begin{align*}
\sup_{x_3\in\R}\,\norm{\snorm{x_3}\,\grad^2_{x'}\soloprvel\np{\phi}}_{\LR{r}(\Rtwo)}
&\leq \Cc{c}\,\norm{\grad \phi}_{\LR{r}\np{\Rtwo}},\\
\sup_{x_3\in\Rdot}\,\norm{\snorm{x_3}\,\partial^2_{x_3}\soloprvel\np{\phi}}_{\LR{r}(\Rtwo)}
&\leq \Cc{c}\,\norm{\grad \phi}_{\LR{r}\np{\Rtwo}},\\
\sup_{x_3\in\R}\,\norm{\snorm{x_3}\,\soloprvel\np{\phi}}_{\LR{r}(\Rtwo)}
&\leq \Cc{c}\,\norm{\phi}_{\LR{r}\np{\Rtwo}},
\end{align*}
where the last estimate relies on the truncation introduced in $\soloprvel$.
It follows that
\begin{align}\label{ADN_DirichletStokes_TruncatedSolOpr_SecondInterpolationEst}
\begin{aligned}
\norm{\grad^2\soloprvel\np{\phi}}_{\LR{1,\infty}_{x_3}\np{\Rdot;\LR{r}(\Rtwo)}}
+\norm{\soloprvel\np{\phi}}_{\LR{1,\infty}_{x_3}\np{\Rdot;\LR{r}(\Rtwo)}}
\leq \Cc{c}\,\norm{\phi}_{\WSR{1}{r}\np{\Rtwo}}.
\end{aligned}
\end{align}
Real interpolation yields
\begin{align*}
\Bp{\LR{1,\infty}\bp{\Rdot;\LR{r}\np{\Rtwo}},\LR{\infty}\bp{\Rdot;\LR{r}\np{\Rtwo}}}_{1-1/r,r} 
&= \LR{r}\bp{\Rdot,\LR{r}\np{\Rtwo}},\\
\Bp{\WSR{2}{r}\np{\Rtwo},\WSR{1}{r}\np{\Rtwo}}_{1-1/r,r} 
&= \WSR{2-1/r}{r}\np{\Rtwo}.
\end{align*}
Consequently, \eqref{ADN_DirichletStokes_TruncatedSolOpr_FirstInterpolationEst} and \eqref{ADN_DirichletStokes_TruncatedSolOpr_SecondInterpolationEst}
imply 
\begin{align*}
\norm{\soloprvel\np{\phi}}_{\WSR{2}{r}\np{\Rdot^3}}\leq \Cc{c}\,\norm{\phi}_{\WSR{2-1/r}{r}\np{\Rtwo}},
\end{align*}
whence $\soloprvel$ extends to a bounded operator $\soloprvel:\WSR{2-1/r}{r}\np{\Rtwo}\ra\WSR{2}{r}\np{\Rdot^3}$.
Recalling the formula \eqref{ADN_DirichletStokes_DefOfuvelupres} for the solution $u$ to \eqref{ADN_DirichletStokes_Eq} and that
$\supp\FT_\Rtwo\nb{b}\subset\R^2\setminus\ball_{1/2}(0)$, we clearly have
$\uvel=\soloprvel(b)$.
It follows that
$\norm{\uvel}_{2,r}\leq \Cclast{c}\,\norm{b}_{2-1/r,r}$. 
In a completely similar manner, one shows that also 
$\norm{\upres}_{1,r}\leq \Cclast{c}\,\norm{b}_{2-1/r,r}$. Thus the lemma follows for this particular choice of $b\in\SR(\Rtwo)$.
Since any $b\in\WSR{2-1/r}{r}(\Rtwo)$ with
$\supp\FT_\Rtwo\nb{b}\subset\R^2\setminus\ball_1(0)$ can be approximated in $\WSR{2-1/r}{r}(\Rtwo)$ by a sequence $\seqkN{b}\subset\SR(\Rtwo)$
with $\supp\FT_\Rtwo\nb{b}\subset\R^2\setminus\ball_{1/2}(0)$ via a standard mollifier procedure, we conclude the lemma in its entirety.
\end{proof}

\begin{lem}\label{ADN_LinTwoPhaseFlowProblem}
Let $r\in(1,\infty)$. For all 
$\Hone\in\WSR{2-1/r}{r}(\Rtwo)$ and $\Htwo\in\WSR{1-1/r}{r}(\Rtwo)^3$ with 
$\supp\FT_\Rtwo\nb{\Hone}\subset\R^2\setminus\ball_1(0)$,
$\supp\FT_\Rtwo\nb{\Htwo}\subset\R^2\setminus\ball_1(0)$ 
and $\Htwo\cdot\ethree=0$
there exists
a solution $(\uvel,\upres)\in\WSR{2}{r}\np{\Rdot^3}^3\times\WSR{1}{r}\np{\Rdot^3}$ to 
\begin{align}\label{ADN_LinTwoPhaseFlowProblem_Eq}
\begin{pdeq}
\Div \fluidstress(\uvel,\upres) &=0 && \text{in } \Rdot^3 , \\
\Div \uvel &= 0 && \text{in } \Rdot^3, \\
\jump{\uvel}&=0  &&\text{on } \partial\Rdot^3, \\
\uvel \cdot \nvec &=\Hone &&\text{on } \partial\Rdot^3,\\
\projtang \jump{\fluidstress(\uvel,\upres)\nvec}&=\Htwo && \text{on }\partial\Rdot^3
\end{pdeq}
\end{align}
that satisfies 
\begin{align}\label{ADN_LinTwoPhaseFlowProblem_Est}
\norm{\uvel}_{2,r}+\norm{\upres}_{1,r} \leq \Cc{C}\,\np{\norm{\Hone}_{2-1/r,r}+\norm{\Htwo}_{1-1/r,r}},
\end{align}
where $\Cclast{C}=\Cclast{C}(r)$.
Moreover,
$\FT_\Rtwo\bb{\uvel(\cdot,x_3)}(\xip)$ and 
$\FT_\Rtwo\bb{\upres(\cdot,x_3)}(\xip)$
are supported away from  $(\xip,x_3)\in\B_{1/2}(0)\times\R$.
\end{lem}

\begin{proof}
Put
\begin{align}\label{ADN_LinTwoPhaseFlowProblem_DefOfb}
b\coloneqq 
\begin{pmatrix}
b_v \\ b_w
\end{pmatrix}\coloneqq 
\iFT_{\R^2}\Bb{
\begin{pmatrix}
0 & -\frac{1}{2\mu\snorm{\xip}}\bp{\idmatrix-\frac{\xip\otimes\xip}{2\snorm{\xip}^2}} \\
1 & 0 
\end{pmatrix}
\begin{pmatrix}
\ft \Hone \\ \ft \Htwo 
\end{pmatrix}}.
\end{align}
Let $\cutoffkappa_{1/4}\in\CRci\np{\R^2}$ with $\kappa_{1/4}=0$ on $\ball_{1/4}(0)$ and
$\kappa_{1/4}=1$ on $\R^2\setminus\ball_{1/2}(0)$.
Clearly, the truncated operator
\begin{align*}
\calm:\SR\np{\Rtwo}^3\ra\SR\np{\Rtwo}^3,\quad
\calm\np{\phi}\coloneqq \iFT_\Rtwo\Bb{\kappa_{1/4}(\xip)\frac{-1}{2\snorm{\xip}}\Bp{\idmatrix-\frac{\xip\otimes\xip}{2\snorm{\xip}^2}}\ft{\phi}}
\end{align*}
corresponding to the Fourier multiplier appearing in \eqref{ADN_LinTwoPhaseFlowProblem_DefOfb} extends to a bounded operator
$\calm:\WSR{1-1/r}{r}(\Rtwo)^3\ra\WSR{2-1/r}{r}(\Rtwo)^3$. The assumption $\ft{\Htwo}\subset\R^2\setminus\ball_1(0)$
implies that $b_v=\calm\np{\Htwo}$. It follows that $b\in\WSR{2-1/r}{r}(\Rtwo)^3$, and we can therefore introduce the corresponding  solution
$(\uvel,\upres)\in\WSR{2}{r}\np{\Rdot^3}^3\times\WSR{1}{r}\np{\Rdot^3}$
to \eqref{ADN_DirichletStokes_Eq} from Lemma \ref{ADN_DirichletStokes}.
By construction, $\uvel \cdot \nvec =\Hone$ on $\partial\Rdot^3$. Moreover,
recalling \eqref{ADN_DirichletStokes_DefOfuvelupres} we compute 
\begin{align*}
\projtang\jump{\fluidstress(\uvel,\upres)\nvec}
= \iFT_\Rtwo \Bb{\begin{pmatrix}
-2\mu\bp{\snorm{\xip}\idmatrix+\frac{\xip\otimes\xip}{\snorm{\xip}}} \ft b_\vvel \\
0
\end{pmatrix}} = \Htwo.
\end{align*}
Consequently, $\np{\uvel,\upres}$ is a solution to \eqref{ADN_LinTwoPhaseFlowProblem_Eq}.
Employing \eqref{ADN_DirichletStokes_Est} 
we deduce
\begin{align*}
\norm{\uvel}_{2,r}+\norm{\upres}_{1,r} \leq \Cc{c}\,\norm{b}_{2-1/r,r}\leq \Cc{c}\,\np{\norm{\Hone}_{2-1/r,r}+\norm{\Htwo}_{1-1/r,r}}
\end{align*}
and conclude the lemma.
\end{proof}

\begin{thm}\label{ADN_TwoFoldHalfSpace}
Let $r\in(1,\infty)$ and  
\begin{align*}
f\in\LR{r}(\Rthree)^3, \ 
g\in\WSR{1}{r}(\Rdot^3), \
h_0\in\WSR{2-1/r}{r}(\Rtwo)^3, \
\hone\in\WSR{2-1/r}{r}(\Rtwo),\ 
\htwo\in\WSR{1-1/r}{r}(\Rtwo)^3.
\end{align*}
Then all solutions  
$(\uvel,\upres)\in\WSR{2}{r}(\Rdot^3)^3\times\WSR{1}{r}(\Rdot^3)$
to \eqref{SystemStokesTwofoldHalfSpace}
satisfy
\begin{align}\label{ADN_TwoFoldHalfSpace_Est}
\norm{\uvel}_{2,r}+\norm{\upres}_{1,r}
\leq \Cc[ADN_ConstantLpEstimatesDHspace]{C}
\bp{
\norm{f}_{r}+\norm{g}_{1,r}+\norm{h_0}_{2-1/r,r}+\norm{\hone}_{2-1/r,r}+\norm{\htwo}_{1-1/r,r}+\norm{\uvel}_{r}
},
\end{align}
where  $\const{ADN_ConstantLpEstimatesDHspace}
=\const{ADN_ConstantLpEstimatesDHspace}(r,k)>0$.
\end{thm}

\begin{proof}
We decompose both the solution and the data into
one part with lower and another part with higher frequency support in tangential directions $e_1,e_2$. For this purpose, we introduce cut-off functions $\cutoffkappa_\alpha\in\CRci\np{\R^2}$ with $\kappa_\alpha=0$ on $\ball_\alpha(0)$ and
$\kappa_\alpha=1$ on $\R^2\setminus\ball_{2\alpha}(0)$, and put 
\begin{align*}
\uvelhf(x',x_3)\coloneqq \iFT_{\R^2}\bb{\cutoffkappa_1(\xip)\FT_\Rtwo\nb{\uvel(\cdot,x_3)}}(x')\in\WSR{2}{r}(\Rdot^3)^3,\qquad
\uvellf\coloneqq \uvel-\uvelhf.
\end{align*} 
Similarly, we introduce $\upreshf,\upreslf$ and $\fhf,\ghf,\hzerohf,\honehf,\htwohf$.
Observe that $\np{\uvelhf,\upreshf}$ solves \eqref{SystemStokesTwofoldHalfSpace} with respect to data
$\np{\fhf,\ghf,\hzerohf,\honehf,\htwohf}$. We shall construct another solution satisfying estimate \eqref{ADN_TwoFoldHalfSpace_Est}, and subsequently show that it coincides with $\np{\uvelhf,\upreshf}$. To this end, we let
$\ghf^+\in\WSR{1}{r}\np{\R^3}$ denote an extension of $\restriction{\ghf}{\R^3_+}$ to $\WSR{1}{r}\np{\R^3}$.
Specifically employing Heesten's extension operator (see for example \cite[Theorem 4.26]{adams:sobolevspaces})
one readily verifies that the extension retains the property that the Fourier transform (in tangential directions)
$\FT_\Rtwo\bb{\ghf^+(\cdot,x_3)}(\xip)$ is supported away from  $(\xip,x_3)\in\B_1(0)\times\R$. Consequently, 
\begin{align*}
\liftgvel^+\coloneqq \iFT_\Rthree\Bb{\frac{-i\xi}{\snorm{\xi}^2}\FT_\Rthree\bb{\ghf^+}}\in\WSR{2}{r}\np{\Rthree}^3
\end{align*}
is well defined. Similarly, we introduce an extension of $\restriction{\ghf}{\R^3_-}$ to $\WSR{1}{r}\np{\R^3}$ and
construct a field $\liftgvel^-\in\WSR{2}{r}\np{\Rthree}^3$ as above. Letting
\begin{align*}
\liftgvel\coloneqq \begin{pdeq}
&\liftgvel^+ &&\tin\R^3_+,\\
&\liftgvel^- &&\tin\R^3_-,
\end{pdeq}
\end{align*}
we then obtain a field $\liftgvel\in\WSR{2}{r}\np{\Rdot^3}^3$ with $\Div\liftgvel=\ghf$ in $\Rdot^3$. Moreover, a straight-forward application
of Marcinkiewicz's Multiplier Theorem (see for example \cite[Corollary 5.2.5]{Grafakos1}) yields
\begin{align*}
\norm{\liftgvel}_{2,r}\leq\Cc{c}\norm{g}_{1,r}.
\end{align*}
Now put 
\begin{align*}
&\liftfvel \coloneqq  \iFT_\Rthree\Bb{\frac{1}{\snorm{\xi}^2}\Bp{\idmatrix-\frac{\xi\otimes\xi}{\snorm{\xi}^2}}\FT_\Rthree\bb{\fhf-\Div \symmgrad\np{\liftgvel}}},\\
&\liftfpres\coloneqq \iFT_\Rthree\Bb{\frac{\xi}{\snorm{\xi}^2}\cdot\FT_\Rthree\bb{\fhf - \Div \symmgrad\np{\liftgvel}}}.
\end{align*}
Owing to the fact that $\liftgvel,\Div \symmgrad\np{\liftgvel}\in\LR{r}\np{\Rthree}^3$ with
$\FT_\Rthree\nb{\liftgvel}$ and $\FT_\Rthree\nb{\Div \symmgrad\np{\liftgvel}}$ supported away from 0, the expressions above are well defined and yield
functions with $\liftfvel\in\WSR{2}{r}\np{\Rthree}^3$ and $\liftfpres\in\WSR{1}{r}\np\Rthree$ satisfying
\begin{align*}
\begin{pdeq}
\Div \fluidstress(\liftfvel,\liftfpres) &=\fhf - \Div \symmgrad\np{\liftgvel} & \text{in } \Rthree , \\
\Div \liftfvel &= 0 & \text{in } \Rthree.
\end{pdeq}
\end{align*}
Moreover, another straight-forward application of Marcinkiewicz's Multiplier Theorem yields
\begin{align*}
\norm{\liftfvel}_{2,r} + \norm{\liftfpres}_{1,r}
\leq \Cc{c} \bp{\norm{\fhf}_{r}+\norm{\Div \symmgrad\np{\liftgvel}}_r}
\leq \Cc{c} \bp{\norm{\fhf}_{r}+\norm{\ghf}_{1,r}}.
\end{align*}
Utilizing Lemma \ref{ADN_DirichletStokes}, we construct a solution $\np{\lifthzerovel,\lifthzeropres}\in\WSR{2}{r}\np{\Rdot^3}^3\times\WSR{1}{r}\np{\Rdot^3}$ to 
\begin{align*}
\begin{pdeq}
\Div \fluidstress(\lifthzerovel,\lifthzeropres) &=0 && \text{in } \Rdot^3 , \\
\Div \lifthzerovel &= 0 && \text{in } \Rdot^3, \\
\jump{\lifthzerovel}&= \hzerohf-\jump{\liftfvel}-\jump{\liftgvel} &&\text{on } \partial\Rdot^3
\end{pdeq}
\end{align*}
satisfying
\begin{align*}
\norm{\lifthzerovel}_{2,r}+\norm{\lifthzeropres}_{1,r}\leq \Cc{c}\bp{\norm{\hzerohf}_{2-1/r,r}+\norm{\liftfvel}_{2-1/r,r}+\norm{\liftgvel}_{2-1/r,r}}.
\end{align*} 
Finally, by Lemma \ref{ADN_LinTwoPhaseFlowProblem} there is a solution
$\np{\lifthtwotwovel,\lifthonetwopres}\in\WSR{2}{r}\np{\Rdot^3}^3\times\WSR{1}{r}\np{\Rdot^3}$ to
\begin{align*}
\begin{pdeq}
\Div \fluidstress(\lifthtwotwovel,\lifthonetwopres) &=0 && \text{in } \Rdot^3 , \\
\Div \lifthtwotwovel &= 0 && \text{in } \Rdot^3, \\
\jump{\lifthtwotwovel}&=0  &&\text{on } \partial\Rdot^3, \\
\lifthtwotwovel \cdot \nvec &=\honehf-\bp{\lifthzerovel+\liftfvel+\liftgvel}\cdot\nvec &&\text{on } \partial\Rdot^3,\\
\projtang \jump{\fluidstress(\lifthtwotwovel,\lifthonetwopres)\nvec}&=
\htwohf- \projtang \jump{\fluidstress(\lifthzerovel+\liftfvel+\liftgvel,\lifthzeropres+\liftfpres)\nvec}
&& \text{on }\partial\Rdot^3,\\
\end{pdeq}
\end{align*}
which obeys
\begin{align*}
\norm{\lifthtwotwovel}_{2,r}+\norm{\lifthonetwopres}_{1,r}
&\leq \Cc{c}\bp{\norm{\honehf}_{2-1/r,r}+ \norm{\htwohf}_{1-1/r,r}
+\norm{\lifthzerovel}_{2-1/r,r}+\norm{\lifthzeropres}_{1-1/r,r}\\
&\qquad +\norm{\liftfvel}_{2-1/r,r}+\norm{\liftfpres}_{1-1/r,r}+\norm{\liftgvel}_{2-1/r,r}}.
\end{align*}
It follows that
\begin{align*}
\Uvelhf \coloneqq  \lifthtwotwovel+\lifthzerovel+\liftfvel+\liftgvel,\qquad
\Upreshf \coloneqq  \lifthonetwopres+\lifthzeropres+\liftfpres
\end{align*}
is a solution to \eqref{SystemStokesTwofoldHalfSpace} with 
$\np{\fhf,\ghf,\hzerohf,\honehf,\htwohf}$ as the right-hand side,
and that  $\np{\Uvelhf,\Upreshf}\in\WSR{2}{r}\np{\Rdot^3}^3\times\WSR{1}{r}\np{\Rdot^3}$ satisfies
\begin{align}\label{ADN_TwoFoldHalfSpace_EstOfHF}
\norm{\Uvelhf}_{2,r}+\norm{\Upreshf}_{1,r}&\leq \Cc{c}\bp{
\norm{\fhf}_r+\norm{\ghf}_{1,r}+\norm{\hzerohf}_{2-1/r,r}+
\norm{\honehf}_{2-1/r,r}+\norm{\htwohf}_{1-1/r,r}
}.
\end{align}
Consequently, $\np{\Uvelhf,\Upreshf}$ and $\np{\uvelhf,\upreshf}$ solve the same equations. Using a classical duality argument,
we shall show that they coincide. To this end, let $\phi\in\CRci(\R^3)$ and put
$\phihf\coloneqq \iFT_\Rtwo\bb{\cutoffkappa_{1/4}(\xip)\FT_\Rtwo\nb{\phi(\cdot,x_3)}}$.
Employing the same procedure as above, we construct a solution 
$\np{\szvelhf,\szpreshf}\in\WSR{2}{r'}(\Rdot^3)^3\times\WSR{1}{r'}(\Rdot^3)$ to \eqref{SystemStokesTwofoldHalfSpace} with 
right-hand side $\np{\phihf,0,0,0,0}$.
Since by construction both
$\FT_\Rtwo\bb{\Uvelhf(\cdot,x_3)}(\xip)$ and
$\FT_\Rtwo\bb{\uvelhf(\cdot,x_3)}(\xip)$ 
are supported away from  $(\xip,x_3)\in\B_{1/2}(0)\times\R$,
we compute
\begin{align*}
\int_{\Rthree} (\uvelhf-\Uvelhf)\cdot\phi\,\dx &= 
\int_{\Rthree} (\uvelhf-\Uvelhf)\cdot\phihf\,\dx \\
&=\int_{\Rdot^3} (\uvelhf-\Uvelhf)\cdot {\Div \fluidstress(\szvelhf,\szpreshf)}\,\dx\\
&=\int_{\Rdot^3} \Div\fluidstress(\uvelhf-\Uvelhf,\upreshf-\Upreshf)\cdot \szvelhf\,\dx
=0.
\end{align*}
Since $\phi$ can be taken arbitrarily, we obtain $\uvelhf=\Uvelhf$, and
in turn from \eqref{SystemStokesTwofoldHalfSpace} also $\upreshf=\Upreshf$.
It follows that also $\np{\uvelhf,\upreshf}$ satisfies \eqref{ADN_TwoFoldHalfSpace_EstOfHF} and thus
\begin{align}\label{ADN_TwoFoldHalfSpaceEstOfHFPart}
\norm{\uvelhf}_{2,r}+\norm{\upreshf}_{1,r}&\leq \Cc{c}\bp{
\norm{f}_r+\norm{g}_{1,r}+\norm{h_0}_{2-1/r,r}+\norm{\hone}_{2-1/r,r}+\norm{\htwo}_{1-1/r,r}
}.
\end{align}
Finally, from 
$\FT_\Rtwo\bb{\uvellf(\cdot,x_3)}(\xip)\subset\ball_1(0)\times\R$ and
$\FT_\Rtwo\bb{\upreslf(\cdot,x_3)}(\xip)\subset\ball_1(0)\times\R$
it follows via the Marcinkiewicz Multiplier Theorem
that
$\norm{\grad_{x'}\grad\uvellf}_r+\norm{\grad_{x'}\upreslf}\leq \Cc{c}\,\norm{\uvellf}_{r}$.
Introducing the decomposition $\uvel=\uvelhf+\uvellf$ in \eqref{SystemStokesTwofoldHalfSpace} and isolating $\partial_3\uvelthreelf$
on the left-hand side in \eqrefsub{SystemStokesTwofoldHalfSpace}{2}, we then infer after differentiation that
$\norm{\partial_3^2\uvelthreelf}_r\leq \Cc{c}\norm{\uvellf}_{r}$. Subsequently isolating $\partial_3\upreslf$ in 
the third coordinate equation of \eqrefsub{SystemStokesTwofoldHalfSpace}{1}, we deduce $\norm{\partial_3\upreslf}_{1,r}\leq\Cc{c}\norm{\uvellf}_{r}$. Lastly isolating
$\partial_3^2\uvelonelf$ and $\partial_3^2\uveltwolf$ in the first and second coordinate equation of \eqrefsub{SystemStokesTwofoldHalfSpace}{1}, respectively, we further deduce
$\norm{\partial_3^2\uvelonelf}_r+\norm{\partial_3^2\uveltwolf}_r\leq \Cc{c}\norm{\uvellf}_{r}$.
In conclusion, 
\begin{align}\label{ADN_TwoFoldHalfSpaceEstOfLFPart}
\norm{\uvellf}_{2,r}+\norm{\grad\upreslf}_r\leq\Cc{c}\norm{\uvellf}_{r}.
\end{align}
Combining \eqref{ADN_TwoFoldHalfSpaceEstOfHFPart} and \eqref{ADN_TwoFoldHalfSpaceEstOfLFPart} we conclude \eqref{ADN_TwoFoldHalfSpace_Est} and thus the theorem.
\end{proof}

\subsection{A priori estimates for strong solutions}

We return to the linearized two-phase-flow Navier--Stokes problem \eqref{SystemOseenOmegaGeneral}, where 
$\Omega$ is a domain of the same type as in Section \ref{Preliminaries_Section}, \ie,
satisfying \eqref{Preliminaries_DomainTypeDef}. Based on the estimates obtained in the twofold-half-space case
in Theorem \ref{ADN_TwoFoldHalfSpace}, we shall establish 
$\LR{r}$ estimates 
of solutions to \eqref{SystemOseenOmegaGeneral}.
The Oseen case $(\Rey\neq 0)$ and  Stokes case $(\Rey= 0)$ are treated separately in Theorem \ref{ThmOseenOmegaHom} and
Theorem \ref{ThmStokesProblemOmega}, respectively.

\begin{thm}\label{ThmOseenOmegaHom}
Let $\Gamma$ be a $\CR{5}$-smooth surface, $q\in(1,\frac{3}{2})$, $r\in(3,\infty)$ and $\reybar>0$. 
For every $0<\Rey\leq\reybar$ and 
$\np{f,g,\hone,\htwo}\in\Yspace_1\times\Yspace_{2,3}\times\Yspace_4$
there exists a unique solution 
$(\uvel,\upres)\in\XreyOne\times\Xspace_2$
to \eqref{SystemOseenOmegaGeneral}
satisfying 
\begin{align}\label{EqNormalizePressure}
\int_{\Omega^\inner} \upres^\inner \,\dx=0.
\end{align}
Moreover,
\begin{align}\label{EstimateOseenOmegaHom}
\norm{\uvel}_{\XreyOne}+\norm{\upres}_{\Xspace_2}
\leq \Cc[ConstOseenOmegaHom]{C}
\norm{(f,g,\hone,\htwo)}_{\Yspace_1\times\Yspace_{2,3}\times\Yspace_4},
\end{align}
where 
$\const{ConstOseenOmegaHom}=\const{ConstOseenOmegaHom}(\Omega,q,r,\reybar)>0$.
\end{thm}

\begin{proof}
We first consider data
\begin{align}\label{ThmOseenOmegaHom_RegularData}
\begin{aligned}
&\np{f,g,\hone,\htwo}\in \CRi(\Omega)\times\CRi\np{\Omega}\times\CR{5}\np{\Gamma}\times\CR{5}\np{\Gamma},\\
&\supp f \text{ and } \supp g \text{ compact in }\R^3,\\
&\int_{\Omega^\inner} g\,\dx = \int_\Gamma \hone\,\dS,
\end{aligned}
\end{align}
so that the theorems from Section \ref{WeakSolutions} can be applied.
Recalling the regularity of $\Gamma$, Theorem \ref{ThmOseenWeakSolutions}, Theorem \ref{PropAssPressure} and Theorem \ref{ThmHigherRegularityWeakSolutions}
yield a solution $\np{\uvel,\upres}\in\DSRN{1}{2}\np{\R^3}^3\times\LRN{2}\np{\R^3}$ to \eqref{SystemOseenOmegaGeneral} satisfying
\begin{align*}
\uvel\in\bigcap_{\ell=0}^2\WSRhom{\ell+2}{2}(\Omega),\qquad 
\upres\in\bigcap_{\ell=0}^2\WSRhom{\ell+1}{2}(\Omega).
\end{align*}
We fix an $R>\delta(\Omega)$ and observe that  
$\np{\uvel,\upres}\in\WSR{2}{r}\np{\Omega_{2R}}^3\times\WSR{1}{r}\np{\Omega_{2R}}$ by Sobolev embedding.
According to the regularity assumptions, $\Gamma$ can be covered by a finite number of
balls $\Gamma\subset\bigcup_{i=1}^m \ballrixi$ each
of which upon a rotation $\rotmatrixi$ can be mapped to $\ballri$ by a $\CR{5}$-diffeomorphism $\Phi_i$, that is,  $\Phi_i\circ\rotmatrixi:\ballrixi\ra\ballri$, in such a way
that $\Phi_i\circ\rotmatrixi\bp{\Gamma\cap\ballrixi}=\setc{x\in\ballri}{x_3=0}$
and with $\norm{\grad\Phi_i}_\infty$ arbitrarily small for sufficiently small radii $r_i$, $i=1,\ldots,m$. The covering can clearly be augmented
with bounded open sets $O_1\subset\subset\Omega^\inner$ and $O_2\subset\subset\Omega^\outer_{2R}$ so that
$\overline{\Omega_R}\subset\cup_{i=1}^m\ballrixi\cup O_1 \cup O_2$. Employing a partition of unity subordinate to such a covering,
we can decompose and transform the solution $\np{\uvel,\upres}$ into $m$ solutions $(\uvel_i,\upres_i)\in\WSR{2}{r}(\Rdot^3)^3\times\WSR{1}{r}(\Rdot^3)$, $i=1,\ldots,m$,
to the twofold half-space Stokes problem \eqref{SystemStokesTwofoldHalfSpace},
two solutions $\np{\uvel_{m+1},\upres_{m+1}},\np{\uvel_{m+2},\upres_{m+2}}\in\WSR{2}{r}(\R^3)^3\times\WSR{1}{r}(\R^3)$ to a whole-space Stokes problem, and finally one solution
$\np{\wvel,\wpres}\in\DSRN{1}{2}\np{\R^3}^3\times\LRN{2}\np{\R^3}$ to the whole-space Oseen problem
\begin{align}\label{ThmOseenOmegaHom_WholeSpaceOseen}
\begin{pdeq}
-\Div \fluidstress(\wvel,\wpres) + \Rey \partial_3 \wvel&=F && \text{in } \R^3, \\
\Div \wvel &= G && \text{in } \R^3.
\end{pdeq}
\end{align}
In all three cases, the data contain
lower-order terms of $\uvel$ and $\upres$ supported in $\ball_{2R}$. Furthermore, the  data in the twofold half-space Stokes equations satisfied by $(\uvel_i,\upres_i)$, $i=1,\ldots,m$, also 
contain higher-order terms of $\uvel$ and $\upres$ supported in $\ball_{2R}$ and multiplied with components of $\grad\Phi_i$.  
By Sobolev embeddings, we have $\wvel\in\DSRN{1}{2}\np{\R^3}\embeds\LR{6}\np{\Rthree}$, and it is therefore easy to verify, for example by applying the Fourier transform in \eqref{ThmOseenOmegaHom_WholeSpaceOseen},
that $(\wvel,\wpres)$ coincides with the solution from \cite[Theorem VII.4.1]{GaldiBookNew} and therefore satisfies
\begin{align}\label{ThmOseenOmegaHom_WholeSpaceOseenEst}
\begin{aligned}
\norm{\wvel}_\XreyOne+ \norm{\wpres}_{\Xspace_2}
&\leq \Cc{c}\bp{\norm{F}_{\LR{q}\np{\Rthree}\cap\LR{r}\np{\Rthree}}+\norm{G}_{\DSR{1}{q}(\Rthree)\cap\LR{{3q}/{3-q}}(\Rthree)\cap\DSR{1}{r}(\Rthree)}}\\
&\leq \Cc{c}\bp{\norm{f}_{\Yspace_1}+\norm{g}_{\Yspace_2}
+\norm{\uvel}_{\WSR{1}{r}(\Omega_{2R})} + \norm{\upres}_{\LR{r}(\Omega_{2R})}
}
\end{aligned}
\end{align}
with a constant $\Cclast{c}=\Cclast{c}(q,r,\reybar)$ independent of $\Rey$.
A similar estimate is satisfied by the solutions $\np{\uvel_{m+1},\upres_{m+1}}$ and $\np{\uvel_{m+2},\upres_{m+2}}$ to the whole-space Stokes problems by \cite[Theorem IV.2.1]{GaldiBookNew}.
Moreover, Theorem \ref{ADN_TwoFoldHalfSpace} implies that $(\uvel_i,\upres_i)$, $i=1,\ldots,m$, also satisfies the estimate, provided
a covering is chosen with $\norm{\grad\Phi_i}_\infty$ 
sufficiently small in relation to $\const{ADN_ConstantLpEstimatesDHspace}$ so that the higher-order terms can be absorbed on the left-hand side.
We thus conclude
\begin{align}\label{ThmOseenOmegaHom_EstWithLowerOrderTerms}
\norm{\uvel}_{\XreyOne}+\norm{\upres}_{\Xspace_2}
\leq
\Cc{c}\bp{\norm{(f,g,\hone,\htwo)}_{\Yspace_1\times\Yspace_{2,3}\times\Yspace_4}+\norm{\uvel}_{\WSR{1}{r}(\Omega_{2R})} + \norm{\upres}_{\LR{r}(\Omega_{2R})}},
\end{align}
where 
$\Cclast{c}=\Cclast{c}(\Gamma,q,r,\reybar)>0$.
It remains to show that the lower-order terms of $\uvel$ and $\upres$ on the right-hand side can be neglected.
This can be achieved by a standard contradiction argument. Assuming that
\begin{multline}\label{ThmOseenOmegaHom_EstOfLowerOrderTerms}
\exists c>0\ \forall 0<\snorm{\Rey}<\reybar\ \forall \text{solutions } (\uvel,\upres)\in\XreyOne\times\Xspace_2
\text{w.r.t. data }\eqref{ThmOseenOmegaHom_RegularData}:\\
\norm{\uvel}_{\WSR{1}{r}(\Omega_{2R})} + \norm{\upres}_{\LR{r}(\Omega_{2R})} \leq c \norm{(f,g,\hone,\htwo)}_{\Yspace_1\times\Yspace_{2,3}\times\Yspace_4}
\end{multline}
does \emph{not} hold, one can utilize \eqref{ThmOseenOmegaHom_EstWithLowerOrderTerms} to construct a sequence $(\rey_n,\uvel_n,\upres_n)$ normalized such that
$\norm{\uvel_n}_{\WSR{1}{r}(\Omega_{2R})} + \norm{\upres_n}_{\LR{r}(\Omega_{2R})}=1$
and with $\rey_n\ra\rey$
and $(\uvel_n,\upres_n)$ weakly convergent in the Banach space 
\begin{align*}
\DSR{2}{r}(\Omega)\cap\DSR{2}{q}(\Omega)\cap\LR{\frac{3q}{3-2q}}(\Omega)\times\DSR{1}{q}(\Omega)\cap\LR{\frac{3q}{3-q}}(\Omega)
\end{align*}
to a solution $(\uvel,\upres)$ to \eqref{SystemOseenOmegaGeneral} with parameter $\rey\in[0,\reybar]$ and homogeneous right-hand side.
The restriction $q<\frac{3}{2}$ is critical in this step.
Theorem \ref{PropUniquenessWeakSolutionGeneral} implies $\np{\uvel,\upres}=(0,0)$, contradicting 
$\norm{\uvel}_{\WSR{1}{r}(\Omega_{2R})} + \norm{\upres}_{\LR{r}(\Omega_{2R})}=1$ obtained
due to the compactness of the embeddings 
$\DSR{2}{r}(\Omega)\cap\LR{\frac{3q}{3-2q}}(\Omega)\embeds\WSR{1}{r}(\Omega_{2R})$
and 
$\DSR{1}{r}(\Omega)\cap\LR{\frac{3q}{3-q}}(\Omega)\embeds\LR{r}(\Omega_{2R})$.
We conclude \eqref{ThmOseenOmegaHom_EstOfLowerOrderTerms}. Therefore,
the lower-order terms of $\uvel$ and $\upres$ on the right-hand side in \eqref{ThmOseenOmegaHom_EstWithLowerOrderTerms} can be neglected, which yields
\eqref{EstimateOseenOmegaHom}. Uniqueness of the solution follows from Theorem \ref{PropUniquenessWeakSolutionGeneral}, and the theorem is thereby established for data satisfying \eqref{ThmOseenOmegaHom_RegularData}. However, it is easy to verify that
data satisfying \eqref{ThmOseenOmegaHom_RegularData} are dense in the space $\Yspace_1\times\Yspace_{2,3}\times\Yspace_4$. Consequently, the general
case follows by a density argument.
\end{proof}

\begin{thm}\label{ThmStokesProblemOmega}
Let $\Gamma$ be a $\CR{5}$-smooth closed surface, $q\in(1,\frac{3}{2})$, $r\in(3,\infty)$ and $\Rey=0$. 
For every
$\np{f,g,\hone,\htwo}\in\Yspace_1\times\Yspace_{2,3}\times\Yspace_4$
there exists a unique solution 
$(\uvel,\upres)$
to \eqref{SystemOseenOmegaGeneral}
with
\begin{align}\label{StokesOmegaFunctionClass}
\begin{split}
\uvel^\inner\in\WSR{2}{r}(\Omega^\inner)^3,&\quad
\uvel^\outer\in\bp{\WSRhom{2}{q}(\Omega^\outer)\cap\WSRhom{2}{r}(\Omega^\outer)
\cap\WSRhom{1}{\frac{3q}{3-q}}(\Omega^\outer)\cap\LR{\frac{3q}{3-2q}}(\Omega^\outer)}^3,\\
\upres^\inner\in\WSR{1}{r}(\Omega^\inner),&\quad
\upres^\outer\in\WSRhom{1}{q}(\Omega^\outer)\cap\WSRhom{1}{r}(\Omega^\outer)\cap\LR{\frac{3q}{3-q}}(\Omega^\outer),
\end{split}
\end{align} 
that satisfies \eqref{EqNormalizePressure} and
\begin{align}\label{EstimateStokesOmega}
\norm{\grad^2\uvel}_q
+\norm{\grad^2\uvel}_r
+\norm{\grad\uvel}_{\frac{3q}{3-q}}
+\norm{\uvel}_{\frac{3q}{3-2q}}
+\norm{\grad\upres}_q
+\norm{\grad\upres}_r
\leq\Cc[ConstStokesOmega]{C}
\norm{(f,g,\hone,\htwo)}_{\Yspace_1\times\Yspace_{2,3}\times\Yspace_4},
\end{align}
where $\const{ConstStokesOmega}=\const{ConstStokesOmega}(q,r,\Omega)>0$.
\end{thm}

\begin{proof}
The proof is similar to that of Theorem \ref{ThmOseenOmegaHom}, the only difference being 
that $\Rey=0$ in \eqref{ThmOseenOmegaHom_WholeSpaceOseen}. This implies that $\np{\wvel,\wpres}$ solves a whole-space Stokes problem instead of an Oseen problem. Therefore, we use \cite[Theorem IV.2.1]{GaldiBookNew} in this case to obtain estimate \eqref{EstimateStokesOmega}.
The rest of the proof is identical to that of Theorem \ref{ThmOseenOmegaHom}.
\end{proof}

\section{Reformulation on a fixed domain}\label{ReformulationFixedDomainSection}

The steady-state equations of motion
as expressed in \eqref{derivation_sseq_final} in a frame attached to the barycenter 
of the falling drop form a classical free boundary problem. Specifically, 
the boundary $\Gamma$ 
depends on the unknown height function $\height$.
For further analysis it is necessary to refer all unknowns in this so-called \textit{current configuration} to a fixed domain \textit{reference configuration}.
This section is devoted to such a reformulation.

As mentioned in the introduction and further elaborated on in Section \ref{EqOfMotionSection}, we investigate a falling drop whose stress-free configuration, \ie,
the configuration when the density in the two liquids is the same,
is the unit ball $\ball_1$ in non-dimensionalized coordinates. Our aim is to establish existence of steady-state configurations
close to the stress-free configuration $\ball_1$ for small density differences. Canonically, we therefore choose 
\begin{align*}
\refdomain \coloneqq   \R^3\setminus\sphere^2
\end{align*}
as the fixed liquid reference domain.

In order to refer the equations of motion to $\refdomain$, we first construct a suitable coordinate transformation $\ctrafo$ 
based on the height function $\height$. For technical reasons, it is important that $\ctrafo$ retains any rotational symmetry possessed by $\height$.

\begin{lem}\label{ConstrucntionOfCoordinateTrafo}
Let $r\in(3,\infty)$.
There is an extension operator
\begin{align*}
\extopr: \WSR{3-1/r}{r}(\sphere^2)\to\WSR{3}{r}(\R^3\setminus\sphere^2)^3
\end{align*}
satisfying $\trace_{\sphere^2} E(\height)=\eta\,\id$, $\supp\extopr\np{\height}\subset\ball_4$ and
\begin{align}\label{ConstrucntionOfCoordinateTrafo_ExtOprEst}
\norm{\extopr(\height)}_{\WSR{3}{r}}
\leq \Cc{C} \norm{\height}_{\WSR{3-1/r}{r}}.
\end{align}
The extension operator is invariant with respect to rotations, that is, for all $\rotmatrix\in\sorthomatrixspace{3}$:
\begin{align}\label{ConstrucntionOfCoordinateTrafo_RotSym}
\extopr\bp{\height\np{R\,\cdot}}(x) = \rotmatrix^\transpose\extopr(\height)(\rotmatrix x).
\end{align}
If $r>3$, there is a $\Cc[heightmin]{delta}>0$ such that for any $\eta\in\WSR{3-1/r}{r}(\sphere^2)$  with $\norm{\eta}_{\WSR{3-1/r}{r}}<\const{heightmin}$ the mapping 
\begin{align*}
\ctrafo:\R^3\ra\R^3,\quad \ctrafo(x)=x+E(\height)(x)
\end{align*}
is continuous and maps $\refdomain$ $\CR{2}$-diffeomorphically onto $\Omega=\Omega_\height$ with
\begin{align*}
\ctrafo(\sphere^2)=\Gamma_\height, \quad 
\ctrafo (\ball_1)=\Omega_\height^\inner, \quad 
\ctrafo(\ball^1)=\Omega_\height^\outer.
\end{align*}
\end{lem}

\begin{proof}
For  $\height\in\WSR{3-1/r}{r}(\sphere^2)$ 
let $H_\height\in\WSR{3}{r}(\ball_4\setminus\sphere^2)$ denote the unique solution  to 
\begin{align}\label{Reformulation_DirichletLaplaceBall}
\begin{pdeq}
\Delta H_\height&=0 &&\tin\ball_1, \\
H_\height&=\height &&\ton\sphere^2,
\end{pdeq}\qquad\qquad  
\begin{pdeq}
\Delta H_\height&=0 &&\tin\ball_4\setminus\overline{\ball_1}, \\
H_\height&=\height &&\ton\sphere^2, \\
H_\height&=0 &&\ton\partial\ball_4.
\end{pdeq}
\end{align}
Since the Laplace operator is rotational invariant, also the solution $H_\height$ is invariant with respect to rotations of the data $\eta$.  
Let $\cutoff\in\CRci(\R)$ be a cut-off function with $\cutoff(s)=1$ for $\snorm{s}\leq 2$ and $\cutoff(s)=1$ for $\snorm{s}\geq 3$. Putting
\begin{align*}
\extopr(\height)(x) \coloneqq  \cutoff(\snorm{x})\,H_\height(x)\,x,
\end{align*}
we obtain an operator with the desired properties. Observe that $\extopr\np{\height}\in\WSR{1}{r}\np{\R^3}$. 
Therefore, $\ctrafo(x)\coloneqq x+E(\height)(x)$ is a well-defined pointwise mapping $\ctrafo\colon\R^3\to\R^3$.
Since $r>3$, the Sobolev embedding $\WSR{3}{r}(\R^3\setminus\sphere^2)\embeds\CR{2}(\R^3\setminus\sphere^2)$
implies that $\ctrafo\in\CR{2}(\refdomain)$. Moreover, by \eqref{ConstrucntionOfCoordinateTrafo_ExtOprEst} we clearly have $\det\grad\ctrafo=\det\bp{\idmatrix+\grad\extopr(\height)}>0$
when $\norm{\height}_{\WSR{3-1/r}{r}(\sphere^2)}$ is sufficiently small. In this case, $\ctrafo$ is a $\CR{2}$-diffeomorphism onto its image $\Omega$ by the
global inverse function theorem of Hadamard.
\end{proof}

We shall use $\ctrafo$ to change the coordinates and consequently express \eqref{derivation_sseq_final} in the reference configuration $\refdomain$.
To this end, we set
\begin{align}\label{EqFieldsTransformedOnFixedDomain}
\wvel\coloneqq\vvel\circ\ctrafo,\qquad
\wpres\coloneqq\vpres\circ\ctrafo.
\end{align}
In order to simplify the notation, we put
\begin{align}
\gradctrafo&\coloneqq \grad \ctrafo = \idmatrix + \grad E(\height),\label{EqDefCtrafoF}\\
\detctrafo & \coloneqq \det \gradctrafo 
= 1 + \Div E(\height) + \sum_{i=1}^3\prod_{\stackrel{j=1}{j\neq i}}^3 \partial_j E(\height) + \det(\grad E(\height)), \label{EqDefCtrafoDet} \\
\cofctrafo & \coloneqq (\cof \gradctrafo )^\transpose 
=\bp{1+\Div E(\height)}\idmatrix - \grad E(\height)+\cof(\grad E(\height))^\transpose, \label{EqDefCtrafoCof}
\end{align}
and introduce the transformed stress tensor
\begin{align}\label{fluidstresstrafoDef}
\fluidstresstrafo(\wvel,\wpres)
\coloneqq \bb{\mu \np{\grad\wvel\gradctrafo^{-1}+\gradctrafo^{-\transpose}\grad\wvel^\transpose} - \wpres \idmatrix} \cofctrafo^\transpose
=\bp{\fluidstress(\vvel,\vpres)\circ\ctrafo}\cofctrafo^\transpose.
\end{align}
Observe that an application of the Piola identity yields
\begin{align*}
\Div \fluidstresstrafo(\wvel,\wpres) = \detctrafo \bp{\Div\fluidstress(\vvel,\vpres)}\circ\ctrafo
\quad\tand\quad
\Div(\cofctrafo\wvel)=\detctrafo\np{\Div\wvel}\circ\ctrafo.
\end{align*}
The normal vector $\nvec_\Gamma$ at $\Gamma$ expressed in the coordinates of the reference configuration is given by
\[
\nvec_\Gamma \circ \ctrafo = \frac{\cofctrafo^\transpose \nvec_{\sphere^2}}{\snorm{\cofctrafo^\transpose \nvec_{\sphere^2}}},
\]
and the transformed tangential projection by
\begin{align*}
\projtangtrafo
\coloneqq
\idmatrix - \snorm{\cofctrafo^\transpose \nvec_{\sphere^2}}^{-2} 
\cofctrafo^\transpose \np{\nvec_{\sphere^2}\otimes\nvec_{\sphere^2}}\cofctrafo
= (\idmatrix-\nvec_{\Gamma}\otimes\nvec_\Gamma)\circ \ctrafo.
\end{align*}
With this notation, the steady-state equations of motion \eqref{derivation_sseq_final} take the following form in the reference configuration:
\begin{align}\label{SystemFluidReferenceConfigurationPreSimplification}
\begin{pdeq}
\rho\bp{\np{\cofctrafo \wvel} \cdot \grad \wvel + \rey \grad\wvel\cofctrafo e_3}
&=\Div \fluidstresstrafo(\wvel,\wpres)  
&& \text{in } \Omega_0 , 
\\
\Div \np{\cofctrafo \wvel} &= 0 && \text{in } \Omega_0, 
\\
\jump{\wvel}&=0  &&\text{on } \sphere^2, 
\\
\detctrafo\wvel \cdot 
\frac{\cofctrafo^\transpose \nvec_{\sphere^2}}{\snorm{\cofctrafo^\transpose \nvec_{\sphere^2}}}
&=
-\detctrafo\rey e_3 \cdot \frac{\cofctrafo^\transpose\nvec_{\sphere^2}}{\snorm{\cofctrafo^\transpose \nvec_{\sphere^2}}} 
&&\text{on } \sphere^2, 
\\
\cofctrafo\,\projtangtrafo \jump{\fluidstresstrafo(\wvel,\wpres)\nvec_{\sphere^2}}
&=0 
&& \text{on }\sphere^2, 
\\
\frac{\cofctrafo^\transpose\nvec_{\sphere^2}}{\snorm{\cofctrafo^\transpose\nvec_{\sphere^2}}^2}
\cdot\jump{\fluidstresstrafo(\wvel,\wpres)\nvec_{\sphere^2}} 
&=
\frac{1}{16\pi}\frac{\cofctrafo^\transpose\nvec_{\sphere^2}}{\snorm{\cofctrafo^\transpose\nvec_{\sphere^2}}}
\cdot \int_{\sphere^2}\zeta\bb{\np{1+\height(\zeta)}^4-1}\,\dS \\
&\qquad+\sigma \bp{\meancurv+2}\circ\ctrafo
+\rhod (1+\height) e_3 \cdot \nvec_{\sphere^2}
&& \text{on } \sphere^2,
\\
\int_{\sphere^2} \jump{\fluidstresstrafo(\wvel,\wpres)\nvec_{\sphere^2}}\, \snorm{\cofctrafo^\transpose\nvec_{\sphere^2}}^{-1}\detctrafo\,\dS 
&= \rhod \frac{4\pi}{3} e_3, 
\\
\int_{\sphere^2}\bb{\np{1+\height}^3-1} \,\dS
&= 0,
\\
\lim_{\snorm{x}\to\infty}\wvel(x) & = 0
\end{pdeq}
\end{align}
with respect to unknowns $\np{\wvel,\wpres,\rey,\height}$. 
We use the notation $\nvec=\nvec_{\sphere^2}$ in the following.

In the next step, we exploit an inherent symmetry in \eqref{SystemFluidReferenceConfigurationPreSimplification} 
and simplify the system by replacing 
\eqrefsub{SystemFluidReferenceConfigurationPreSimplification}{7} 
with 
\begin{align*}
e_3\cdot\int_{\sphere^2} \jump{\fluidstresstrafo(\wvel,\wpres)\nvec}\, \snorm{\cofctrafo^\transpose\nvec}^{-1}\detctrafo\,\dS 
= \rhod \frac{4\pi}{3}.
\end{align*}
We shall \textit{a posteriori} verify that a solution to the simplified system exhibits axial symmetry around $e_3$ and 
consequently satisfies 
\begin{align*}
e_j\cdot\int_{\sphere^2} \jump{\fluidstresstrafo(\wvel,\wpres)\nvec}\, \snorm{\cofctrafo^\transpose\nvec}^{-1}\detctrafo\,\dS 
= 0\quad\text{for }j=1,2.
\end{align*}
Consequently, a solution to the simplified system 
\begin{align}\label{SystemFluidReferenceConfiguration}
\begin{pdeq}
\rho\bp{\np{\cofctrafo \wvel} \cdot \grad \wvel + \rey \grad\wvel\cofctrafo e_3}
&=\Div \fluidstresstrafo(\wvel,\wpres)  
&& \text{in } \Omega_0 , 
\\
\Div \np{\cofctrafo \wvel} &= 0 && \text{in } \Omega_0, 
\\
\jump{\wvel}&=0  &&\text{on } \sphere^2, 
\\
\detctrafo\wvel \cdot 
\frac{\cofctrafo^\transpose \nvec_{\sphere^2}}{\snorm{\cofctrafo^\transpose \nvec_{\sphere^2}}}
&=
-\detctrafo\rey e_3 \cdot \frac{\cofctrafo^\transpose\nvec_{\sphere^2}}{\snorm{\cofctrafo^\transpose \nvec_{\sphere^2}}} 
&&\text{on } \sphere^2, 
\\
\cofctrafo\,\projtangtrafo \jump{\fluidstresstrafo(\wvel,\wpres)\nvec_{\sphere^2}}
&=0 
&& \text{on }\sphere^2, 
\\
\frac{\cofctrafo^\transpose\nvec_{\sphere^2}}{\snorm{\cofctrafo^\transpose\nvec_{\sphere^2}}^2}
\cdot\jump{\fluidstresstrafo(\wvel,\wpres)\nvec_{\sphere^2}} 
&=
\frac{1}{16\pi}\frac{\cofctrafo^\transpose\nvec_{\sphere^2}}{\snorm{\cofctrafo^\transpose\nvec_{\sphere^2}}}
\cdot \int_{\sphere^2}\zeta\bb{\np{1+\height(\zeta)}^4-1}\,\dS \\
&\qquad+\sigma \bp{\meancurv+2}\circ\ctrafo
+\rhod (1+\height) e_3 \cdot \nvec_{\sphere^2}
&& \text{on } \sphere^2,\\
e_3\cdot\int_{\sphere^2} \jump{\fluidstresstrafo(\wvel,\wpres)\nvec}\, \snorm{\cofctrafo^\transpose\nvec}^{-1}\detctrafo\,\dS 
&= \rhod \frac{4\pi}{3},\\
\int_{\sphere^2}\bb{\np{1+\height}^3-1} \,\dS
&= 0,\\
\lim_{\snorm{x}\to\infty}\wvel(x) &= 0
\end{pdeq}
\end{align}
with unknowns $\np{\wvel,\wpres,\rey,\height}$ is also a solution to \eqref{SystemFluidReferenceConfigurationPreSimplification}.
The analysis in the remaining part of the article is carried out on the system \eqref{SystemFluidReferenceConfiguration}.

\section{Linearization}\label{LinearizationSection}

A main challenge is to identify a suitable linearization of \eqref{SystemFluidReferenceConfiguration}
such that the fully nonlinear system can be solved via a perturbation technique. Indeed, as explained in the introduction, the trivial linearization obtained by neglecting
all nonlinear terms is not suitable since it leads to a Stokes-type rather than an Oseen-type problem. 
Instead, we shall linearize the equations around a non-trivial first-order approximation.

In order to identify the first-order approximation, we
utilize an idea going back to \textsc{Happel} and \textsc{Brenner} \cite{happelbrenner} and introduce as \emph{auxiliary field} a solution to the system
\begin{align}\label{SystemFluidAux}
\begin{pdeq}
\Div \fluidstress(\Uvel,\Upres) &=0 && \text{in } \Omega_0, \\
\Div \Uvel &= 0 && \text{in } \Omega_0, \\
\jump{\Uvel}&=0  &&\text{on } \sphere^2, \\
\Uvel \cdot \nvec &=- e_3 \cdot \nvec &&\text{on } \sphere^2, \\
\projtang \jump{\fluidstress(\Uvel,\Upres)\nvec}&=0 && \text{on }\sphere^2,\\
\lim_{\snorm{x}\ra\infty}\Uvel(x)&=0.&&
\end{pdeq}
\end{align}
By Theorem \ref{ThmStokesProblemOmega},  a solution $\np{\Uvel,\Upres}$ 
to \eqref{SystemFluidAux} exists with
\begin{align}\label{UvelUpresIntegrability}
\begin{aligned}
&\forall s\in(3,\infty]:\quad \Uvel\in\LR{s}\np{\refdomain},\\
&\forall s\in\big(\tfrac{3}{2},\infty\big]:\quad \grad\Uvel,\,\Upres\in\LR{s}\np{\refdomain},\\
&\forall s\in(1,\infty):\quad \grad^2\Uvel,\,\grad\Upres\in\LR{s}\np{\refdomain}.
\end{aligned}
\end{align}
Moreover, standard regularity theory for the Stokes problem
implies that both $\Uvel$ and $\Upres$ are smooth in $\Omega_0$, and well-known decay estimates for the 3D exterior domain Stokes problem (see for example \cite[Theorem V.3.2]{GaldiBookNew})
yield
\begin{align}\label{UVelUpresPointwiseDecayEst}
\begin{aligned}
\Uvel=\bigo\np{\snorm{x}^{-1}},\quad
\grad\Uvel=\bigo\np{\snorm{x}^{-2}}\quad\tand\quad
\Upres=\bigo\np{\snorm{x}^{-2}}
\quad\tas \snorm{x}\ra\infty.
\end{aligned}
\end{align}
Additionally, both the Stokes operator and the boundary operator on the left-hand side of \eqref{SystemFluidAux}
are invariant with respect to rotations. Since the data on the right-hand side is clearly invariant with respect to
rotations $\rotmatrix\in\sorthomatrixspace{3}$ leaving $e_3$ invariant, the solution $(\Uvel,\Upres)$ retains this symmetry:
\begin{align}\label{AuxFieldSymmetry}
\forall \rotmatrix\in\sorthomatrixspace{3},\ \rotmatrix e_3=e_3:\quad R^\transpose \Uvel(Rx)=\Uvel(x),\ \Upres(Rx)=\Upres(x).
\end{align}
By adding a constant to $\Upres^\inner$, that is, replacing $\Upres$ with
\begin{align*}
\widetilde{\Upres}\coloneqq 
\begin{pdeq}
&\Upres + C &&\tin\B_1,\\
&\Upres &&\tin\B^1,
\end{pdeq}
\end{align*}
we may assume, by choosing the constant $C$ appropriately, that
\begin{align}\label{Linearization_VanishingNormalTotalForce}
\int_{\sphere^2} \nvec\cdot \jump{\fluidstress(\Uvel,\Upres)\nvec}\,\dS = 0.
\end{align}
Moreover, we utilize \eqrefsub{SystemFluidAux}{5} to compute 
\begin{align}\label{Linearization_PositivityCondTotalForce}
\begin{aligned}
-e_3 \cdot \int_{\sphere^2}\jump{\fluidstress(\Uvel,\Upres)\nvec} \,\dS 
&= -\int_{\sphere^2} \bp{e_3 \cdot \nvec} \, \nvec \cdot \jump{\fluidstress(\Uvel,\Upres)\nvec} \,\dS \\
&= \int_{\sphere^2} \bp{\Uvel \cdot \nvec} \, \nvec \cdot \jump{\fluidstress(\Uvel,\Upres)\nvec}\,\dS 
= \int_{\sphere^2} \jump{\Uvel \cdot \fluidstress(\Uvel,\Upres)\nvec}\,\dS \\
&= \int_{\Omega_0} \grad\Uvel : \fluidstress(\Uvel,\Upres) + \Uvel\cdot \Div\fluidstress(\Uvel,\Upres) \,\dx \\ 
&=\int_{\Omega_0} 2\mu\,\snorml{\symmgrad(\Uvel)}^2 \,\dx>0.
\end{aligned}
\end{align}
We can therefore choose 
\begin{align}\label{EquationReyAux}
\Reyrhod \coloneqq  \Bp{e_3 \cdot \int_{\sphere^2}\jump{\fluidstress(\Uvel,\Upres)\nvec} \,\dS}^{-1}\rhod \frac{4\pi}{3}.
\end{align}
This choice of $\Reyrhod$ combined with the fact that the symmetry \eqref{AuxFieldSymmetry} implies
\begin{align*}
e_j\cdot \int_{\sphere^2}\jump{\fluidstress(\Uvel,\Upres)\nvec} \,\dS = 0 \quad (j=1,2)
\end{align*}
means that
$(\Reyrhod\Uvel,\Reyrhod\Upres,\Reyrhod,0)$ is a solution to the trivial linearization of \eqref{SystemFluidReferenceConfiguration} around the zero state,
that is, to the system obtained by neglecting in \eqref{SystemFluidReferenceConfiguration} all nonlinear terms with respect to $\np{\wvel,\wpres,\rey,\height}$.
The state $(\Reyrhod\Uvel,\Reyrhod\Upres,\Reyrhod,0)$ can therefore be seen as a first-order approximation of the solution to \eqref{SystemFluidReferenceConfiguration}.

We shall seek to linearize \eqref{SystemFluidReferenceConfiguration} around $(\Reyrhod\Uvel,\Reyrhod\Upres,\Reyrhod,0)$. Since $\rhod\neq 0$ implies
$\Reyrhod\neq 0$, a linearization around $(\Reyrhod\Uvel,\Reyrhod\Upres,\Reyrhod,0)$ would result in an Oseen-type problem. However, a direct linearization
around $(\Reyrhod\Uvel,\Reyrhod\Upres,\Reyrhod,0)$ is still precarious since $(\Uvel,\Upres)$ is a solution to a Stokes problem, whence a linearization around
this state would bring about right-hand side terms inadmissible in an Oseen setting. Instead, we introduce a truncation of the state. 
More specifically, 
we let $\cutoff\in\CRci(\R)$ be a cut-off function with 
$\cutoff\np{r}=1$ for $\snorm{x} \leq 1$ and $\cutoff\np{r}=0$ for $\snorm{r} \geq 2$, and define
$\cutoff_R\in\CRci(\R^3)$ by $\cutoff_R(x)\coloneqq \cutoff\bp{R^{-1}\snorm{x}}$ for $R>4$.
Via the truncated auxiliary fields
\begin{align}\label{AuxfieldsDef}
\Uvel_R\coloneqq \cutoff_R \Uvel, \quad \Upres_R\coloneqq \cutoff_R \Upres,
\end{align}
we finally obtain the state $(\Reyrhod\Uvel_R,\Reyrhod\Upres_R,\Reyrhod,0)$ around which we shall linearize the system \eqref{SystemFluidReferenceConfiguration}.
Specifically, we let
\begin{align}\label{EqLinearizationAtAuxField}
\reyd\coloneqq\rey-\Reyrhod,\quad\uvel \coloneqq \wvel-\Reyrhod\Uvel_R-\reyd\Uvel_R, \quad \upres\coloneqq\wpres-\Reyrhod\Upres_R-\reyd\Upres_R
\end{align}
and investigate \eqref{SystemFluidReferenceConfiguration} with respect to the unknowns $(\uvel,\upres,\kappa,\height)$.

To conclude the linearization, we 
express the mean curvature $\meancurv$ on $\Gamma$ as a function of $\height$. As in \cite[Section 2.2.5]{PruessSimonett_MovingInterfacesParabolicEvolutionEquations}, we
obtain
\begin{align*}
\meancurv\circ\ctrafo=
\frac{1}{1+\height} \Bp{\frac{\LaplaceBeltrami \height}{\sqrt{g}}
+\grad_{\sphere^2}\frac{1}{\sqrt{g}}\cdot\grad_{\sphere^2}\height
-\frac{2(1+\height)}{\sqrt{g}}},
\end{align*}
where $\LaplaceBeltrami$ and $\grad_{\sphere^2}$ denote the Laplace--Beltrami operator and the surface gradient on the unit sphere $\sphere^2$, respectively, and 
\begin{align*}
g \coloneqq (1+\height)^2+\snorm{\grad_{\sphere^2}\height}^2.
\end{align*}
Then we have
\begin{align*}
(\meancurv+2)\circ\ctrafo
=\LaplaceBeltrami \height + 2 \height -\calg_\meancurv(\height)
\end{align*}
with  
\begin{align*}
\calg_\meancurv(\height)
\coloneqq  -\frac{1}{1+\height} \frac{1-(1+\height)\sqrt{g}}{\sqrt{g}}\LaplaceBeltrami\height
-\frac{1}{1+\height}\grad_{\sphere^2}\frac{1}{\sqrt{g}}\cdot\grad_{\sphere^2}\height
+\frac{2-2(1-\height)\sqrt{g}}{\sqrt{g}}
\end{align*}
containing all the nonlinear terms.

We are now in a position to express \eqref{SystemFluidReferenceConfiguration} as a suitable perturbation of a linear problem with respect to the
unknowns $(\uvel,\upres,\kappa,\height)$.
Indeed, in a setting of velocity fields satisfying $\jump{\uvel}=0$ and $\lim_{\snorm{x}\ra\infty}\uvel(x)=0$ 
we can express \eqref{SystemFluidReferenceConfiguration} equivalently as 
\begin{align}\label{SystemAbstractFormulation}
\loprrhod\np{\uvel,\upres,\reyd,\height}=\nopr(\uvel,\upres,\reyd,\height),
\end{align}
where the linear operator $\loprrhod$ is given by
\begin{align}\label{LinearOperatorAbstract}
\begin{aligned}
\loprrhod(\uvel,\upres,\reyd,\height)&\coloneqq 
\begin{pmatrix}
-\Div \fluidstress(\uvel,\upres) + \rho \Reyrhod\,\partial_3\uvel
\\
\Div \uvel 
\\
\uvel\cdot \nvec 
\\
\projtang \jump{\fluidstress(\uvel,\upres)\nvec}
\\
\reyd e_3\cdot\int_{\sphere^2} \jump{\fluidstress(\Uvel,\Upres)\nvec}\,\dS 
+ e_3\cdot\int_{\sphere^2} \jump{\fluidstress(\uvel,\upres)\nvec}\,\dS
\\
\int_{\sphere^2} \height\,\dS
\\
\sigma\np{\LaplaceBeltrami+2}\height
+\frac{1}{4\pi}\nvec\cdot\int_{\sphere^2}\height\nvec \,\dS
-\reyd\nvec\cdot\jump{\fluidstress(\Uvel,\Upres)\nvec}
-\nvec\cdot\jump{\fluidstress(\uvel,\upres)\nvec}
\end{pmatrix}\\
&=:\begin{pmatrix}
\loprcomp{1-4}(\uvel,\upres) \\
\loprcomp{5}(\uvel,\upres,\reyd) \\
\loprcomp{6}(\height) \\
\loprcomp{7}(\uvel,\upres,\reyd,\height) \\
\end{pmatrix}
\end{aligned}
\end{align}
and the nonlinear operator $\nopr=\np{\noprcomp_1,\ldots,\noprcomp_7}$ consists of the components
\begin{align*}
\noprcomp_1(\uvel,\upres,\reyd,\height)
&\coloneqq  (\Reyrhod+\reyd)\Div\fluidstresstrafo(\Uvel_R,\Upres_R) 
+ \Div\fluidstresstrafo(\uvel,\upres)-\Div\fluidstress(\uvel,\upres)
- \rho\cofctrafo\uvel\cdot\grad\uvel \\
&\quad 
- \rho(\Reyrhod+\reyd)\bp{\cofctrafo\Uvel_R \cdot \grad\uvel + \cofctrafo\uvel\cdot\grad\Uvel_R}
- \rho(\Reyrhod+\reyd)^2\cofctrafo\Uvel_R\cdot\grad\Uvel_R\\
&\quad 
- \rho\reyd\grad\uvel\cofctrafo e_3 
- \rho\Reyrhod\grad\uvel(\cofctrafo-\idmatrix)e_3
- \rho(\Reyrhod+\reyd)^2\grad\Uvel_R\cofctrafo e_3, \\
\noprcomp_2(\uvel,\upres,\reyd,\height)
&\coloneqq \Div(\np{\idmatrix-\cofctrafo}\uvel)
-(\Reyrhod+\reyd)\Div(\cofctrafo\Uvel_R), \\
\noprcomp_3(\uvel,\upres,\reyd,\height)
&\coloneqq  \bp{\uvel+(\Reyrhod+\reyd)\np{\Uvel+e_3}}
\cdot\bp{\idmatrix-\detctrafo \snorm{\cofctrafo^\transpose\nvec}^{-1}\cofctrafo^\transpose}\nvec, \\
\noprcomp_4(\uvel,\upres,\reyd,\height)
&\coloneqq \projtangtrafozero \jump{\fluidstress(\uvel,\upres)\nvec}
-\cofctrafo\,\projtangtrafo \jump{\fluidstresstrafo(\uvel,\upres)\nvec}
-\np{\Reyrhod+\reyd}\cofctrafo\,\projtangtrafo \jump{\fluidstresstrafo(\Uvel,\Upres)\nvec}, \\
\noprcomp_5(\uvel,\upres,\reyd,\height)
&\coloneqq (\Reyrhod+\reyd)e_3\cdot\int_{\sphere^2} \bp{\jump{\fluidstress(\Uvel,\Upres)\nvec}-\jump{\fluidstresstrafo(\Uvel,\Upres)\nvec}\snorm{\cofctrafo^\transpose\nvec}\detctrafo}\,\dS \\
&\quad
+ e_3\cdot\int_{\sphere^2} \bp{\jump{\fluidstress(\uvel,\upres)\nvec}-\jump{\fluidstresstrafo(\uvel,\upres)\nvec}\snorm{\cofctrafo^\transpose\nvec}\detctrafo}\,\dS, \\
\noprcomp_6(\uvel,\upres,\reyd,\height) 
&\coloneqq -\int_{\sphere^2}\height^2+\frac{1}{3}\height^3 \,\dS,\\
\noprcomp_7(\uvel,\upres,\reyd,\height)
&\coloneqq \frac{\cofctrafo^\transpose\nvec}{\snorm{\cofctrafo^\transpose\nvec}^2}
\cdot\jump{\fluidstresstrafo(\uvel,\upres)\nvec}
-\nvec\cdot\jump{\fluidstress(\uvel,\upres)\nvec} 
+ \Reyrhod\frac{\cofctrafo^\transpose\nvec}{\snorm{\cofctrafo^\transpose\nvec}^2}
\cdot\jump{\fluidstresstrafo(\Uvel,\Upres)\nvec} \\*
&\quad
+\reyd\Bp{\frac{\cofctrafo^\transpose\nvec}{\snorm{\cofctrafo^\transpose\nvec}^2}
\cdot\jump{\fluidstresstrafo(\Uvel,\Upres)\nvec}
-\nvec\cdot\jump{\fluidstress(\Uvel,\Upres)\nvec}} \\*
&\quad
-\frac{1}{4\pi}\frac{\cofctrafo^\transpose\nvec}{\snorm{\cofctrafo^\transpose\nvec}}
\cdot \int_{\sphere^2}\bp{\frac{3}{2}\height^2+\height^3+\frac{1}{4}\height^4}\,\nvec\,\dS
+\frac{1}{4\pi}\bp{\nvec-\frac{\cofctrafo^\transpose\nvec}{\snorm{\cofctrafo^\transpose\nvec}}}
\cdot\int_{\sphere^2} \height\,\nvec \,\dS \\*
&\quad
-\rhod (1+\height) e_3 \cdot \nvec + \sigma\calg_\meancurv(\height).
\end{align*}

\section{Main Theorems}\label{ExistenceSteadySolutionSection}

The formulation \eqref{SystemAbstractFormulation} is compatible with the framework of function spaces introduced in Section \ref{Preliminaries_Section}. More specifically,
we shall show that $\lopr$ maps $\Xrey(\refdomain)$ homeomorphically onto $\Yspace(\refdomain)$, and
a solution to the fully nonlinear problem  \eqref{SystemAbstractFormulation} can
be established via the contraction mapping principle. We start with the first assertion:

\begin{thm}\label{ThmLinearOperatorHom}
Let $q\in(1,\frac{3}{2})$, $r\in(3,\infty)$ and $0<\snorm{\Rey}\leq\reybar$.
Then 
\begin{align*}
\lopr:\Xrey(\refdomain)\ra \Yspace(\refdomain)
\end{align*}
is a homeomorphism with
$\norm{\loprinv}\leq \Cc[ConstantHomeomorphism]{C}$ and $\const{ConstantHomeomorphism}=\const{ConstantHomeomorphism}(q,r,\reybar)$ \emph{independent} of $\Rey$.
\end{thm}

\begin{proof}
We first show that $\lopr$ is onto. To this end, we consider $\np{f,g,h_1,h_2,a_1,a_2,h_3}\in\Yspace(\refdomain)$ and establish existence of
$(\uvel, \tupres, \reyd,\height)\in\Xrey(\refdomain)$ such that $\lopr(\uvel, \tupres, \reyd,\height)=\np{f,g,h_1,h_2,a_1,a_2,h_3}$.
By Theorem \ref{ThmOseenOmegaHom} there is a solution $(\uvel,\upres)\in\XreyOne(\refdomain)\times\Xspace_2(\refdomain)$ to \eqref{SystemOseenOmegaGeneral} with $\Omega=\Omega_0$.
We put 
\begin{align}\label{ThmLinearOperatorHom_presconstant}
c_\upres \coloneqq  \frac{1}{\snorm{\sphere^2}} \Bp{ 2\sigma a_2 - \int_{\sphere^2} \nvec\cdot\jump{\fluidstress(\uvel,\upres)\nvec} - \int_{\sphere^2} h_3\,\dS }
\end{align}
and replace $\upres$ with
\[
\tupres\coloneqq 
\begin{pdeq}
&\upres + c_\upres &&\tin\B_1,\\
&\upres &&\tin\B^1.
\end{pdeq}
\]
Then $(\uvel,\tupres)$ still solves \eqref{SystemOseenOmegaGeneral}, whence
\begin{align}\label{ThmLinearOperatorHom_Leq1-4}
\loprcomp{1-4}(\uvel,\tupres) = (f,g,h_1,h_2).
\end{align}
Recalling \eqref{Linearization_PositivityCondTotalForce}, we can define
\begin{align}\label{ThmLinearOperatorHom_DefOfReyD}
\reyd\coloneqq  \Bp{e_3 \cdot \int_{\sphere^2}\jump{\fluidstress(\Uvel,\Upres)\nvec} \,\dS}^{-1}\Bp{a_1 - e_3 \cdot \int_{\sphere^2}\jump{\fluidstress(\uvel,\tupres)\nvec} \,\dS}
\end{align}
and thus obtain
\begin{align}\label{ThmLinearOperatorHom_Leq5}
\loprcomp{5}(\uvel,\tupres,\reyd)= a_1.
\end{align}
It remains to solve $\loprcomp{6}(\height)=a_2$ and $\loprcomp{7}(\uvel,\upres,\reyd,\height)=h_3$ with respect to $\height$. We briefly recall some properties of the operator $\Delta_\sphere+2$.
In particular, it is Fredholm in the setting $\Delta_\sphere+2:\WSR{3-1/r}{r}\np{\sphere^2}\ra\WSR{1-1/r}{r}\np{\sphere^2}$
(see for example \cite[Theorem 7.4.3]{Triebel_TheoryOfFunctionSpaces2}). It is well known, and easy to verify by a direct computation, that the components of the outer normal $\nvec$ on $\sphere^2$
span its kernel, that is, $\ker\np{\Delta_\sphere+2}=\vecspan\set{\nvec_1,\nvec_2,\nvec_3}$. We denote the projection onto this kernel and the corresponding complementary projection by 
\begin{align*}
\proj\psi\coloneqq \frac{1}{4\pi}\nvec\cdot\int_{\sphere^2}\psi\nvec\,\dS \quad\tand \quad \projcompl\coloneqq \id-\proj.
\end{align*}
The self-adjoint nature of $\Delta_\sphere+2$ implies that $\proj$ is also a projection onto the kernel of its adjoint $\np{\Delta_\sphere+2}^*$. The Fredholm property thus implies
that 
\begin{align}\label{ThmLinearOperatorHom_DeltaPlusTwoHomeomorphism}
\Delta_\sphere+2:\projcompl\WSR{3-1/r}{r}\np{\sphere^2}\ra\projcompl\WSR{1-1/r}{r}\np{\sphere^2} \quad\text{homeomorphically}.
\end{align}
We can therefore introduce
\begin{align*}
&\height_\parallel \coloneqq \proj\bp{h_3 +\reyd\nvec\cdot\jump{\fluidstress(\Uvel,\Upres)\nvec} +\nvec\cdot\jump{\fluidstress(\uvel,\tupres)\nvec}},\\
&\height_\bot \coloneqq 
\sigma^{-1}\np{\Delta_\sphere+2}^{-1}\,\projcompl\bp{h_3 +\reyd\nvec\cdot\jump{\fluidstress(\Uvel,\Upres)\nvec} +\nvec\cdot\jump{\fluidstress(\uvel,\tupres)\nvec}},
\end{align*}
and obtain a solution $\height\coloneqq \height_\parallel+\height_\bot \in\WSR{3-1/r}{r}\np{\sphere^2}$ to
\begin{align}\label{ThmLinearOperatorHom_Leq7}
\loprcomp{7}(\uvel,\tupres,\reyd,\height)=h_3.
\end{align}
Moreover, integrating \eqref{ThmLinearOperatorHom_Leq7} over $\sphere^2$ and recalling both the choice of $c_\upres$ in \eqref{ThmLinearOperatorHom_presconstant} and \eqref{Linearization_VanishingNormalTotalForce}, we observe that
\begin{align}\label{ThmLinearOperatorHom_Leq6}
\loprcomp{6}(\height)=a_2.
\end{align}
From \eqref{ThmLinearOperatorHom_Leq1-4}, \eqref{ThmLinearOperatorHom_Leq5}, \eqref{ThmLinearOperatorHom_Leq6} and \eqref{ThmLinearOperatorHom_Leq7} we deduce $\lopr(\uvel, \tupres, \reyd,\height)=\np{f,g,h_1,h_2,a_1,a_2,h_3}$ and consequently that $\lopr$ is onto.
Uniqueness of the solution $(\uvel, \tupres, \reyd,\height)$ is a direct consequence of Theorem \ref{ThmOseenOmegaHom} and \eqref{ThmLinearOperatorHom_DeltaPlusTwoHomeomorphism}, which
means that $\lopr$ is also injective. The operator is clearly continuous and therefore a homeomorphism.
Furthermore, from Theorem \ref{ThmOseenOmegaHom} we deduce the estimate
\begin{align*}
\norm{\np{\uvel,\tupres}}_{\XreyOne\times\Xspace_2}
&\leq \Cc{c}
\bp{\norm{(f,g,\hone,\htwo)}_{\Yspace_1\times\Yspace_{2,3}\times\Yspace_4}+ \snorm{c_\upres}  }\\
&\leq \Cc{c}
\bp{\norm{(f,g,\hone,\htwo)}_{\Yspace_1\times\Yspace_{2,3}\times\Yspace_4}+\norm{a_2}_{\Yspace_6}+ \norm{h_3}_{\Yspace_7}  }
\end{align*}
with $\Cclast{c}=\Cclast{c}(q,r,\reybar)$ \emph{independent} of $\Rey$. In turn, we estimate in \eqref{ThmLinearOperatorHom_DefOfReyD}
\begin{align*}
\snorm{\reyd} = \norm{\reyd}_{\Xspace_3} 
&\leq \Cc{c}
\bp{\norm{(f,g,\hone,\htwo)}_{\Yspace_1\times\Yspace_{2,3}\times\Yspace_4}+\norm{a_1}_{\Yspace_5}+\norm{a_2}_{\Yspace_6}+ \norm{h_3}_{\Yspace_7}  }
\end{align*}
with $\Cclast{c}=\Cclast{c}(q,r,\reybar)$ \emph{independent} of $\Rey$. Since additionally
\begin{align*}
\norm{\height}_{\Xspace_4}
&\leq
\norm{\height_\parallel}_{\WSR{3-1/r}{r}}+
\norm{\height_\bot}_{\WSR{3-1/r}{r}}\\
&\leq \Cc{c}\bp{
\norm{\proj\bp{h_3 +\reyd\nvec\cdot\jump{\fluidstress(\Uvel,\Upres)\nvec} +\nvec\cdot\jump{\fluidstress(\uvel,\tupres)\nvec}}}_{\WSR{3-1/r}{r}}}\\
&\qquad+\norm{\np{\Delta_\sphere+2}^{-1}}\norm{\projcompl\bp{h_3 +\reyd\nvec\cdot\jump{\fluidstress(\Uvel,\Upres)\nvec} +\nvec\cdot\jump{\fluidstress(\uvel,\tupres)\nvec}}}_{\WSR{1-1/r}{r}}\\
&\leq \Cc{c}\bp{
\norm{h_3}_{\WSR{1-1/r}{r}} + \snorm{\reyd} + \norm{\jump{\fluidstress(\uvel,\tupres)\nvec}}_{\WSR{1-1/r}{r}}
+\snormL{\int_{\sphere^2}h_3\,\dS} + \snormL{\int_{\sphere^2}\nvec\cdot\jump{\fluidstress(\uvel,\tupres)\nvec}\,\dS}
}\\
&\leq \Cc{c}\bp{
\norm{h_3}_{\WSR{1-1/r}{r}} + \snorm{\reyd} + \norm{\np{\uvel,\tupres}}_{\XreyOne\times\Xspace_2}
},
\end{align*}
we conclude
\begin{align*}
\norm{(\uvel, \tupres, \reyd,\height)}_{\Xrey(\refdomain)}\leq \Cc{c}\norm{\np{f,g,h_1,h_2,a_1,a_2,h_3}}_{\Yspace(\refdomain)}
\end{align*}
with $\Cclast{c}=\Cclast{c}(q,r,\reybar)$ \emph{independent} of $\Rey$. It follows that
$\norm{\loprinv}\leq \Cc{c}$ with $\Cclast{c}=\Cclast{c}(q,r,\reybar)$ \emph{independent} of $\Rey$.
\end{proof}

The proof that the composition $\loprinv\circ\nopr$ is a contraction is prepared in the following
two lemmas.
We first establish estimates of the change-of-coordinate matrices.

\begin{lem}\label{LemmaEstimatesCtrafo}
Let $r\in(3,\infty)$. There is $\Cc[heightminforinverse]{delta}>0$ such that
for all 
$\height_1, \height_2\in\WSR{3-1/r}{r}(\sphere^2)$ with $\norm{\height_j}_{\WSR{3-1/r}{r}}\leq\const{heightminforinverse}$ $(j=1,2)$ the
following estimates are valid:
\begin{align*}
&\norm{\idmatrix-\cofctrafoone}_{\WSR{1}{\infty}}
\leq 
\Cc[EstimatesCtrafo]{C}\norm{\height_1}_{\WSR{3-1/r}{r}},
&&\norm{\cofctrafoone-\cofctrafotwo}_{\WSR{1}{\infty}}
\leq 
\const{EstimatesCtrafo}\norm{\height_1-\height_2}_{\WSR{3-1/r}{r}},\\
&\norm{\idmatrix-\gradctrafoone^{-1}}_{\WSR{1}{\infty}}
\leq 
\const{EstimatesCtrafo}\norm{\height_1}_{\WSR{3-1/r}{r}}, 
&&\norm{\gradctrafoone^{-1}-\gradctrafotwo^{-1}}_{\WSR{1}{\infty}}
\leq 
\const{EstimatesCtrafo}\norm{\height_1-\height_2}_{\WSR{3-1/r}{r}},\\
&\norm{1-\detctrafoone}_{\WSR{1}{\infty}}
\leq 
\const{EstimatesCtrafo}\norm{\height_1}_{\WSR{3-1/r}{r}}, 
&&\norm{\detctrafoone-\detctrafotwo}_{\WSR{1}{\infty}}
\leq 
\const{EstimatesCtrafo}\norm{\height_1-\height_2}_{\WSR{3-1/r}{r}}
\end{align*}
where $\const{EstimatesCtrafo}=\const{EstimatesCtrafo}(\const{heightminforinverse},r)$.
\end{lem}

\begin{proof}
Recalling \eqref{EqDefCtrafoCof}, we observe that $\idmatrix-\cofctrafoone$ contains only
terms of first and second order with respect to components of $\grad E(\height_1)$.
Utilizing that $\WSR{2}{r}(\R^3\setminus\sphere^2)$ is an algebra for $r>3$, and the Sobolev embedding
$\WSR{2}{r}(\R^3\setminus\sphere^2)\embeds\WSR{1}{\infty}(\R^3\setminus\sphere^2)$,
we deduce
\begin{align*}
\norm{\idmatrix-\cofctrafoone}_{\WSR{1}{\infty}}
\leq 
\Cc{c}\norm{\idmatrix-\cofctrafoone}_{\WSR{2}{r}}
\leq
\Cc{c} \bp{1+\norm{\grad E(\height_1)}_{\WSR{2}{r}}} \norm{\grad E(\height_1)}_{\WSR{2}{r}}.
\end{align*}
The first assertion of the lemma then 
follows from \eqref{ConstrucntionOfCoordinateTrafo_ExtOprEst} in Lemma
\ref{ConstrucntionOfCoordinateTrafo}. The next assertions follows in a similar manner.
Concerning the estimates involving $\gradctrafoone^{-1}$, we recall from \eqref{EqDefCtrafoF}--\eqref{EqDefCtrafoCof} that $\gradctrafoone^{-1}=\detctrafoone^{-1}\cofctrafoone$.
Consequently, we obtain an estimate of $\norm{\idmatrix-\gradctrafoone^{-1}}_{\WSR{1}{\infty}}$ as above, provided $\detctrafoone$ is bounded away from $0$. To this end, we recall
\eqref{EqDefCtrafoDet} and choose $\const{heightminforinverse}$ so small that $\detctrafoone>\half$ for $\norm{\height_1}_{\WSR{3-1/r}{r}}\leq\const{heightminforinverse}$.
One may now verify the rest of the assertions analogously. 
\end{proof}

The linearization \eqref{SystemAbstractFormulation} is a result of expressing the velocity field and pressure term
as a perturbation \eqref{EqLinearizationAtAuxField} around a truncated auxiliary field $\np{\Uvel_R,\Upres_R}$.
The truncation is necessary to avoid right-hand side terms in \eqref{SystemAbstractFormulation} with inadmissible decay properties.
Instead, compactly supported right-hand side terms appear. Suitable estimates of these terms are established in the following lemma.
In particular, the magnitude of their norms are estimated in terms of the distance $R$ of the truncation $\cutoff_R$ from the drop domain: 

\begin{lem}\label{LemmaAuxFieldTruncatedEstimates}
Let $q\in\bp{1,\frac{3}{2}}$, $r\in(3,\infty)$ and $\const{heightminforinverse}$ be the constant from Lemma \ref{LemmaEstimatesCtrafo}. For all
$\height_1, \height_2\in\WSR{3-1/r}{r}(\sphere^2)$ with $\norm{\height_j}_{\WSR{3-1/r}{r}}\leq\const{heightminforinverse}$ $(j=1,2)$ 
\begin{align}\label{LemmaAuxFieldTruncatedEstimates_Ests}
\begin{aligned}
&\norm{\Div \fluidstresstrafoOne (\Uvel_R,\Upres_R)}_{\Yspace_1}\leq\Cc[LemmaAuxFieldTruncatedEstimatesConst]{C}(R^{-3+3/q}+\norm{\height_1}_{\WSR{3-1/r}{r}}),\\
&\norm{\Div \fluidstresstrafoOne (\Uvel_R,\Upres_R)-\Div \fluidstresstrafoTwo (\Uvel_R,\Upres_R)}_{\Yspace_1}\leq\const{LemmaAuxFieldTruncatedEstimatesConst}\norm{\height_1-\height_2}_{\WSR{3-1/r}{r}},\\
&\norm{\Div \np{\cofctrafoone\Uvel_R}}_{\Yspace_2}\leq\const{LemmaAuxFieldTruncatedEstimatesConst}(R^{-3+3/q}+\norm{\height_1}_{\WSR{3-1/r}{r}}),\\
&\norm{\Div \np{\cofctrafoone\Uvel_R}-\Div \np{\cofctrafotwo\Uvel_R}}_{\Yspace_2}\leq\const{LemmaAuxFieldTruncatedEstimatesConst}\norm{\height_1-\height_2}_{\WSR{3-1/r}{r}},
\end{aligned}
\end{align}
where $\const{LemmaAuxFieldTruncatedEstimatesConst}=\const{LemmaAuxFieldTruncatedEstimatesConst}(q,r,\const{heightminforinverse})$. Moreover,
\begin{align}\label{LemmaAuxFieldTruncatedEstimates_EstDriftterm}
\begin{aligned}
&\norm{\cofctrafoone\Uvel_R\cdot\grad\Uvel_R}_{\Yspace_1}\leq\Cc[LemmaAuxFieldTruncatedEstimatesConst2]{C} ,\\
&\norm{\cofctrafoone\Uvel_R\cdot\grad\Uvel_R-\cofctrafotwo\Uvel_R\cdot\grad\Uvel_R}_{\Yspace_1}\leq\const{LemmaAuxFieldTruncatedEstimatesConst2}\norm{\height_1-\height_2}_{\WSR{3-1/r}{r}},\\
&\norm{\grad\Uvel_R\cofctrafoone e_3}_{\Yspace_1}\leq\const{LemmaAuxFieldTruncatedEstimatesConst2},\\
&\norm{\grad\Uvel_R\cofctrafoone e_3-\grad\Uvel_R\cofctrafotwo e_3}_{\Yspace_1}\leq\const{LemmaAuxFieldTruncatedEstimatesConst2}\norm{\height_1-\height_2}_{\WSR{3-1/r}{r}},
\end{aligned}
\end{align}
where $\const{LemmaAuxFieldTruncatedEstimatesConst2}=\const{LemmaAuxFieldTruncatedEstimatesConst2}(q,r,\const{heightminforinverse})$.
\end{lem}

\begin{proof}
Let $\height=\height_1$. Recalling from Lemma \ref{ConstrucntionOfCoordinateTrafo} that $\cofctrafo(x)=\gradctrafo(x)=\idmatrix$ for $\snorm{x}\geq4$, we utilize Lemma \ref{LemmaEstimatesCtrafo} to estimate
\begin{align*}
&\norm{\Div \fluidstresstrafoOne (\Uvel_R,\Upres_R)}_{\Yspace_1}\\
&\qquad\leq\norm{\Div \fluidstresstrafoOne (\Uvel_R,\Upres_R) - \Div \fluidstress (\Uvel_R,\Upres_R)}_{\LR{q}\cap\LR{r}} + \norm{\Div \fluidstress (\Uvel_R,\Upres_R)}_{\LR{q}\cap\LR{r}}\\
&\qquad\leq \norm{{\mu \np{\grad\Uvel_R\gradctrafo^{-1}\cofctrafo^\transpose- \grad\Uvel_R}  + \mu\np{\gradctrafo^{-\transpose}\grad\Uvel_R^\transpose\cofctrafo^\transpose-\grad\Uvel_R^\transpose }} 
- \np{\Upres_R \cofctrafo^\transpose - \Upres_R \idmatrix}}_{\WSRhom{1}{q}(\ball_4)\cap\WSRhom{1}{r}(\ball_4)}\\
&\qquad \quad+\norm{\Div \fluidstress (\Uvel_R,\Upres_R)}_{\LR{q}\cap\LR{r}}\\
&\qquad\leq \Cc{c} \np{\norm{\height}^2_{\WSR{3-1/r}{r}}+\norm{\height}_{\WSR{3-1/r}{r}}}
\np{\norm{\grad\Uvel_R}_{\WSR{1}{q}(\ball_4)\cap\WSR{1}{r}(\ball_4)}+\norm{\Upres_R}_{\WSR{1}{q}(\ball_4)\cap\WSR{1}{r}(\ball_4)}\\
&\qquad \quad+\norm{\Div \fluidstress (\Uvel_R,\Upres_R)}_{\LR{q}\cap\LR{r}}}.
\end{align*}
Recalling the truncation \eqref{AuxfieldsDef}, the pointwise decay of the auxiliary fields \eqref{UVelUpresPointwiseDecayEst}, and that $\supp\grad\cutoff_R\subset\ball_{2R,R}$ with  $\snorm{\grad\cutoff_R(x)}\leq \Cc{c}R^{-1}$ as well as $\snorm{\grad^2\cutoff_R(x)}\leq \Cc{c}R^{-2}$, 
we further obtain
\begin{align*}
\norm{\Div \fluidstress(\Uvel_R,\Upres_R)}_q &=
\norm{\Div \bp{\mu\np{\grad\nb{\cutoff_R\Uvel}+\grad\nb{\cutoff_R\Uvel}^\transpose}-\cutoff_R\Upres}}_q\\
&\leq
\Cc{c}\bp{\norm{R^{-2}\Uvel}_{\LR{q}\np{\ball_{2R,R}}}+
\norm{R^{-1}\grad\Uvel}_{\LR{q}\np{\ball_{2R,R}}}+
\norm{R^{-1}\Upres}_{\LR{q}\np{\ball_{2R,R}}}}\\
&\leq
\Cc{c}R^{-3+3/q}.
\end{align*}
Since $r>q$, we obtain an even better estimate for $\norm{\Div \fluidstress(\Uvel_R,\Upres_R)}_r$ with respect to decay in $R$, and thus
conclude the first assertion of the lemma. The other inequalities in \eqref{LemmaAuxFieldTruncatedEstimates_Ests} follow in a similar manner. 

The most critical estimate in \eqref{LemmaAuxFieldTruncatedEstimates_EstDriftterm} is the second one.
Employing Lemma \ref{LemmaEstimatesCtrafo} together with the integrability properties \eqref{UvelUpresIntegrability} 
and the pointwise decay \eqref{UVelUpresPointwiseDecayEst} of the auxiliary fields,
we conclude
\begin{align*}
&\norm{\cofctrafoone\Uvel_R\cdot\grad\Uvel_R -\cofctrafotwo\Uvel_R\cdot\grad\Uvel_R}_{\Yspace_1}
\leq\norm{\cofctrafoone-\cofctrafotwo}_{\infty}
\norm{\cutoff_R\np{\Uvel\cdot\grad\cutoff_R}\Uvel+\cutoff_R^2\Uvel\cdot\grad\Uvel}_{\LR{q}\cap\LR{r}} \\
&\quad\leq\Cc{c}\norm{\height_1-\height_2}_{\WSR{3-1/r}{r}}\bp{
\norm{R^{-1}\snorm{\Uvel}^2}_{\LR{q}(\ball_{2R,R})\cap\LR{r}(\ball_{2R,R})}
+\norm{\Uvel}_{\LR{3q}\cap\LR{3r}}\norm{\grad\Uvel}_{\LR{3q/2}\cap\LR{3r/2}}
}\\
&\quad\leq\Cc{c}\norm{\height_1-\height_2}_{\WSR{3-1/r}{r}}\bp{R^{-3+1/q}+R^{-3+1/r}+\Cc{c}}
\leq\Cc{c}\norm{\height_1-\height_2}_{\WSR{3-1/r}{r}}
\end{align*}
since $R>4$.
The remaining estimates in \eqref{LemmaAuxFieldTruncatedEstimates_EstDriftterm} are verified in a similar fashion.
\end{proof}

We are now in a position to show existence of a solution to \eqref{SystemAbstractFormulation}. 

\begin{thm}\label{ThmExistenceNonLinearFixedDomain}
Let $q\in\big(1,\frac{4}{3}\big]$, $r\in(3,\infty)$ and $\frac{3}{4}<\ppot<1$. There is an $\epsilon>0$ such that for all $0<\snorm{\rhod}<\epsilon$ there is an $R>0$ and a solution $(\uvel,\upres,\reyd,\height)\in\Xreyrhod(\Omega_0)$ to
\begin{align}\label{ThmExistenceNonLinearFixedDomain_Eq}
\loprrhod\np{\uvel,\upres,\reyd,\height}=\nopr(\uvel,\upres,\reyd,\height),
\end{align}
which satisfies 
\begin{align}\label{ThmExistenceNonLinearFixedDomain_Est}
\norm{(\uvel,\upres,\reyd,\height)}_{\Xrey} 
\leq \snorm{\rhod}^{\ppot}.
\end{align}
This solution is unique in the class of elements in $\Xreyrhod(\Omega_0)$ 
satisfying \eqref{ThmExistenceNonLinearFixedDomain_Est}.
\end{thm}

\begin{proof}
We let $R\coloneqq R(\rhod)\coloneqq \snorm{\rhod}^{-\ppot}$ and show \eqref{ThmExistenceNonLinearFixedDomain_Eq} by establishing existence of a fixed point of the mapping
\begin{align*}
\mmap:\Xreyrhod(\Omega_0)\ra \Xreyrhod(\Omega_0),\quad
\mmap\np{\uvel,\upres,\reyd,\height} \coloneqq   \loprrhodinv\circ\noprs\np{\uvel,\upres,\reyd,\height}
\end{align*}
for sufficiently small $\rhod$.
To ensure that $\mmap$ is well defined, observe that
\begin{align*}
\Div(\cofctrafo\uvel)=\detctrafo\np{\Div(\uvel\circ(\ctrafo)^{-1})}\circ\ctrafo
\end{align*}
and
\begin{align*}
0=\int_{\sphere^2}\bp{\Uvel+e_3}\cdot\nvec\,\dS,\qquad
0= \int_{\sphere^2} e_3\cdot \detctrafo \snorm{\cofctrafo^\transpose\nvec}^{-1}\cofctrafo^\transpose\,\dS,
\end{align*}
which implies
\begin{align*}
\int_{\ball_1} \noprcomp_2(\uvel,\upres,\reyd,\height)\,dx = \int_{\sphere^2} \noprcomp_3(\uvel,\upres,\reyd,\height)\,\dS.
\end{align*}
Moreover, a change of coordinates yields $\noprcomp_4(\uvel,\upres,\reyd,\height)\cdot\nvec=0$, 
and we conclude that $\noprs(\uvel,\upres,\reyd,\height)\in\Yspace(\refdomain)$
after establishing the corresponding estimates below.
By fixing some $\reybar$ and choosing $\epsilon$ so small that $\snorm{\Reyrhod}\leq\reybar$, Theorem \ref{ThmLinearOperatorHom} ensures that $\loprrhod$ is invertible from $\Yspace(\refdomain)$
onto $\Xreyrhod(\Omega_0)$, and
$\mmap$ therefore well defined. In the next step, we show that $\mmap$ is a contractive self-mapping on the ball $\ball_{\rhodpot}(0)\subset\Xreyrhod(\Omega_0)$. To this end,  
consider $\np{\uvel,\upres,\reyd,\height}\in\ball_{\rhodpot}(0)$. The most critical part of the proof is to obtain a suitable estimate of $\noprs\np{\uvel,\upres,\reyd,\height}$. 
We first utilize Lemma \ref{LemmaAuxFieldTruncatedEstimates} and
recall from \eqref{EquationReyAux} that $\Reyrhod$ depends linearly on $\rhod$
to estimate
\begin{align}\label{ThmExistenceNonLinearFixedDomain_Est1}
\norm{(\Reyrhod+\reyd)\Div\fluidstresstrafo(\Uvel_R,\Upres_R)}_{q}\leq\Cc{c}(\snrhod+\snrhod^\ppot)\bp{\snrhod^{\np{3-3/q}\ppot}+\snrhod^\ppot} = \smallo\bp{\snrhod^\ppot}\ \text{as }\snorm{\rhod}\ra 0. 
\end{align}
An application of Lemma \ref{LemmaEstimatesCtrafo} yields
\begin{align}\label{ThmExistenceNonLinearFixedDomain_Est2}
\begin{aligned}
\norm{\Div\fluidstresstrafo(\uvel,\upres)-\Div\fluidstress(\uvel,\upres)}_q&\leq \Cc{c}{\norm{\height}_{\WSR{3-1/r}{r}}\np{\norm{\uvel}_{\XreyOne}+\norm{\upres}_{\Xspace_2}}}\\
&\leq \Cclast{c}\snrhod^{2\ppot}= \smallo\bp{\snrhod^\ppot}\ \text{as }\snorm{\rhod}\ra 0.
\end{aligned}
\end{align}
Lemma \ref{LemmaEstimatesCtrafo} also implies $\norm{\cofctrafo}_\infty\leq\Cc{c}(\const{heightminforinverse})$.
Employing first H\"older's inequality and then estimate \eqref{PropEmbeddingsX1_Embedding2} from Proposition \ref{PropEmbeddingsX1} with $t=2$, we obtain
\begin{align}\label{ThmExistenceNonLinearFixedDomain_Est3}
\begin{aligned}
\norm{\rho\cofctrafo\uvel\cdot\grad\uvel}_q &\leq \Cc{c} \norm{\cofctrafo}_\infty  \norm{\uvel}_{\frac{2q}{2-q}}\norm{\grad\uvel}_2\\
&\leq \Cc{c} \snrhod^{-\half-(1+\frac{3}{2}-\frac{3}{q})}\norm{\uvel}^2_{\XreyOne}
\leq \Cc{c} \snrhod^{\frac{3}{q}-3+2\ppot} = \smallo\bp{\snrhod^\ppot}\ \text{as }\snorm{\rhod}\ra 0
\end{aligned}
\end{align}
since $\frac{3}{4}<\alpha$.
Further applications of H\"older's inequality in combination with the integrability properties \eqref{UvelUpresIntegrability} of $\Uvel$
yield
\begin{align}\label{ThmExistenceNonLinearFixedDomain_Est4}
\begin{aligned}
\norm{\rho\np{\Rey+\reyd}&\bp{
\cofctrafo\Uvel_R\cdot\grad\uvel+\cofctrafo\uvel\cdot\grad\Uvel_R
}}_{q}\\
&\leq \bp{\snorm{\Rey}+\snorm{\reyd}} \norm{\cofctrafo}_\infty \bp{
\norm{\Uvel_R}_4 \norm{\grad\uvel}_{\frac{4q}{4-q}}
+\norm{\uvel}_{\frac{2q}{2-q}} \norm{\grad\Uvel_R}_2
}\\
&\leq \Cc{c}\bp{\snorm{\rhod}+\snrhod^\ppot}
\bp{
\snorm{\rhod}^{-\frac{1}{4}}\norm{\uvel}_{\XreyOne}+\snorm{\rhod}^{-\frac{1}{2}}\norm{\uvel}_{\XreyOne}}\\
&\leq \Cc{c}\bp{\snorm{\rhod}+\snrhod^\ppot}
\bp{
\snorm{\rhod}^{-\frac{1}{4}}+\snorm{\rhod}^{-\frac{1}{2}}
}\snorm{\rhod}^\ppot
= \smallo\bp{\snrhod^\ppot}\ \text{as }\snorm{\rhod}\ra 0
\end{aligned}
\end{align}
since $\frac{1}{2}<\alpha$.
From the integrability properties \eqref{UvelUpresIntegrability} we also obtain $\Uvel_R\cdot\grad\Uvel_R\in\LR{s}(\R^3)$ for all $s>1$ and thus
\begin{align}\label{ThmExistenceNonLinearFixedDomain_Est5}
\begin{aligned}
\norm{\rho(\Reyrhod+\reyd)^2\cofctrafo\Uvel_R\cdot\grad\Uvel_R}_q\leq \Cc{c} \norm{\cofctrafo}_\infty \bp{\snrhod+\snrhod^\ppot}^2
= \smallo\bp{\snrhod^\ppot}\ \text{as }\snorm{\rhod}\ra 0.
\end{aligned}
\end{align}
We move on to the so-called drift terms. Recalling that $\cofctrafo=\idmatrix$ on $\ball_4^c$, we estimate
\begin{align}\label{ThmExistenceNonLinearFixedDomain_Est6}
\norm{\rho\reyd\grad\uvel\cofctrafo e_3}_q
\leq \Cc{c} \snrhod^\ppot \np{\norm{\grad\uvel}_{\LR{q}(\ball_4)} + \norm{\partial_3\uvel}_q}\leq
\Cclast{c} \snrhod^\ppot \snrhod^\ppot  
= \smallo\bp{\snrhod^\ppot}\ \text{as }\snorm{\rhod}\ra 0
\end{align}
and similarly
\begin{align}\label{ThmExistenceNonLinearFixedDomain_Est7}
\norm{\rho\Reyrhod\grad\uvel(\cofctrafo-\idmatrix)e_3}_q\leq 
\Cc{c} \snrhod^{2\ppot}   
= \smallo\bp{\snrhod^\ppot}\ \text{as }\snorm{\rhod}\ra 0.
\end{align}
Finally, we once more employ Lemma \ref{LemmaAuxFieldTruncatedEstimates} to deduce
\begin{align}\label{ThmExistenceNonLinearFixedDomain_Est8}
\begin{aligned}
\norm{ \rho(\Reyrhod+\reyd)^2\grad\Uvel_R\cofctrafo e_3}_q
&\leq \Cc{c}\bp{\snrhod + \snrhod^\ppot}^2 R^{-2+3/q}\\
&= \Cclast{c}\bp{\snrhod + \snrhod^\ppot}^2 \snrhod^{(2-3/q)\ppot}
= \smallo\bp{\snrhod^\ppot}\ \text{as }\snorm{\rhod}\ra 0
\end{aligned}
\end{align}
Summarizing \eqref{ThmExistenceNonLinearFixedDomain_Est1}--\eqref{ThmExistenceNonLinearFixedDomain_Est8}, we
conclude
$\norm{\noprcomp_1\np{\uvel,\upres,\reyd,\height}}_q= \smallo\bp{\snrhod^\ppot}\ \text{as }\snorm{\rhod}\ra 0$,
which is the most critical estimate of the proof. With less effort, the same estimate can be established for $\norm{\noprcomp_1\np{\uvel,\upres,\reyd,\height}}_r$. 
Hence, $\norm{\noprcomp_1\np{\uvel,\upres,\reyd,\height}}_{\Yspace_1}= \smallo\bp{\snrhod^\ppot}$ as $\snorm{\rhod}\ra 0$.
The other components $\noprcomp_2,\ldots,\noprcomp_7$ of $\noprs\np{\uvel,\upres,\reyd,\height}$ are estimated similarly. In particular, employing
that $\WSR{1}{r}(\sphere^2)$ is an algebra due to $r>3$, the nonlinear term $\norm{\calg_\meancurv(\height)}_{\WSR{1-1/r}{r}(\sphere^2)}$ can be estimated
such that we obtain
\begin{align*}
\norm{\noprcomp_7(\uvel,\upres,\reyd,\height)}_{\Yspace_7}\leq \Cc{c}\bp{\snrhod+ \snrhod^2 + \snrhod^3 + \snrhod^4}.
\end{align*}
Since $\alpha<1$, we deduce 
$\norm{\noprcomp_7\np{\uvel,\upres,\reyd,\height}}_{\Yspace_7}= \smallo\bp{\snrhod^\ppot}$
and thus $\norm{\noprcomp\np{\uvel,\upres,\reyd,\height}}_{\Yspace}= \smallo\bp{\snrhod^\ppot}$ as $\snorm{\rhod}\ra 0$.
Recalling from Theorem \ref{ThmLinearOperatorHom} that $\norm{\loprrhodinv}$ is independent of $\Reyrhod$, we conclude that also
$\norm{\mmap}_{\Xreyrhod}= \smallo\bp{\snrhod^\ppot}$ as $\snorm{\rhod}\ra 0$. Consequently, $\mmap$ is a
self-mapping on the ball $\ball_{\rhodpot}(0)\subset\Xreyrhod(\Omega_0)$ for sufficiently small $\rhod$. Estimates completely similar to the ones
above can be used to verify that $\mmap$ is also a contraction on $\ball_{\rhodpot}(0)\subset\Xreyrhod(\Omega_0)$ for sufficiently small $\rhod$. Therefore, the contraction
mapping principle (or Banach's Fixed Point Theorem) yields a unique fixed point $\np{\uvel,\upres,\reyd,\height}$ in $\ball_{\rhodpot}(0)$ of $\mmap$, which is clearly a solution to
\eqref{ThmExistenceNonLinearFixedDomain_Eq} satisfying \eqref{ThmExistenceNonLinearFixedDomain_Est}.
\end{proof}

Finally, we are able to prove the main theorem of the article.

\begin{proof}[Proof of Theorem \ref{MainThm}]
Choosing the parameters as in Theorem \ref{ThmExistenceNonLinearFixedDomain}, we let $(\uvel,\upres,\reyd,\height)\in\Xreyrhod(\Omega_0)$
denote the corresponding solution to \eqref{ThmExistenceNonLinearFixedDomain_Eq}.

A boot-strapping argument based on
coercive $\LR{r}$ estimates in the whole and half space for the principle part of the operators $\loprcomp{1-4}$ and $\loprcomp{7}$, furnished by Theorem \ref{ADN_TwoFoldHalfSpace} in the former case 
and well-know estimates for the classical Laplace operator in the latter case, yields higher-order regularity. More specifically, after smoothing out the boundary in the $\loprcomp{1-4}$ part of equation \eqref{ThmExistenceNonLinearFixedDomain_Eq}, difference quotients of $\np{\uvel,\upres}$ can be
estimated using Theorem \ref{ADN_TwoFoldHalfSpace}, which implies additional regularity of $\np{\uvel,\upres}$.
In turn, classical $\LR{r}$ estimates for the Laplace operator in the 2D whole space yields bounds on difference quotients for $\height$ after smoothing out the interface in the  
$\loprcomp{7}$ part of equation \eqref{ThmExistenceNonLinearFixedDomain_Eq}. 
In both cases, we choose $\epsilon$ and thus $\rhod$ sufficiently small in order to absorb higher-order terms from the right-hand side.
Bootstrapping this procedure, we conclude regularity of arbitrary order for both $\np{\uvel,\upres}$ and $\height$, and thereby deduce that the solution is smooth up to the boundary.

We further claim that the solution is invariant with respect to rotations that leave the $e_3$-axis invariant.
To this end, consider an arbitrary  $\rotmatrix\in\sorthomatrixspace{3}$
with $Re_3=e_3$.
Define
\begin{align*}
\tuvel(x)\coloneqq R^\transpose\uvel(Rx), \quad
\tupres(x)\coloneqq \upres(Rx), \quad
\treyd \coloneqq \reyd, \quad
\theight(x)\coloneqq\height(Rx).
\end{align*}
Utilizing that \eqref{AuxFieldSymmetry} leads to rotation invariance of $(\Uvel_R,\Upres_R)$, 
and that \eqref{ConstrucntionOfCoordinateTrafo_RotSym} implies 
$\tctrafo(x)=R^\transpose\ctrafo(Rx)$,
one readily verifies that $\np{\tuvel,\tupres,\treyd,\theight}\in\Xreyrhod(\Omega_0)$
is another solution to \eqref{ThmExistenceNonLinearFixedDomain_Eq} satisfying \eqref{ThmExistenceNonLinearFixedDomain_Est}.
The uniqueness assertion of Theorem \ref{ThmExistenceNonLinearFixedDomain} therefore yields 
$\np{\uvel,\upres,\reyd,\height}=\np{\tuvel,\tupres,\treyd,\theight}$, and we conclude the claimed rotational symmetry of the solution.

Now recall from \eqref{EqLinearizationAtAuxField} that a solution to \eqref{SystemAbstractFormulation} yields a solution $\np{\wvel,\wpres,\rey,\height}$ to
\eqref{SystemFluidReferenceConfiguration}. Due to the rotation symmetry of $\np{\wvel,\wpres,\rey,\height}$,
we thereby obtain a solution in $\Xreyrhod(\Omega_0)$
to \eqref{SystemFluidReferenceConfigurationPreSimplification}. Finally recalling \eqref{EqFieldsTransformedOnFixedDomain},
we deduce existence of a solution $\np{\tdvel,\sspres,\rey,\height}$ to \eqref{derivation_sseq_final} satisfying \eqref{MainThm_IntegrabilityProperties_SmoothUpToBoundary} and
\eqref{MainThm_Symmetry}.

Since $\tdvel\in\Xoseen\bp{\ssreservoirdomainh}$, we may ``test'' the system with $\tdvel$, \ie, multiplication of \eqrefsub{derivation_sseq_final}{1} by $\tdvel$ and
subsequent integration by parts is a valid computation. Under the assumption $\rey=0$ this computation yields $\rhod=0$. Since we are assuming $\rhod\neq 0$, we conclude
that also $\rey\neq 0$.

Finally, since $\np{\tdvel,\sspres}$ solves the classical Navier--Stokes equations in a 3D exterior domain with $\tdvel\in\Xoseen\bp{\ssreservoirdomainh}$ and $\rey\neq 0$, the
integrability properties \eqref{MainThm_IntegrabilityProperties} and asymptotic structure  \eqref{MainThm_AsymptoticProfile}
follow from \cite[Theorem X.6.4]{GaldiBookNew} and \cite[Theorem X.8.1]{GaldiBookNew}, respectively.
\end{proof}

\bibliographystyle{abbrv}

\end{document}